\documentclass[12pt]{amsart}

\usepackage{fullpage}
\usepackage{graphicx}
\usepackage{algorithmicx}
\usepackage{algpseudocode}
\usepackage{amssymb}

\newtheorem{theorem}{Theorem}[section]
\newtheorem{corollary}[theorem]{Corollary} 
\newtheorem{lemma}[theorem]{Lemma}
\newtheorem{proposition}[theorem]{Proposition}
\newtheorem{problem}[theorem]{Problem}

\theoremstyle{definition}
\newtheorem{definition}[theorem]{Definition}

\newtheorem{example}[theorem]{Example}

\newcommand{\sparse}{\operatorname{sparse}}
\newcommand{\R}{{\mathbb R}}
\newcommand{\C}{{\mathbb C}}

\newcommand{\F}{{\mathbb F}}
\newcommand{\Q}{{\mathbb Q}}
\newcommand{\Z}{{\mathbb Z}}
\newcommand{\OO}{\operatorname{O}}
\newcommand{\gsparse}{\operatorname{gsparse}}
\newcommand{\trace}{\operatorname{trace}}

\newcommand{\sgn}{\operatorname{sgn}}
\newcommand{\GL}{\operatorname{GL}}
\newcommand{\proj}{\operatorname{proj}}
\newcommand{\shrink}{\operatorname{shrink}}

\newcommand{\rank}{\operatorname{rank}}
\newcommand{\perm}{\operatorname{perm}}
\title{A  general theory of singular values\\with applications to Signal Denoising}
\author{Harm Derksen}
\thanks{The author was partially supported by NSF grant DMS 1601229.}
\begin{document}
\maketitle
\begin{abstract}
We study the Pareto frontier for two competing norms $\|\cdot\|_X$ and $\|\cdot\|_Y$ on a vector space. For a given vector $c$,
the Pareto frontier describes the possible values of $(\|a\|_X,\|b\|_Y)$ for a decomposition $c=a+b$.
The singular  value decomposition of a matrix is closely related to the Pareto frontier for the spectral and nuclear norm. We will develop a general theory that extends the notion of  singular values  of a matrix to arbitrary finite dimensional euclidean vector spaces
equipped with dual norms. This also generalizes the diagonal singular value decompositions for tensors introduced by the author in previous work. We can apply the results to denoising, where $c$ is a noisy signal, $a$ is a sparse signal and $b$ is noise. Applications  include 1D total variation denoising,
2D  total variation Rudin-Osher-Fatemi image denoising, LASSO, basis pursuit denoising and tensor decompositions.

\end{abstract}
\section{Introduction}
Sound, images, and videos can be corrupted by noise. Noise removal is a fundamental problem in signal and image processing   In the additive noise model, we have an original signal $a$, additive noise $b$ and a corrupted signal $c=a+b$.  We will work with discrete signals and view $a$, $b$ and $c$  as vectors or arrays.  The problem of noise removal can be framed in terms of competing norms on a vector space. The Pareto frontier defines the optimal trade-off between the two norms. The Pareto frontier was used in the L-curve method
in Tikhonov regularization (see \cite{Phillips,Tik,Hansen,HO} and \cite[Chapter 26]{LH}), and in basis pursuit denoising (see~\cite{vdBF,TMJS,GLW}) to find optimal regularization parameters.
The Pareto frontier is a continuous convex curve, and has a continuous derivative if one of the norms is the euclidean norm (see~\cite{vdBF}, Lemma~\ref{lem:convexdecreasing} and Proposition~\ref{prop:derivativefsquared}). 


We will assume that the original signal $a$ is {\em sparse}. A  vector is sparse when it has  few nonzero values. We will also consider other notions of sparseness. For example, a piecewise constant function on an interval can be considered sparse because its derivative has few nonzero values (or values where it is not defined), and a piecewise linear function can be considered sparse because the second derivative has  few nonzero values. A sound signal from music is sparse because it contains only  a few frequencies.
A typical image is sparse, because it has large connected areas of the same color, i.e., the image is piecewise constant. This is exploited in the total variation image denoising method of Rudin, Osher and Fatemi (\cite{ROF}). A matrix of low rank can also be viewed as sparse because it is the sum of a few rank 1 matrices. In this context, principal component analysis can be viewed as a method for  recovering a sparse signal. 
The noise signal $b$, on the other hand, is not sparse. For example, white noise  contains all frequences in the sound spectrum. gaussian additive noise in an image will be completely discontinuous and not locally constant at all.

There are many ways to measure sparseness. Examples of sparseness measures are the $\ell_0$ ``norm'' (which actually is not a norm), the rank of a matrix or
the number of different frequencies in a sound signal. 
It is difficult to use these measure because they are not convex. We deal with this using {\em convex relaxation}, i.e., we replace the non-convex sparseness measure by a convex one. In this paper, we will measure sparseness using a norm $\|\cdot\|_X$ on the vector space of all signals. These norms coming from convex relaxation are typically $\ell_1$-type norms. For example, we may replace the $\ell_0$ ``norm'' by the $\ell_1$ norm, or replace the rank of a matrix by the nuclear norm. 
Noise will be measured using a different norm $\|\cdot\|_Y$.
This typically will be a euclidean $\ell_2$ norm or perhaps an $\ell_\infty$ type norm. The quotient $\|a\|_Y/\|a\|_X$ is large for the sparse signal $a$,
and $\|b\|_Y/\|b\|_X$ is small for the noise signal $b$. To denoise a signal $c$ we search for a decomposition $c=a+b$ where
$\|a\|_X$ and $\|b\|_Y$ are small. Minimizing $\|a\|_X$ and $\|b\|_Y$ are two competing objectives. The trade-off between these two objectives  is governed by the Pareto frontier. The concept of Pareto efficiency was used by Vilfredo Pareto (1848--1923) to decribe economic efficiency. 
A point $(x,y)\in \R^2$ is called Pareto efficient if there exists a decomposition $c=a+b$ with $\|a\|_X=x$ and $\|b\|_Y=y$
such that for every decomposition $c=a'+b'$ we have $\|a'\|_X>x$, $\|b'\|_Y>y$ or $(\|a'\|_X,\|b'\|_Y)=(x,y)$. If $(x,y)$ is Pareto efficient then we will call the decomposition $c=a+b$ an $XY$-decomposition. Many methods such as  LASSO, basis pursuit denoising, the Dantzig selector, total variation denoising and principal component analysis can be formulated as finding an $XY$-decomposition for certain norms $\|\cdot\|_X$ and $\|\cdot\|_Y$.

In most examples that we consider, the space of signals has a positive definite inner product and the norms $\|\cdot\|_X$ and $\|\cdot\|_Y$ are dual to each other.  The inner product gives an euclidean norm defined by $\|c\|_2=\sqrt{\langle c,c\rangle}$. We now have $3$ different norms. We will show that $X2$-decompositions and $2Y$-decompositions are the same. We define the {\em Pareto sub-frontier} as the set of all points $(\|a\|_X,\|b\|_Y)$
where $c=a+b$ is an $X2$-decomposition (or equivalently, a $2Y$-decomposition). The Pareto sub-frontier lies on or above the 
Pareto frontier (by the definition of the Pareto frontier). A vector $c$ is called {\em tight} if its Pareto frontier and Pareto sub-frontier coincide. 
If every vector is tight then the norms $\|\cdot\|_X$ and $\|\cdot\|_Y$ are called {\em tight}. We show that for tight vectors, the Pareto (sub-)frontier is piecewise linear.
For a tight vector $c$, we will define the slope decomposition of $c$ which can be thought of as a generalization of the singular value decomposition.

 The nuclear norm $\|\cdot\|_\star$ and spectral norm $\|\cdot\|_\sigma$ of a matrix are dual to each others (see for example \cite{CR,RFP} and Lemma~\ref{lem:sigmadual}) and  we will show that these norms are tight in Section~\ref{sec:matrixSVD}.   If the singular values of a matrix $A$ are $\lambda_1,\lambda_2,\dots,\lambda_r$ with multiplicities $m_1,m_2,\dots,m_r$ respectively, we have the following well-known formulas for the 
  spectral, nuclear and euclidean (Frobenius) norm:
 \begin{eqnarray}\label{eq:norms1}
 \|A\|_\sigma &=& \lambda_1\\
 \|A\|_* & = & m_1\lambda_1+m_2\lambda_2+\cdots+m_r\lambda_r\label{eq:norms2}\\
 \|A\|_2 & = & \sqrt{m_1\lambda_1^2+m_2\lambda_2^2+\cdots+m_r\lambda_r^2} \label{eq:norms3}
 \end{eqnarray}  
The singular value region of $A$ is a bar of height $\lambda_1$ and width $m_1$, followed by a bar of height $\lambda_2$ and width $m_2$, etc.
The singular values and their multiplicity can now easily be read off from the singular value region.
For exampe if a matrix  $A$ has eigenvalues $2.5$ with multiplicity 2, 1 with multiplicity 1 and 0.5 with multiplicity 2, then the singular value region is plotted below:

\centerline{\includegraphics[width=4in]{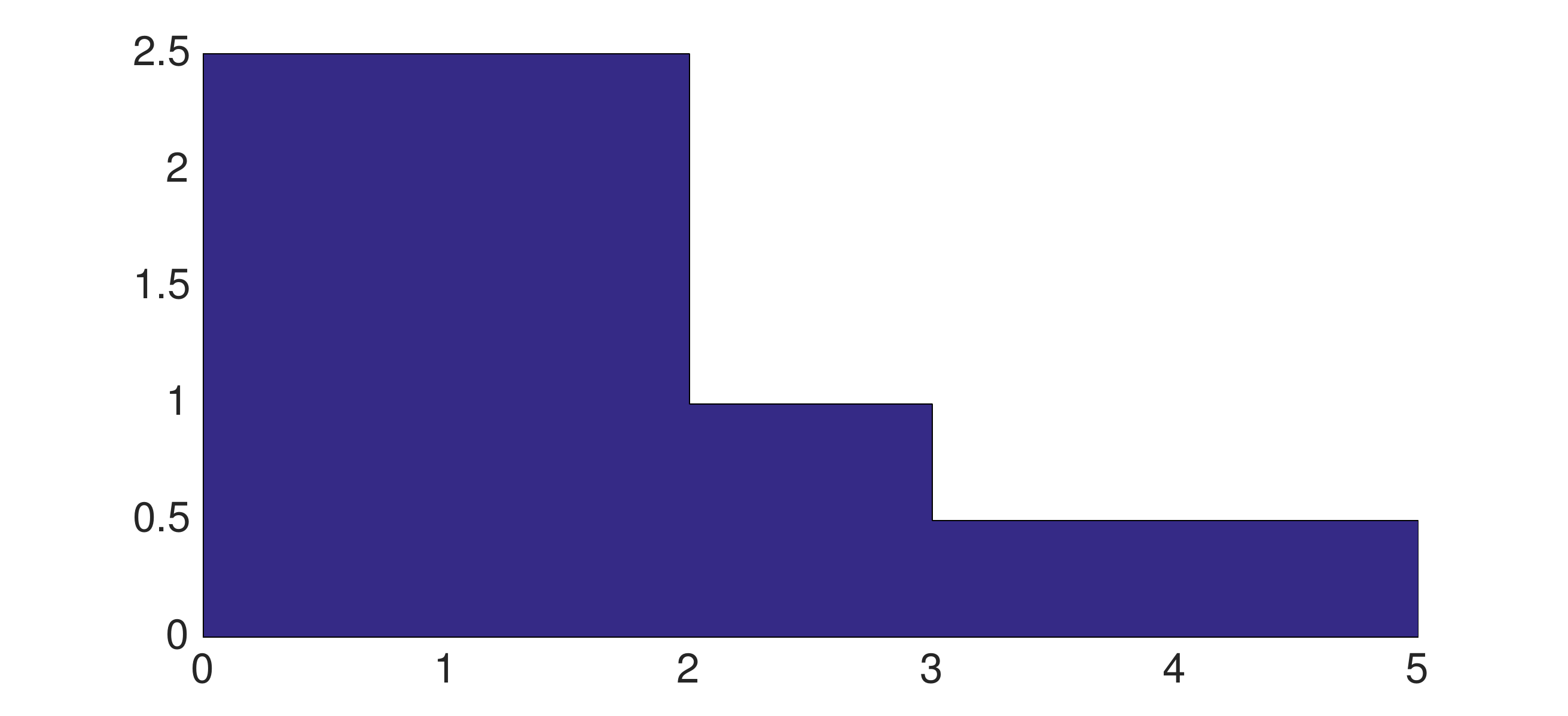}}

The height of the singular value region is the spectral norm, the width is the rank, the area is the nuclear norm, and if we integrate $2y$ over the region we obtain the square of the Frobenius norm $\|A\|_2^2$.
 The Pareto frontier of a matrix (which is also its Pareto sub-frontier) encodes the singular values of the matrix, and
 the slope decomposition is closely related to the singular value decomposition. 
  The slope decomposition is unique, but the singular value decomposition may not be unique if some of the singular values coindice. 
    
A tensor is a higher dimensional array. A $d$-way tensor is a vector for $d=1$ and a matrix for $d=2$. For $d\geq 3$, one can generalize the rank of a matrix to tensor rank (\cite{Hitchcock}). The tensor rank is closely related to the canonical polyadic decomposition of a tensor (CP-decomposition). This decomposition is also known as the PARAFAC (\cite{Harshman}) or the CANDECOMP model (\cite{CC}).
The nuclear norm of a matrix can be generalized to the nuclear norm of a tensor (\cite{Grothendieck,Schatten}), and this can be viewed as a convex relaxation of the tensor rank. The spectral norm of a matrix can be generalized to a spectral norm of a tensor, and this norm is dual to the nuclear norm of a tensor (\cite{Derksen}). Not every tensor is tight. A tensor that is tight will have a slope decomposition which generalizes the diagonal singular value decomposition introduced by the author in \cite{Derksen}. 
Every tensor that has a diagonal singular value has a slope decomposition but the converse is not true.  We will define singular values and multiplicities for tight tensors such that the formulas (\ref{eq:norms1}), (\ref{eq:norms2}), (\ref{eq:norms3}) are satisfied.  The multiplicities of the singular values of tensors are nonnegative, but are not always integers. For example, in Section~\ref{sec:tensor} we will show that the tensor
$$
e_1\otimes e_2\otimes e_3+e_1\otimes e_3\otimes e_2+e_2\otimes e_1\otimes e_3+e_2\otimes e_3\otimes e_1+e_3\otimes e_1\otimes e_2+e_3\otimes e_2\otimes e_1\in \R^{3\times 3 \times 3}
$$
is tight and has the singular value $\frac{2}{\sqrt{3}}$ with multiplicitiy $\frac{9}{2}$ (and not singular value $1$ with multiplicity 6 as one might expect). 
For tensors that are not tight we can still define the singular value region, but the singular value interpretation may me more esoteric. For example, we will show that the tensor
$$
e_2\otimes e_1\otimes e_1+e_1\otimes e_2\otimes e_1+e_1\otimes e_1\otimes e_2\in \R^{2\times 2\times 2}
$$
has the following singular value region:

\centerline{\includegraphics[width=4in]{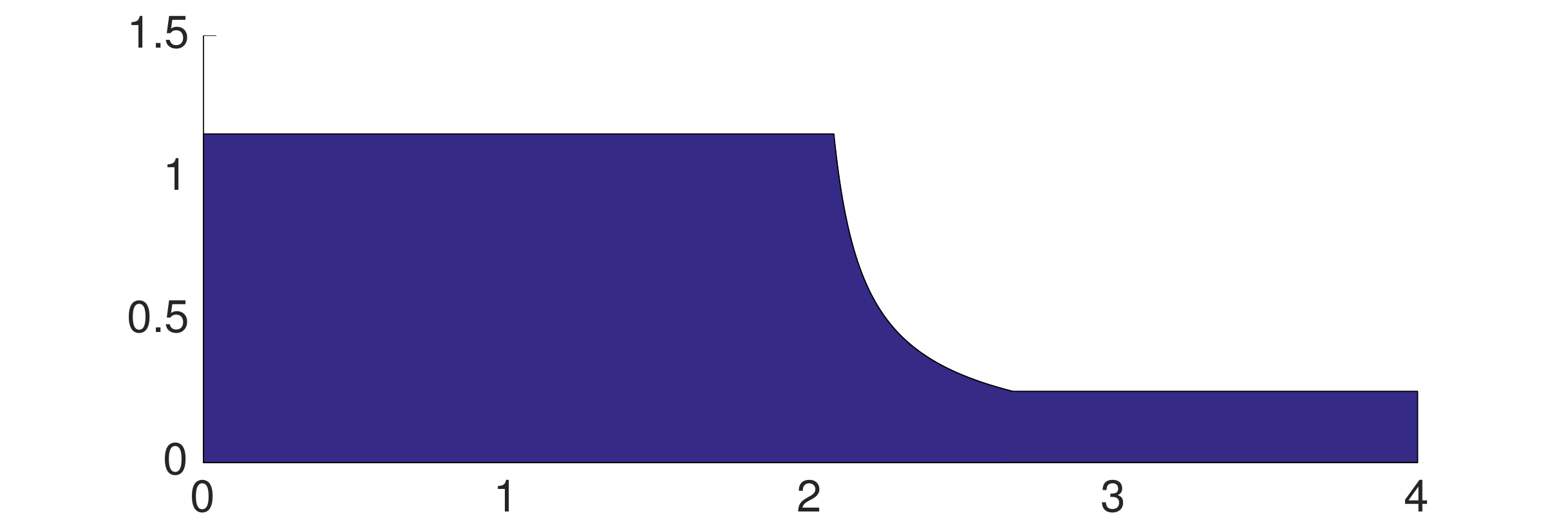}}

We will define the singular value region in a very general context. Whenever $V$ is a finite dimensional euclidean vector space, and $\|\cdot\|_X$ and $\|\cdot\|_Y$ are norms that are dual to each other, then we can define the singular value region for any $c\in V$. 

\tableofcontents

\section{Main Results}
\subsection{The Pareto frontier}
Let us consider a finite dimensional $\R$-vector space $V$ equipped with two norms, $\|\cdot\|_X$ and $\|\cdot\|_Y$.
Suppose that $c\in V$. We are looking for decompositions $c=a+b$ that are optimal in the sense that we cannot reduce $\|a\|_X$ without
increasing $\|b\|_Y$ and we cannot reduct $\|b\|_Y$ without increasing $\|a\|_X$.  We recall the definition from the introduction:
\begin{definition}
 A pair $(x,y)\in \R^2_{\geq 0}$ is called {\em Pareto efficient} if there exists a decomposition $c=a+b$ with $\|a\|_X=x$, $\|b\|_Y=y$ such that for every decomposition $c=a'+b'$  we have $\|a'\|_X> x$, $\|b'\|_Y> y$ or $(\|a'\|_X,\|b'\|_Y)=(x,y)$. If $(x,y)$ is a Pareto efficient pair then we call $c=a+b$ an {\em  
$XY$-decomposition}. 
\end{definition}
By symmetry,  $c=a+b$ is an $XY$-decomposition if and only if $c=b+a$ is a $YX$-decomposition.
The {\em Pareto frontier} consists of all Pareto efficient pairs (see~\cite{vdBF}).  The Pareto frontier is the graph of a strictly decreasing, continuous convex function 
$$f_{YX}^c:[0,\|c\|_X]\to [0,\|c\|_Y]$$ 
(see \cite{vdBF} and~Lemmas~\ref{lem:convexdecreasing} and~\ref{lem:ParetoEfficient}). If we change the role of $X$ and $Y$ we get the graph of $f_{XY}^c$, so
$f_{XY}^c$ and $f_{YX}^c$ are inverse functions of each other.


\begin{example}\label{ex:l1linf}
Consider the vector space  $V=\R^n$. Sparseness of vectors in $\R^n$ can be measured by the number of nonzero 
entries. For $c\in V$ we define
$$\|c\|_0=|\{i\mid 1\leq i\leq n,\ c_i\neq 0\}|=\lim_{p\to 0} \|c\|_p^p.$$
Note that $\|\cdot\|_0$ is not a norm on $V$ because it does not satisfy $\|\lambda c\|_0=|\lambda|\|c\|_0$ for $\lambda\in \R$. Convex relaxation  of $\|\cdot\|_0$ gives us the $\ell_1$ norm $\|\cdot\|_1$.  This means that the unit ball for the norm $\|\cdot\|_1$ is the convex hull of all vectors $c$ with $\|c\|_0=\|c\|_2=1$. Let us take  $\|\cdot\|_X=\|\cdot\|_1$ and $\|\cdot\|_Y=\|\cdot\|_\infty$ and describe the Pareto frontier.
Suppose that  $c=(c_1 \cdots c_n)^t\in \R^n$ and $0\leq y\leq \|c\|_\infty$. If $\|b\|_\infty=y$ then we have
$$
\textstyle \|c-b\|_1=\sum_{i=1}^n |c_i-b_i|\geq \sum_{i=1}^n\min\{|c_i|-y,0\}.
$$
If we take $b_i=\sgn(c_i)\min\{|c_i|,y \}$ for all $i$, then we have equality. So  $c=a+b$ is an $XY$-decomposition where $a:=c-b$.
This shows that 
$$
\textstyle f_{XY}^c(y )=\|a\|_1=\sum_{i=1}^n\max\{0,|c_i|-y\}.
$$


For the vector $c=(-1,2,4,1,-2,1,-1)^t$
we plotted $f_{XY}^c=f_{1\infty}^c$ and $f_{YX}^c=f_{\infty 1}^c$ (see also Example~\ref{ex:1.10})

\centerline{\includegraphics[width=3in]{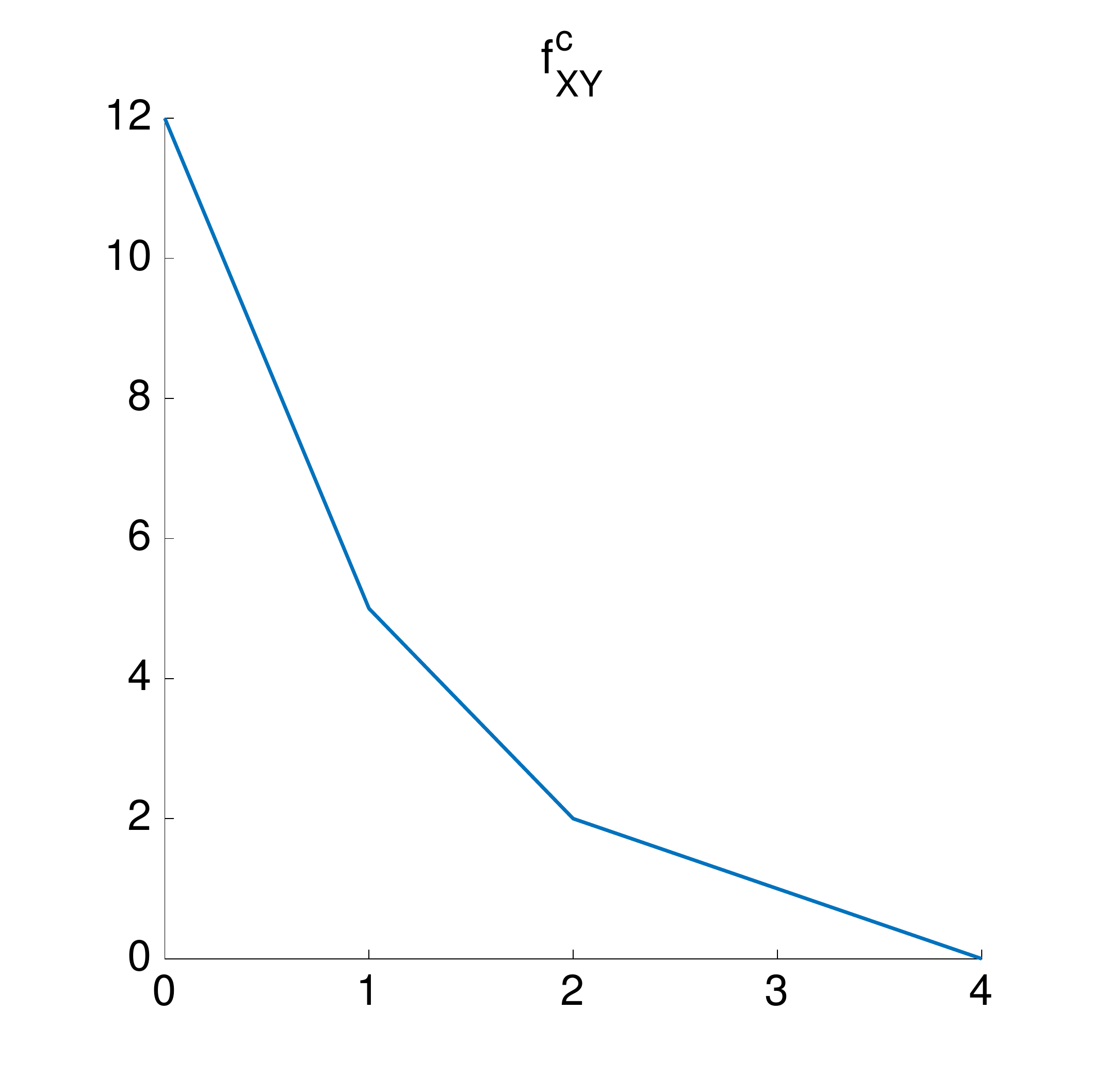}\includegraphics[width=3in]{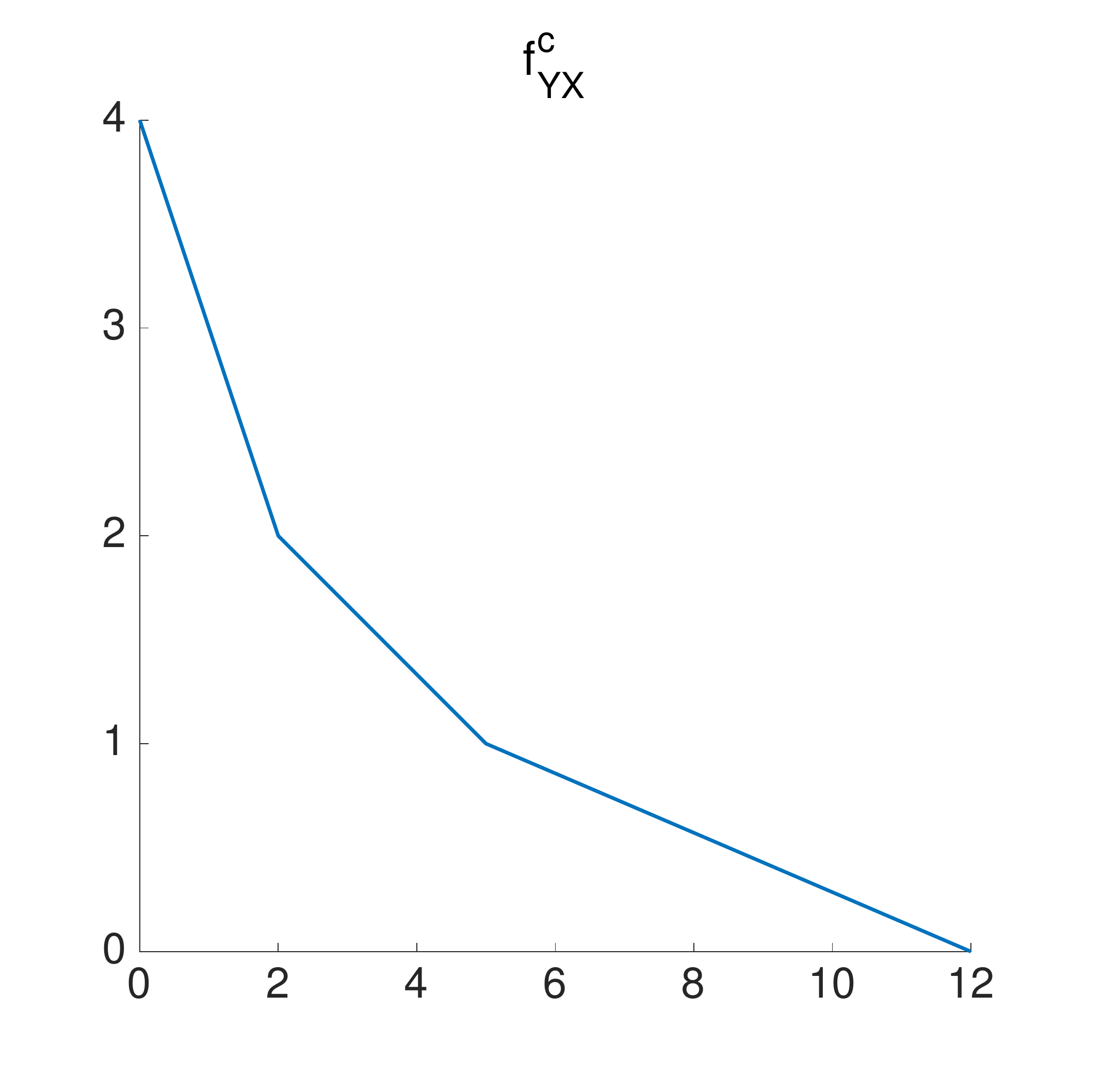}}

For example if we take $y=\frac{3}{2}$ then we get the decomposition with $a=(0,\frac{1}{2},\frac{5}{2},0,-\frac{3}{2},0,0)^t$ and $b=(-1,\frac{3}{2},\frac{3}{2},1,-\frac{3}{2},1,-1)^t$ and we have $\|a\|_1=\frac{7}{2}$ and $\|a\|_0=3$. The vector $a$ is sparser than  $c$. This procedure of noise reduction is {\em soft tresholding} in its simplest form.
\end{example}
In Section~\ref{sec:Pareto}, we will study the Pareto frontier and $XY$-decompositions in more detail.
\subsection{Dual norms and the Pareto sub-frontier}
We now assume that we have a positive definite bilinear form $\langle\cdot,\cdot\rangle$ on the finite dimensional vector space $V$. 
The euclidean $\ell_2$ norm on $V$  defined by $\|v\|_2:=\sqrt{\langle v,v\rangle}$. Suppose that $\|\cdot\|_X$ is another norm on $V$. We may think of $\|\cdot\|_X$ as a norm that measures sparseness. For denoising, we compare the norms $\|\cdot\|_X$ and $\|\cdot\|_2$. We also consider the
 {\em dual norm} on $V$  defined by
$$
\|v\|_Y=\max\{\langle v,w\rangle\mid w\in V,\ \|w\|_X=1\}.
$$
The dual norm of $\|\cdot\|_Y$ is $\|\cdot\|_X$ again.
There is an interesting interplay between the 3 norms, and the $X2$-decompositions, $2Y$-decompositions and $XY$-decompositions are closely connected. 
 The following proposition and other results in this section will be proved in Section~\ref{sec:3}.
\begin{proposition}\label{prop:equivalent}
For a vector $c\in V$, the following three statements are equivalent:
\begin{enumerate}
\item $c=a+b$ is an $X2$-decomposition;
\item $c=a+b$ is a  $2Y$-decomposition;
\item $c=a+b$ and $\langle a,b\rangle=\|a\|_X\|b\|_Y$.
\end{enumerate}
\end{proposition}
\begin{definition}
If the statements (1)--(3) in Proposition~\ref{prop:equivalent} imply
\begin{itemize}
\item[(4)] $c=a+b$ is an $XY$-decomposition.
\end{itemize}
 then $c$ is called {\em tight}. 
If all vectors in $V$ are tight then the norm $\|\cdot\|_X$ is called {\em tight}. 
\end{definition}
\begin{definition}
The {\em Pareto sub-frontier} is the set of all pairs $(x,y)\in \R_{\geq 0}^2$
such that there exists an $X2$-decomposition $c=a+b$ with $\|a\|_X=x$ and $\|b\|_Y=y$.
\end{definition}
We will show in Section~\ref{sec:3} that the Pareto sub-frontier is the graph of a decreasing  Lipschitz continuous function 
$$\textstyle h_{YX}^c:[0,\|c\|_X]\to [0,\|c\|_Y].$$ 
By symmetry,  $h_{XY}^c$ is the inverse function of $h_{YX}^c$.
From the definitions it is clear that $f_{YX}^c(x)\leq h_{YX}^c(x)$. If $f_{XY}^c=h_{XY}^c$, then every $X2$-decomposition is
automatically a $XY$-decomposition, and $c$ is tight
\begin{corollary}
A vector $c\in V$ is tight if and only if $f_{XY}^c=h_{XY}^c$.
\end{corollary}
If $f_{XY}^c=h_{XY}^c$ then there is no space between the two graphs and they fit together {\em tightly}, which explains the name of this property.
Let us work out an example where the norms are not tight.
\begin{example}
Define a norm $\|\cdot\|_X$ on $\R^2$ by
$$
\left\|\begin{pmatrix}Ä
z_1\\z_2\end{pmatrix}\right\|_X=\sqrt{\textstyle\frac{1}{2}z_1^2+2z_2^2}.
$$
Its dual norm is given by
$$
\left\| \begin{pmatrix}
z_1\\z_2\end{pmatrix}\right\|_Y=\sqrt{\textstyle 2z_1^2+\frac{1}{2}z_2^2}.
$$
For $t\in [\frac{1}{2},2]$, the decomposition
$$
c=
\begin{pmatrix}
3\\
3
\end{pmatrix}
=a(t)+b(t)
$$
is an $X2$-decomposition, where 
$$
a(t)=\begin{pmatrix}
4-2t\\
2t^{-1}-1\end{pmatrix}\mbox{ and }b(t)=\begin{pmatrix} 2t-1\\4-2t^{-1}\end{pmatrix}.
$$
To verify this, we compute
$$
\|a(t)\|_X=\sqrt{2}(2t^{-1}-1)\sqrt{t^2+1},\quad \|b(t)\|_Y=\sqrt{2}(2-t^{-1})\sqrt{t^2+1}
$$
and
$$
\langle a(t),b(t)\rangle=(4-2t)(2t-1)+(2t^{-1}-1)(4-2t^{-1})=2(2t^{-1}-1)(2-t^{-1})(t^2+1)=\|a(t)\|_X\|b(t)\|_Y.
$$
The Pareto sub-frontier is parameterized by 
$$
(\sqrt{2}(2t^{-1}-1)\sqrt{t^2+1},\sqrt{2}(2-t^{-1})\sqrt{t^2+1})),\quad t\in [\textstyle\frac{1}{2},2].
$$
The $XY$-decompositions of $c$ are $c=\widetilde{a}(t)+\widetilde{b}(t)$ for $t\in [\frac{1}{4},4]$, 
where
$$
\widetilde{a}(t)=\frac{1}{5}\begin{pmatrix}
16-4t\\
4t^{-1}-1\end{pmatrix}\mbox{ and }\widetilde{b}(t)=\frac{1}{5}\begin{pmatrix} 4t-1\\16-4t^{-1}\end{pmatrix}.
$$
The Pareto frontier is parameterized by
$$\textstyle(\frac{\sqrt{2}}{5}(4t^{-1}-1)\sqrt{4t^2+1},\frac{\sqrt{2}}{5}(4-t^{-1})\sqrt{t^2+4})).
$$
We plotted the Pareto frontier in green and the Pareto sub-frontier in blue:\\
\centerline{\includegraphics[width=3in]{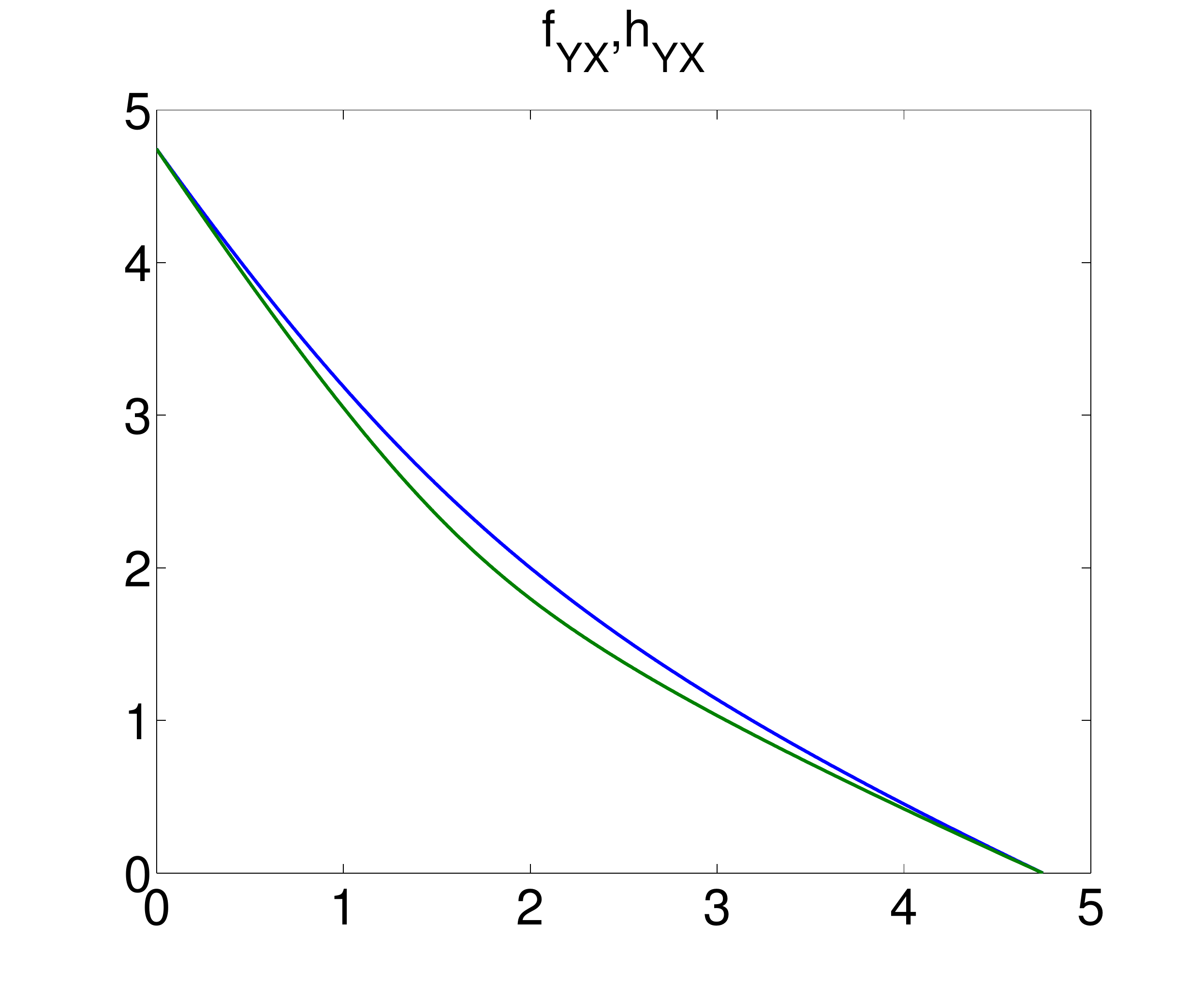}}
The Pareto frontier and sub-frontier are not the same, so the vector $c$ is not tight.
\end{example}

The Pareto sub-frontier  encodes crucial information about the $X2$-decompositions. If $(x_0,y_0)$ is a point on the Pareto subfrontier and
$c=a+b$ is an $X2$-decomposition with $\|a\|_X=x_0$ and $\|b\|_Y=y_0$, then we can read off $\|c\|_X$, $\|c\|_Y$, $\|c\|_2$, $\|a\|_X$, $\|a\|_2$, $\|b\|_Y$, $\|b\|_2$ and $\langle a,b\rangle$ from the Pareto sub-frontier using the following proposition.
\begin{proposition}\label{prop:areas}
Suppose that $(x_0,y_0)$ is a point on the  $XY$-Pareto sub-frontier of $c\in V$, and let $c=a+b$ be an $X2$ decomposition with $\|a\|_X=x_0$ and $\|b\|_Y=y_0$.
\begin{enumerate}
\item The area below the sub-frontier is equal to $\frac{1}{2}\|c\|_2^2$.
\item The area below the sub-frontier and to  the right of $x=x_0$ is equal to $\frac{1}{2}\|b\|_2^2$.
\item The area below the sub-frontier and above $y=y_0$ is equal to $\frac{1}{2}\|a\|_2^2$.
\item The area of the rectangle $[0,x_0]\times [0,y_0]$ is $\langle a,b\rangle$.
\end{enumerate}
\end{proposition}

\centerline{\includegraphics[width=4in]{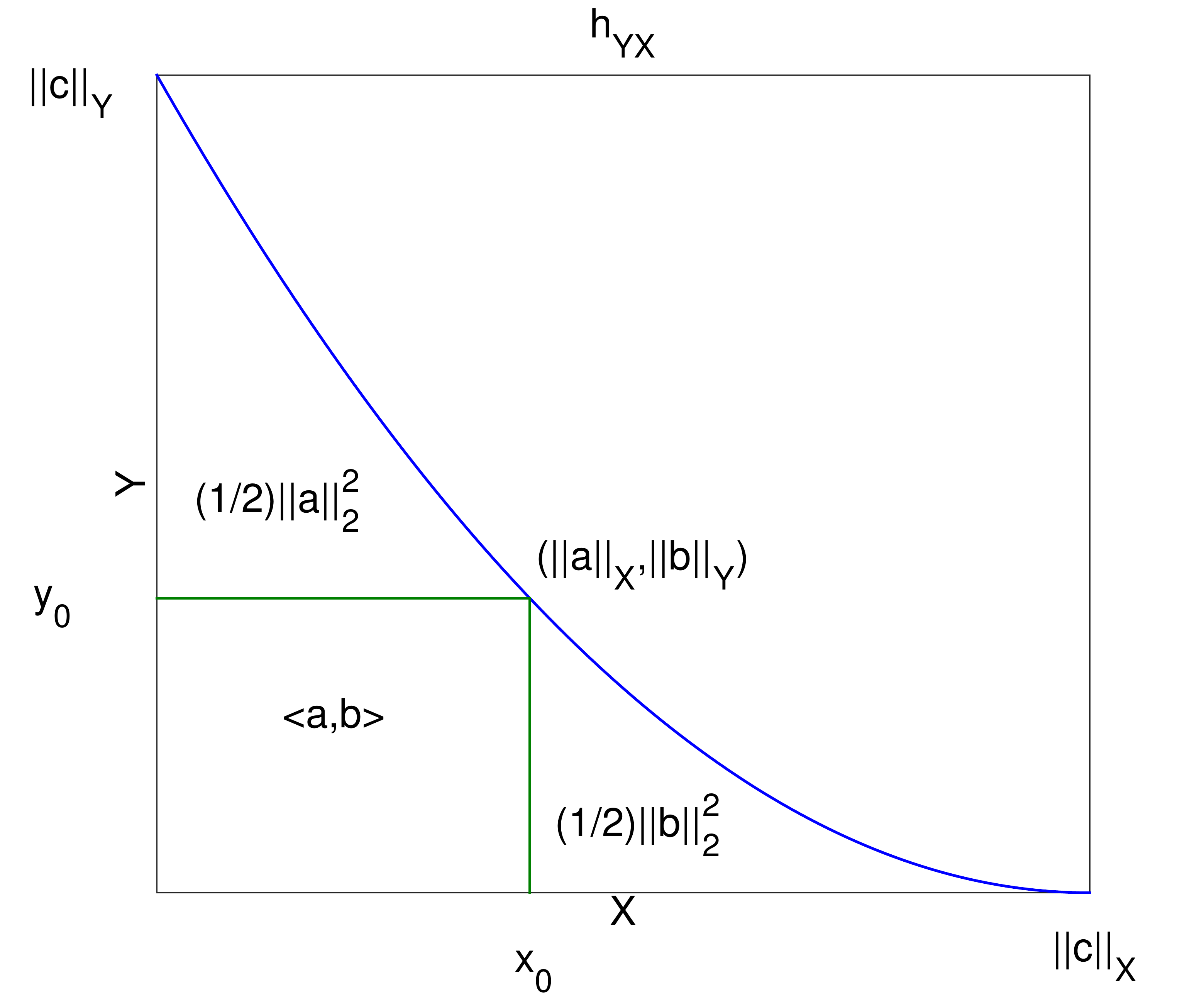}}
\subsection{The slope decomposition}
Proofs of results in this subsection will be given in Section~\ref{sec:slope}.
We define the slope of a nonzero vector $c\in V$ as the ratio 
$$\mu_{XY}(c):=\frac{\|c\|_Y}{\|c\|_X}.$$
If $\|\cdot\|_X$ is a norm that is small for sparse signals, then $\mu_{XY}$ is large for sparse signals, and small for noise.
Note that $\mu_{YX}(c)=(\mu_{XY}(c))^{-1}$.
Using the slope function, we define the slope decomposition. If $\|\cdot\|_X$ is the nuclear norm for matrices, then the slope decomposition is closely related to the singular value decomposition. So we can think of the slope decomposition as a generalization of the singular value decomposition. 
\begin{definition}
An expression $c=c_1+c_2+\cdots+c_r$ is called an  {\em $XY$-slope decomposition} if $c_1,\dots,c_r$ are nonzero, $\langle c_i,c_j\rangle=\|c_i\|_X\|c_j\|_Y$
for all $i\leq j$ and $\mu_{XY}(c_1)>\mu_{XY}(c_2)>\cdots>\mu_{XY}(c_r)$.
\end{definition}
Note that because of symmetry, $c=c_1+c_2+\cdots+c_r$ is an $XY$ slope decomposition if and only if $c=c_r+c_{r-1}+\cdots+c_1$ is a $YX$ slope decomposition. There may be vectors that do not have a slope decomposition.
We will prove the following result.
\begin{theorem}\label{theo:tightslope}\ 
\begin{enumerate}
\item  A vector $c\in V$ is tight if and only if $c$ has a slope decomposition.
\item  Suppose that $c=c_1+c_2+\cdots+c_r$ is a slope decomposition, and let $x_i=\|c_i\|_X$ and $y_i=\|c_i\|_Y$ for all $i$. 
Then $\|c\|_X=\sum_{i=1}^r x_i$, $\|c\|_Y=\sum_{i=1}^r y_i$ and the Pareto frontier  (which is the same as the sub-Pareto frontier) is the piecewise linear curve through the points
$$
(x_1+\cdots+x_i,y_{i+1}+\cdots+y_r), \quad i=0,1,2,\dots,r.
$$
\end{enumerate}
\end{theorem}
\begin{example}\label{ex:1.10}
Let us go back to Example~\ref{ex:l1linf}. The norms $\|\cdot\|_X=\|\cdot\|_1$ and $\|\cdot\|_Y=\|\cdot\|_\infty$ are dual to each other. 
We will show that these norms are tight.
If we integrate the function
$$
f_{XY}^c(y)=\sum_{i=1}^n\max\{0,|c_i|-y\}
$$
from $0$ to $\|c\|_Y=\max\{|c_1|,|c_2|,\dots,|c_n|\}$, we get $\frac{1}{2}\sum_{i=1}^n c_i^2=\frac{1}{2}\|c\|_2^2$ which is the area under the graph of $h_{XY}^c$. So the areas under the graphs of $f_{XY}^c$ and $h_{XY}^c$ are the same, and we deduce that $h_{XY}^c=f_{XY}^c$. This shows that $c$ is tight. 
Since $c$ is arbitrary, the norms $\|\cdot\|_X$ and $\|\cdot\|_Y$ are tight
By Theorem~\ref{theo:tightslope} above, every vector has a slope decomposition. For example
$$
c=\begin{pmatrix}
-1\\
2\\
4\\
1\\
-2\\
1\\
-1\end{pmatrix}=
\begin{pmatrix}
0\\
0\\
2\\
0\\
0\\
0\\
0\end{pmatrix}
+
\begin{pmatrix}
0\\
1\\
1\\
0\\
-1\\
0\\
0
\end{pmatrix}
+
\begin{pmatrix}
-1\\
1\\
1\\
1\\
-1\\
1\\
-1
\end{pmatrix}
=c_1+c_2+c_3
$$
is a slope decomposition. Let $x_i=\|c_i\|_1$ and $y_i=\|c_i\|_Y$. Then we have
$x_1=2$, $x_2=3$, $x_3=7$, $y_1=2$, $y_2=1$, $y_3=1$. The Pareto curve $f_{YX}^c$ is the piecewise linear function going through
$$
(0,2+1+1)=(0,4),(2,1+1)=(2,2),(2+3,1)=(5,1),(2+3+7,0)=(12,0).
$$
\end{example}

\subsection{Geometry of the unit ball}
For $x\geq 0$ we define the $X$-ball of radius $x$ by
$$
B_X(x)=\{v\in V\mid \|v\|_X\leq x\}.
$$
We explain denoising in terms of the geometry of the $X$-balls. Suppose we want to denoise a signal $c$, such that the denoised signal $a$
is sparse. We impose the constraint $\|a\|_X\leq x$. Under this constraint, we minimize the amount of noise, by minimizing the $\ell_2$ norm of $b:=c-a$.
This means that the $a$ is the vector inside the ball $B_X(x)$ that is closest to the vector $c$. We call $a$ the projection of $c$ onto the ball $B_X(x)$
and write $a=\proj_X(c,x)$.
 The function $\proj_X(\cdot,x)$ is a retraction of $\R^n$ onto the ball $B_X(x)$. 
If $x_1\leq x_2$, then it is clear that 
$$
\proj_X(\proj_X(c,x_1),x_2)=\proj_X(c,x_1)
$$
For $x_1>x_2$ one might expect that $\proj_X(\proj_X(c,x_1),x_2)=\proj_X(c,x_2)$. This is not always true, but it is true in the case where $c$ is tight by Proposition~\ref{prop:denoisetransitive} below.

We also define a shrinkage operator by
$$
\shrink_X(c,x)=c-\proj_X(c,x).
$$
If $c=a+b$ is an $X2$-decomposition, $\|a\|_X=x$ and $\|b\|_Y=y$ then we have
$$
a=\proj_X(c,x)=\shrink_Y(c,y)
$$
and
$$
b=\proj_Y(c,y)=\shrink_X(c,x).
$$
The function $\shrink_Y(\cdot,y)$
can be seen as a denoising function where $y$ is the noise level.

A nice property of tight vectors is the transitivity of denoising.
\begin{proposition}\label{prop:denoisetransitive}
 If $c$ is tight, then we have
 \begin{enumerate}
 \item $\proj_X(\proj_X(c,x_1),x_2)=\proj_X(c,\min\{x_1,x_2\})$ and
 \item
$\shrink_X(\shrink_X(c,x_1),x_2)=\shrink_X(c,x_1+x_2)$.
\end{enumerate}
\end{proposition}
The unit ball $B_X=B_X(1)$ is a closed convex set, but not always a polytope. We recall the definition of a face of a closed convex set:
\begin{definition}
A face of a closed convex set $B\subseteq V$ is a convex closed subset $F\subseteq B$ with the following property: if $a,b\in B$, $0<t<1$ and $t a+(1-t)b\in F$ then  we have $a,b\in F$.
\end{definition}
We will study faces of the unit ball $B_X=B_X(1)$, and the cones associated to it.
\begin{definition}
A {\em facial $X$-cone} is a cone of the form $\R_{\geq 0} F$ where $F\subset B_X$ is a (proper) face of the unit ball $B_X$. The set $\{0\}$ is considered a facial $X$-cone as well.
\end{definition}
If $C$ is a nonzero  facial $X$-cone, then $C=\R_{\geq 0} F$ for some proper face $F$ of the unit ball $B_X$. In that case, we have $F=C\cap \partial B_X$, where $\partial B_X=\{c\in V\mid \|c\|_X=1\}$ is the unit sphere.
We now will discuss two notions of sparseness related to a norm $\|\cdot\|_X$. 
\begin{definition}
For a nonzero vector $c$ we define its {\em $X$-sparseness} $\sparse_X(c)$ as
the smallest nonnegative integer  $r$ such that we can write
$$
c=c_1+c_2+\cdots+c_r
$$
where $c_i/\|c_i\|_X$ is a extremal point of the unit ball $B_X$ for $i=1,2,\dots,r$. The {\em geometric $X$-sparseness} $\gsparse_X(c)$ is $\dim C$ where $C$ is the smallest facial $X$-cone  containing $c$. 
\end{definition}
Each notion of sparseness has its merits. 
We have
$$
\sparse_X(a+b)\leq\sparse_X(a)+\sparse_X(b)
$$
but a similar inequality does not always hold for geometric sparseness. On the other hand, the set
$$
\{c\mid \gsparse_X(c)\leq k\}
$$
of $k$-geometric $X$-sparse vectors
is closed, but the set 
$$
\{c\mid \sparse_X(c)\leq k\}
$$
of $k$-$X$-sparse vectors is not always closed. 

We have
$$
\sparse_X(c)\leq \gsparse_X(c)
$$
by the Carath\'eodory Theorem (see~\cite[Theorem 2.3]{Barvinok}).
\begin{example}
Consider $\R^n$, and let $\|\cdot\|_X=\|\cdot\|_1$ and $\|\cdot\|_Y=\|\cdot\|_\infty$. We have
$$
\sparse_X(c)=\gsparse_X(c)=\|c\|_0
$$
where
$$
\|c\|_0=\#\{i\mid c_i\neq 0\}=\lim_{p\to 0}\|c\|_p^p.
$$
The function $\|\cdot\|_0$ is sometimes referred to as the $\ell_0$ norm, but is strictly speaking not a norm.
$$
\gsparse_Y(c)=1+\#\big\{i\,\big| \,|c_i|\neq \max\{|c_1|,\dots,|c_n|\}\big\}.
$$
We have
$$
\begin{pmatrix}
2\\
0\\
0
\end{pmatrix}=\begin{pmatrix}
1\\
1\\
1
\end{pmatrix}+
\begin{pmatrix}
1\\
-1\\
-1
\end{pmatrix}
$$
but
$$
\gsparse_Y\begin{pmatrix}
2\\
0\\
0
\end{pmatrix}=3>1+1=\gsparse_Y\begin{pmatrix}
1\\
1\\
1
\end{pmatrix}+\gsparse_Y
\begin{pmatrix}
1\\
-1\\
-1
\end{pmatrix}.
$$
\end{example}

\begin{example}
If $\|\cdot\|_X=\|\cdot\|_\star$ is the nuclear norm on a space $V=\R^{m\times n}$ of $m\times n$ matrices. and $c\in V$ is a matrix,
then $\sparse_X(c)=\gsparse_X(c)=\rank(c)$  is the rank of $c$.
\end{example}
\begin{example}
Let $V=\R^{2\times 2\times 2}=\R^2\otimes \R^2\otimes \R^2$ be the tensor product space. The value $\sparse_X(c)$ is the tensor rank of $c\in V$. It is known that the set of all tensors of rank $\leq 2$ is not closed. If $\{e_1,e_2\}$ is the standard
orthogonal basis of $\R^2$, then the tensor
$$
c=e_2\otimes e_1\otimes e_1+e_1\otimes e_2\otimes e_1+e_1\otimes e_1\otimes e_2
$$
has rank $3$, but is a limit of tensors of rank $2$. In this case, the geometric sparseness and sparseness are not the same. For example, if 
$$
d=(e_1+(0.1) e_2)\otimes (e_1+(0.1)e_2)\otimes (e_1+(0.1)e_2)-e_1\otimes e_1\otimes e_1
$$
then $\sparse_X(d)=2$ is the tensor rank, but $\gsparse_X(d)>2$.
\end{example}

\begin{theorem}\label{thm:sparser}
Suppose that $c=a+b$ is an $X2$-decomposition. Then we have
$$
 \gsparse_X(a)+\gsparse_Y(b)\leq n+1
$$
where $n=\dim V$.
Moreover, if $c$ is tight then we have
$$
\gsparse_X(a)\leq \gsparse_X(c)\mbox{ and } \gsparse_Y(b)\leq \gsparse_Y(c).
$$
\end{theorem}
since $\gsparse_X\leq \sparse_X$, the theorem also implies that
$$
\sparse_X(a)+\sparse_Y(b)\leq n+1.
$$
\subsection{The Singular Value Region}
We can generalize the notion of the singular values region to an arbitrary finite dimensional vector space $V$ with dual norms $\|\cdot\|_X$ and $\|\cdot\|_Y$. Let $y=h_{YX}^c(x)$ be the Pareto sub-frontier of $c\in V$. For the definition of the singular value region, we view $x$ as a function of $y$ which gives $x=h_{XY}^c(y)$. The function $h_{XY}^c(y)$ is Lipschitz and decreasing and is differentiable almost everywhere.
\begin{definition}\label{def:SVregion}
The {\em singular value region} is the region in the first quadrant to the left of the graph of $x=-\frac{d}{d\,y}h_{XY}^c(y)$.
\end{definition}
If the graph of $-\frac{d}{d\,y}h_{XY}^c(y)$ is a step function, then we can interpret the region as singular values with multiplicities similarly as in the case of matrices.
If a vector $c\in V$ has a slope decomposition, then we can easily find the singular values and multiplicities from Theorem~\ref{theo:tightslope}.
Recall that $\mu_{XY}(c)=\|c\|_Y/\|c\|_X$ and $\mu_{YX}(c)=\|c\|_X/\|c\|_Y$.
\begin{corollary}
If
$$c=c_1+c_2+\dots+c_r
$$
is an $XY$ slope decomposition, then the singular values are
$$\|c_1\|_Y,\|c_2\|_Y,\dots,\|c_r\|_Y$$ with multiplicities
$$
\mu_{YX}(c_1),\mu_{YX}(c_2)-\mu_{YX}(c_1),\dots,\mu_{YX}(c_r)-\mu_{YX}(c_{r-1}).
$$
respectively. 
\end{corollary}
As we will see in Section~\ref{sec:matrixSVD}, this notion of the singular value region is indeed a generalization of the singular value region defined for matrices in the introduction.

\section{The Pareto Curve}\label{sec:Pareto}
\subsection{Optimization problems}
We will formulate signal denoising as an optimization problem. Suppose that $V$ is an $n$-dimensional $\R$-vector space equipped with two norms, $\|\cdot\|_X$ and $\|\cdot\|_Y$. Let $c\in V$ be a fixed vector.
Given a noisy signal $c$ we would like to find a decomposition
$c=a+b$ where both $\|a\|_X$ and $\|b\|_Y$ is small. We can do this by minimizing $\|a\|_X$ under the constraint $\|b\|_Y\leq y$ for some fixed $y\geq 0$,
or by minimizing $\|b\|_Y$ under the constraint $\|a\|_X\leq x$ for some $x$ with  $ 0\leq x\leq \|c\|_X$, we formally define the following optimization problem:\\

\noindent{\bf Problem ${\bf M}_{YX}^c(x)$:} {\it
Minimize $\|c-a\|_Y$ for $a\in V$ under the constraint $\|a\|_X\leq x$. }\\

This has a solution because the function $a\mapsto \|c-a\|_Y$ is continuous, and the domain $\{a\in \R^n\mid \|a\|_X\leq x\}$ is compact.
Let $f_{YX}^c(x)$ be the smallest value of $\|c-a\|_Y$. As we will see later in Lemma~\ref{lem:ParetoEfficient},  the graph of $f_{YX}^c(x)$ consists
of all Pareto efficient pairs $(x,y)$, and we call $f_{YX}^c$ the {\em Pareto curve} or {\em Pareto frontier}. 
Since $c$ will be fixed most of the time, we may just write ${\bf M}_{YX}(x)$ and $f_{YX}(x)$ instead of ${\bf M}_{YX}^c(x)$ and $f_{YX}^c(x)$.
We will prove some properties of the Pareto curve.

%
%

\begin{lemma}\label{lem:aXx}
If $0\leq x\leq \|c\|_X$ and $a$ is a solution to ${\bf M}_{YX}^c(x)$, then $\|a\|_X=x$.
\end{lemma}
\begin{proof}
Suppose that $a$ is a solution to ${\bf M}_{YX}(x)$ and $\|a\|_X<x$. Note that $\|a\|_X<x\leq \|c\|_X$, so $c\neq a$ and $\|c-a\|_Y>0$. 
Choose $\varepsilon>0$ such that $\|a\|_X+\varepsilon\|c\|_X\leq x$ and define $a'=(1-\varepsilon)a+\varepsilon c$.
Then we have $\|a'\|_X\leq (1-\varepsilon)\|a\|_X+\varepsilon \|c\|_X\leq x$ and 
$$\|c-a'\|_Y=\|(1-\varepsilon)(c-a)\|_Y=(1-\varepsilon)\|c-a\|_Y<\|c-a\|_Y.
$$
Contradiction.
\end{proof}

\subsection{Properties of the Pareto curve}
The following lemma gives some basic properties of the Pareto curve. In various contexts, these properties are already known. We formulate the properties for the problem of two arbitrary competing norms on a finite dimensional vector space.

\begin{lemma}\label{lem:convexdecreasing}
The function $f_{YX}$ is a convex, strictly decreasing, continuous function on the interval $[0,\|c\|_X]$.
\end{lemma}
\begin{proof}
Suppose that $0\leq x_1<x_2\leq \|c\|_X$. There exist $a_1,a_2\in V$
with $\|a_i\|_X=x_i$ and $\|c-a_i\|_Y=f_{YX}(x_i)$ for $i=1,2$. 

Let $a=(1-t)a_1+ta_2$ for some $t\in [0,1]$.
Then we have $\|a\|_X=\|(1-t)a_1+ta_2\|_X\leq (1-t)x_1+tx_2$, so by definition we have
\begin{multline*}
(1-t)f_{YX}(x_1)+tf_{YX}(x_2)=(1-t)\|c-a_1\|_Y+t\|c-a_2\|_Y\geq\\
\geq \|(1-t)(c-a_1)+t(c-a_2)\|_Y  =\|c-a\|_Y\geq f_{YX}((1-t)x_1+tx_2).
\end{multline*}
This proves that $f$ is convex.

For some $t\in (0,1]$ we can write $x_2=(1-t)x_1+t\|c\|_X$. We have 
$$f_{YX}(x_2)\leq (1-t) f_{YX}(x_1)+tf_{YX}(\|c\|_X)=(1-t)f_{YX}(x_1)<f_{YX}(x_1).$$ 
This shows that $f$ is strictly decreasing.

Since $f_{YX}$ is convex,  it is continuous on $(0,\|c\|_X)$. Because $f$ is decreasing, it is continuous at $x=\|c\|_X$. We will show that $f$ is also continuous at $x=0$.
Suppose that $x_1,x_2,\dots$ is a sequence in $[0,\|c\|_X]$ with $\lim_{n\to\infty}x_n=0$. Choose $a_n\in V$ such that
$\|a_n\|_X=x_n$ and $\|c-a_n\|_Y=f_{YX}(x_n)$. Since $\lim_{n\to\infty}\|a_n\|_X=\lim_{n\to\infty}x_n=0$, we have that $\lim_{n\to\infty}a_n=0$. 
It follows that $\lim_{n\to\infty} f_{YX}(x_n)=\lim_{n\to\infty}\|c-a_n\|_Y=\|c\|_Y=f_{YX}(0)$ because $\|\cdot\|_Y$ is a continuous function.
\end{proof}
We show now that the graph of $f_{YX}$ consists of all Pareto efficient pairs.
\begin{lemma}\label{lem:ParetoEfficient} Suppose that $c\in V$.
\begin{enumerate}
\item $c=a+b$ is an $XY$-decomposition if and only if $a$ is a solution to ${\bf M}_{YX}^c(\|a\|_X)$.
\item The pair $(x,y)$ is Pareto efficient if and only if $0\leq x\leq \|c\|_X$ and $y=f^{c}_{YX}(x)$.
\end{enumerate}
\end{lemma}
\begin{proof}
(1) Suppose that  $c=a+b$. If $c=a+b$ is an $XY$-decomposition, then it is clear from the definitions that $a$ is a solution to ${\bf M}_{YX}^c(\|a\|_X)$.

On the other hand, suppose that $a$ is a solution to ${\bf M}_{YX}^c(\|a\|_X)$ and $c=a'+b'$.  Let $x=\|a\|_X$ and $y=\|b\|_Y$. Assume that $\|a'\|_X\leq \|a\|_X=x$. Then we have $\|b'\|_Y= \|c-a'\|_Y\geq \|c-a\|_Y=\|b\|_Y=y$.
If $\|b'\|_Y=y$, then $a'$ is also a solution to ${\bf M}_{YX}^c(x)$ and $\|a'\|_X=x$ by Lemma~\ref{lem:aXx}. This shows that $c=a+b$ is an $XY$-decomposition.

(2) Suppose that $(x,y)$ is Pareto efficient. Assume that $\|c\|_X<x$. From the decomposition $c=c+0$ follows that $0=\|0\|_Y>y$. Contradiction.
This shows that $0\leq x\leq \|c\|_X$. There exists a decomposition $c=a+b$ with $\|a\|_X=x$ and $\|c-a\|_Y=\|b\|_Y=y$. Because $\|a\|_X\leq x$,
we have $y=\|c-a\|_Y\geq f^c_{YX}(x)$. Suppose that $a'$ is a solution of ${\bf M}^c_{XY}(x)$. Then $\|a'\|_X=x$ by Lemma~\ref{lem:aXx}. 
Because $(x,y)$ is Pareto efficient, we have $f^c_{YX}(x)=\|c-a'\|_Y\geq y$. We conclude that $f_{YX}^c(x)=y$.

Conversely, suppose that $0\leq x\leq \|c\|_X$ and $y=f^{c}_{YX}(x)$. Let $a$ be a solution of ${\bf M}_{YX}^c(x)$ and $b:=c-a$.
Suppose that $c=a'+b'$ is another decomposition. 
If $\|a'\|_X<x$, then we have $\|b'\|_Y=\|c-a'\|_Y\geq f^c_{YX}(\|a'\|_X)>f^c_{YX}(x)=y$.
If $\|a'\|_X=x$ then we have $\|b'\|_Y=\|c-a'\|_Y\geq  f^c_{YX}(\|a'\|_X)=f^c_{YX}(x)=y$.
 We conclude that $(x,y)$ is Pareto efficient.
\end{proof}

\begin{corollary}\ 
\begin{enumerate}
\item  The function $f_{YX}:[0,\|c\|_X]\to [0,\|c\|_Y]$ is a homeomorphism and its inverse is $f_{XY}$. 
\item A vector $a$ is a solution to ${\bf M}_{YX}(x)$ if and only if $c-a$ is a solution to ${\bf M}_{XY}(f_{YX}(x))$.
\end{enumerate}

\end{corollary}
\begin{proof}
(1) If $0\leq x\leq \|c\|_X$ and $0\leq y\leq \|c\|_Y$ then we have
$$
f_{YX}(x)=y\Leftrightarrow \mbox{$(x,y)$ is Pareto efficient}\Leftrightarrow f_{XY}(y)=x.
$$
So$ f_{XY}$ and $f_{YX}$ are inverse of each other. Since both functions are continuous, the functions are homeomorphisms.

If $y=f_{XY}(x)$ and $b=c-a$ then we have
\begin{multline*}
\mbox{$a$ solution to ${\bf M}_{YX}(x)$}\Leftrightarrow c=a+b\mbox{ is a $XY$ decomposition}\Leftrightarrow\\ \Leftrightarrow \mbox{$c-a=b$ is a solution to ${\bf M}_{XY}(y)={\bf M}_{XY}(f_{YX}(x))$}
\end{multline*}
%
\end{proof}



\subsection{Rigid norms}
The problem ${\bf M}_{YX}^c(x)$ does not always have a unique solution. 
\begin{definition}
We say that $c\in V$ is {\em rigid} if ${\bf M}_{YX}^c(x)$ has a {\em unique} solution for all $x\in [0,\|c\|_X]$.
\end{definition}
Let us give an example of a vector that is not rigid.
\begin{example}
Suppose $V=\R^2$ and $\|c\|_X=\|c\|_Y=\|c\|_{\infty}$ for all $c\in V$.
Then $(1,1)^t$
is rigid because for $x\in [0,1]$ the vector $(x,x)^t$
is the unique solution to ${\bf M}^c_{YX}(x)$.
The vector $c=(1,0)^t$
is not rigid: If $0<x<1$ then ${\bf M}^c_{YX}(x)$ has infinitely many solutions, namely
$(x,s)^t$, $|s|\leq \min\{x,1-x\}$.
\end{example}
If $c\in V$ is rigid, then we can study how the unique solution of ${\bf M}_{YX}^c(x)$ varies as we change the value of $x$. The lemma below shows that the solution varies continuously. In various contexts, this property is well-known and used in homotopy continuation methods (see for example~\cite{OPT2,OPT,EJHT,vdBF}) for some optimization problem.

\begin{lemma}
Suppose that $c\in V$ is rigid and let $\alpha_{YX}(x)=\alpha_{YX}^c(x)$ be the unique solution to ${\bf M}_{YX}^c(x)$ for $x\in [0,\|c\|_X]$.
Then $\alpha_{YX}:[0,\|c\|_X]\to V$  is continuous.
\end{lemma}
\begin{proof}
Suppose that $x_1,x_2,\dots\in [0,\|c\|_X]$ is a sequence for which $x=\lim_{n\to\infty}x_n$ exists.
We assume that  $\lim_{n\to\infty}\alpha_{YX}(x_n)$ does not exist, or that it is not equal to $\alpha_{YX}(x)$.
By replacing $x_1,x_2,\dots$ by a subsequence, we may assume that $a=\lim_{n\to\infty}\alpha_{YX}(x_n)$ exists,
but that it is not equal to $\alpha_{YX}(x)$. We have $\|a\|_X=\lim_{n\to\infty} \|\alpha_{YX}(x_n)\|_X=\lim_{n\to\infty} x_n=x$. Also, we get $\|c-a\|_Y=\lim_{n\to\infty}\|c-\alpha_{YX}(x_n)\|_Y=\lim_{n\to\infty}f_{YX}(x_n)=f_{YX}(x)$
because $f$ is continuous. Because ${\bf M}_{YX}(x)$ has a  unique solution, we conclude
that $a=\alpha_{YX}(x)$. Contradiction. We conclude that $\lim_{n\to\infty} \alpha_{YX}(x_n)=\alpha_{YX}(x)$. This proves that $\alpha_{YX}$ is continuous. 
\end{proof}

\begin{definition}
The norm $\|\cdot\|_Y$ is called {\em strictly convex} if  $\|a+b\|_Y=\|a\|_Y+\|b\|_Y$ implies that $a$ and $b$ are linearly dependent.
\end{definition}
The $\ell_2$-norm on $\R^n$ is strictly convex, for example.
\begin{lemma}
If $\|\cdot\|_Y$ is strictly convex, then every vector is rigid.
\end{lemma}
\begin{proof}
Suppose that $\|\cdot\|_Y$ is strictly convex and that $c\in V$. We will prove that $c$ is rigid. Suppose
 $\|a\|_X=\|a'\|_X=x$ and $\|b\|_{Y}=\|b'\|_Y=f_{YX}(x)$, where $b=c-a$ and $b'=c-a'$.
Let $\overline{a}=(a+a')/2$. Then we have $\|\overline{a}\|_X\leq \frac{1}{2}(\|a\|_X+\|a'\|_X)=x$.
By definition, we have $\|c-\overline{a}\|_Y\geq f_{YX}(x)$. It follows that 
$$\|b+b'\|_Y=\|2(c-\overline{a})\|_Y=2\|c-\overline{a}\|_Y\geq 2f_{YX}(x)=\|b\|_Y+\|b'\|_Y.$$
Since $\|\cdot\|_Y$ is strictly convex, $b$ and $b'$ are linearly dependent. Because $\|b\|_Y=\|b'\|_Y$, it follows that $b=\pm b'$.
If $b=-b'$ then we have $b+b'=0$. From $\|b\|_Y+\|b'\|_Y\leq \|b+b'\|_Y=0$ follows that $b=b'=0$ and 
  $a=a'=c$ and the uniqueness is established.
If $b=b'$, then $a=a'$ and we have again uniqueness. 
\end{proof}
\subsection{The Pareto curve of sums}
Next we will compare the Pareto curve of $f_{YX}^{a+b}$ with the Pareto curves $ f_{YX}^a$ and $f_{YX}^b$. For this purpose, we introduce the concatenation of two functions.
Suppose that $u:[0,t]\to \R$ and $v:[0,s]\to \R$ are two  functions with $u(t)=v(s)=0$. The {\em concatenation} $u\star v:[0,s+t]\to \R$ is defined by
$$
u\star v(x)=\begin{cases}
u(x)+v(0) & \mbox{if $0\leq x\leq t$}\\
v(x-t) & \mbox{if $t\leq x\leq s+t$}.
\end{cases}
$$
Note that $(u\star v)(0)=u(0)+v(0)$ and $(u\star v)(s+t)=v(s)=0$.
If $u$ and $v$ are decreasing, then so is $u\star v$. If $u$ and $v$ are continuous, then so  is $u\star v$. Note that concatenation is associative: $(u\star v)\star w=u\star (v\star w)$.
However,  it is not commutative.
\begin{example}
Suppose that $u,v:[0,1]\to \R$ are defined by $u(x)=1-x$ and $v=(1-x)^2$:

\centerline{\includegraphics[width=2.5in]{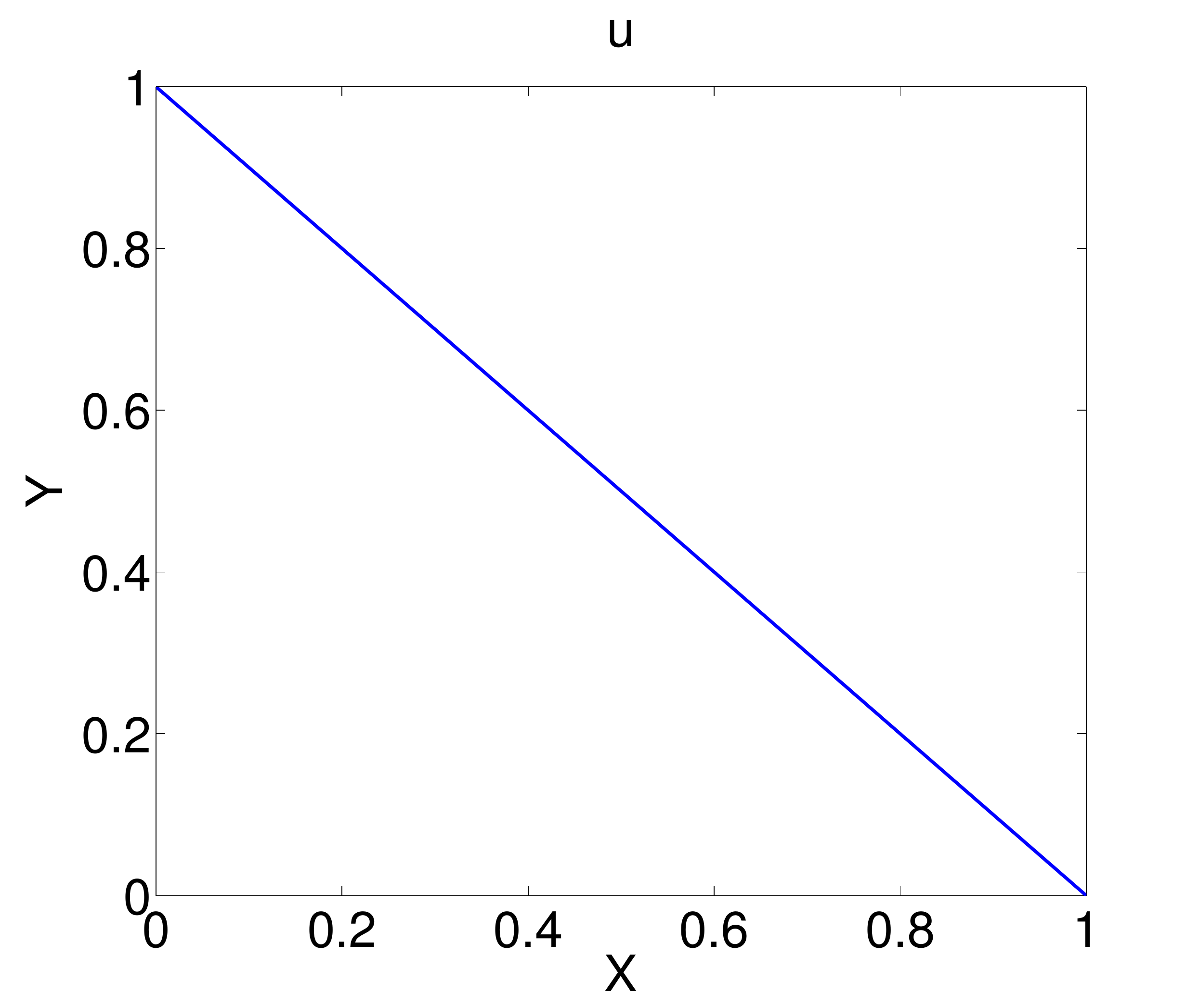}\includegraphics[width=2.5in]{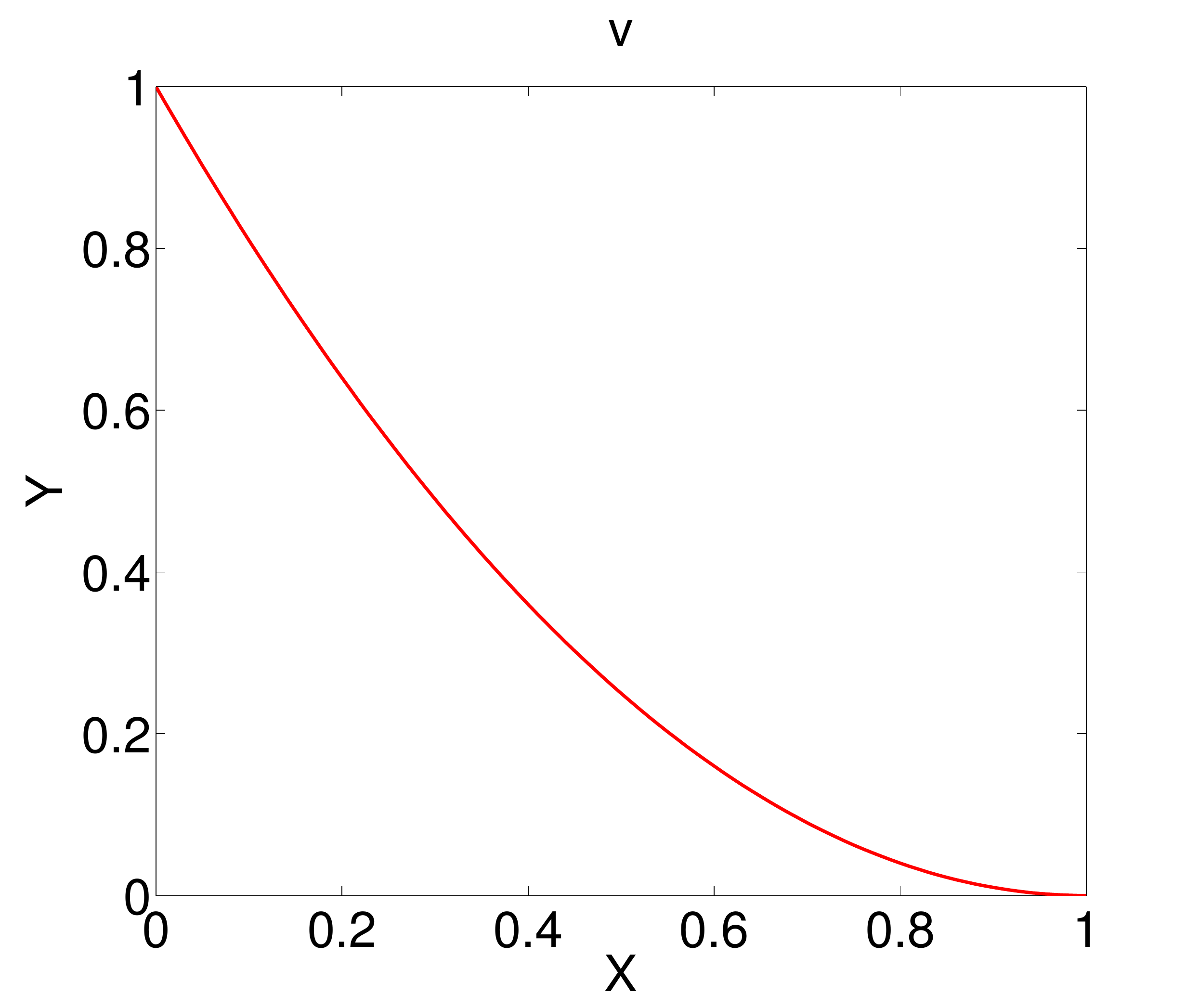}}

The graphs of $u\star v$ and $v\star u$ are:

\centerline{\includegraphics[width=2.5in]{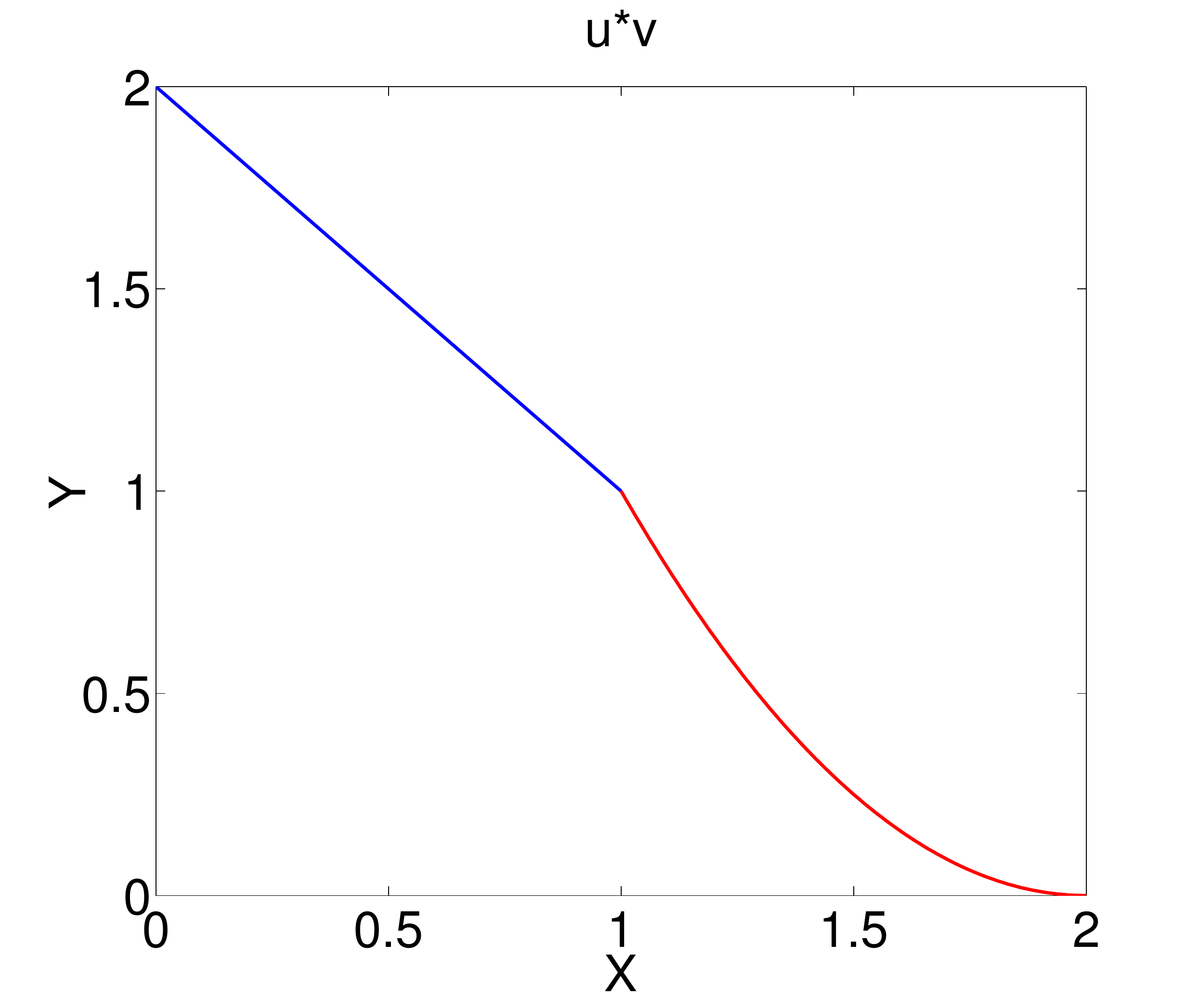}\includegraphics[width=2.5in]{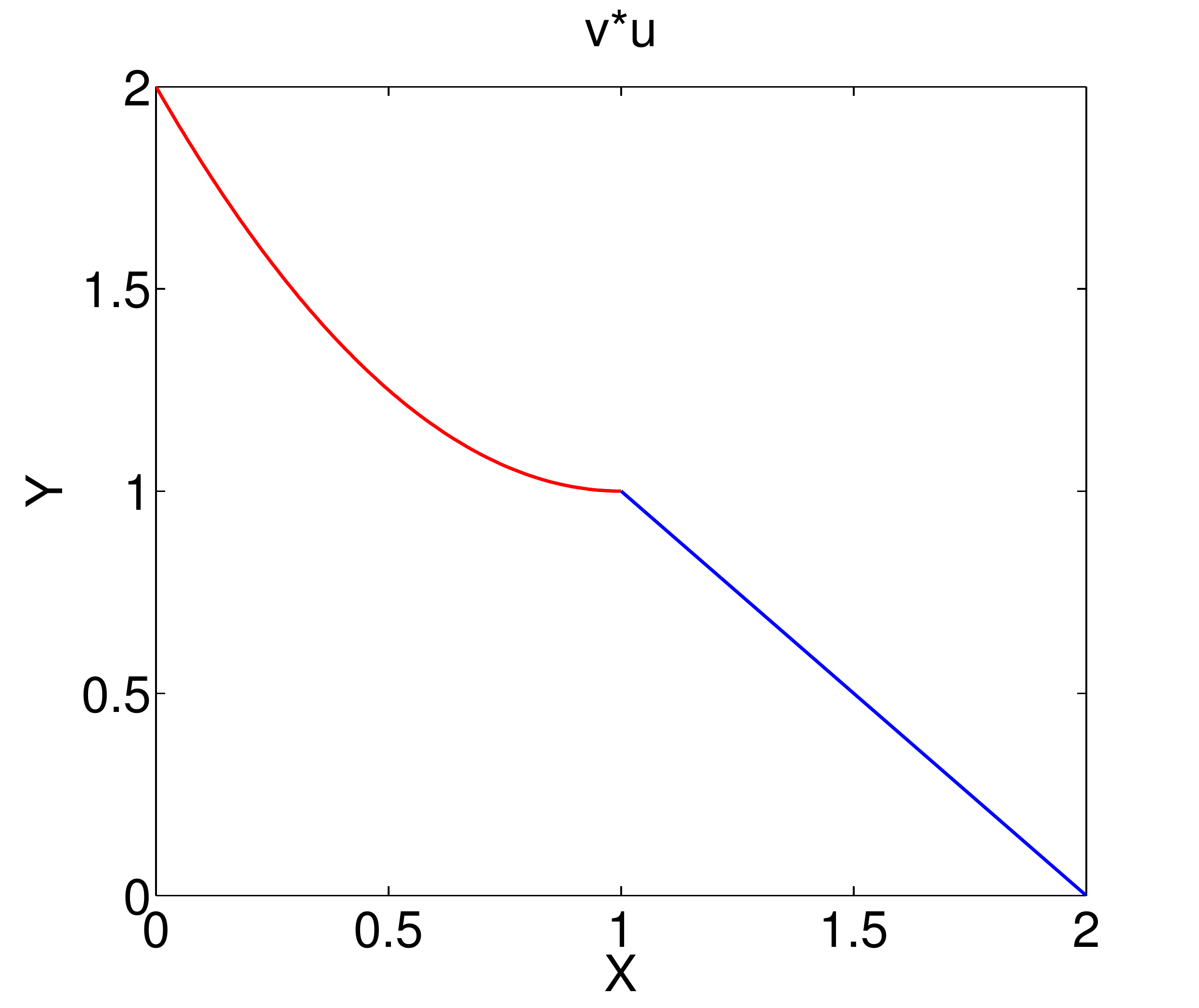}}
\end{example}

\begin{lemma}
Suppose that $a,b\in V$.
If $0\leq x\leq \|a+b\|_X$ then we have
$$
f_{YX}^{a+b}(x)\leq (f_{YX}^a\star f_{YX}^b)(x).
$$
\end{lemma}
\begin{proof}
Suppose that $0\leq x\leq \|a\|_X$.
Choose $d\in V$ such that $\|d\|_X=x$ and $\|a-d\|_Y=f_{YX}^a(x)$. Then we have
$$
f_{YX}^{a+b}(x)\leq \|a+b-d\|_Y\leq \|a-d\|_Y+\|b\|_Y=f_{YX}^a(x)+f_{YX}^b(0)=(f_{YX}^a\star f_{YX}^b)(x).
$$

Reversing the roles of $X$ and $Y$, and $a$ and $b$ gives
$$
f_{XY}^{a+b}(y)\leq f_{XY}^b(y)+f_{XY}^a(0)=f_{XY}^b(y)+\|a\|_X
$$
if $0\leq y\leq \|b\|_Y$.
Substituting $y=f_{YX}^{a+b}(x)$ where $\|a\|_X\leq x\leq \|a+b\|_X$ yields
$$
x-\|a\|_X\leq f_{XY}^b(f_{YX}^{a+b}(x)).
$$
Applying the decreasing function $f_{YX}^b$ gives
$$
(f_{YX}^a\star f_{YX}^b)(x)=f_{YX}^b(x-\|a\|_X)\geq f_{YX}^{a+b}(x).
$$
\end{proof}

\section{Duality}\label{sec:3}
\subsection{$X2$- and $2Y$-decompositions}
Suppose that the vector space $V$ is equipped with a positive definite bilinear form $\langle\cdot,\cdot\rangle$ and a norm $\|\cdot\|_X$. The bilinear form gives an $\ell_2$-norm $\|\cdot\|_2$ and let $\|\cdot\|_Y$ be the dual norm of $\|\cdot\|_X$. In this section we will study $X2$-decompositions, which turn out to be the same as $2Y$-decompositions. We start with the following characterization of an $X2$-decomposition:
\begin{proposition}\label{prop:XYchar}
The expression $c=a+b$ is an $X2$-decomposition if and only if $\langle a, b\rangle=\|a\|_X\|b\|_Y$.
\end{proposition}
\begin{proof}
Suppose that $c=a+b$ is an $X2$-decomposition.
Choose a vector $d$ such that $\langle b,d\rangle=\|b\|_Y$ and $\|d\|_X=1$. Let $\varepsilon>0$.
Define $a'=(1-\varepsilon)a+\varepsilon \|a\|_Xd$ and $b'=c-a'$. We have $\|a'\|_X\leq (1-\varepsilon)\|a\|_X+\varepsilon \|a\|_X=\|a\|_X$. Therefore, $\|b'\|_2\geq \|b\|_2$. 
It follows that
\begin{multline*}
0\leq \|b'\|_2^2-\|b\|^2_2=\|b+(a-a')\|_2^2-\|b\|_2^2=2\langle a-a',b\rangle+\|a-a'\|_2^2=\\
=2\varepsilon\langle a-\|a\|_Xd,b\rangle+\varepsilon^2\|a-\|a\|_Xd\|_2^2.
\end{multline*}
Taking the limit $\varepsilon\downarrow 0$ yields the inequality
$$
0\leq\langle a-\|a\|_Xd,b\rangle=\langle a,b\rangle-\|a\|_X\|b\|_Y,
$$
so $\langle a,b\rangle\geq \|a\|_X\|b\|_Y$.
The opposite inequality $\langle a,b\rangle\leq \|a\|_X\|b\|_Y$ holds because the norms are dual to each other. We conclude that $\langle a,b\rangle=\|a\|_X\|b\|_Y$.

Conversely, suppose that $\langle a,b\rangle=\|a\|_X\|b\|_Y$. Let $x=\|a\|_X$, let $a'$ be the solution to ${\bf M}_{2X}^c(x)$ and define $b'=c-a'$.
Then $c=a'+b'$ is an $X2$-decomposition. 
\begin{multline*}
\|a'-a\|_2^2=\langle a'-a,b-b'\rangle=
\langle a',b\rangle+\langle a,b'\rangle-\langle a',b'\rangle-\langle a,b\rangle=\\=
\langle a',b\rangle+\langle a,b'\rangle-x\|b\|_Y-x\|b'\|_Y\leq
x\|b'\|_Y+x\|b\|_Y-x\|b'\|_Y-x\|b\|_Y=0.
\end{multline*}
So we conclude that $a=a'$, and $c=a+b$ is an $X2$-decomposition.
\end{proof}
The equivalence between the $X2$-decomposition and $2Y$-decomposition (Proposition~\ref{prop:equivalent}) now easily follows.
\begin{proof}[Proof of Proposition~\ref{prop:equivalent}]
In Proposition~\ref{prop:equivalent}, (1) and (3) are equivalent because of Proposition~\ref{prop:XYchar}. 
Dually, (2) and (3) are equivalent. 
\end{proof}

\subsection{The Pareto sub-frontier} We define $h_{YX}=h^c_{YX}:[0,\|c\|_X]\to [0,\|c\|_Y]$ by $h_{YX}(x)=\|c-\alpha_{2X}(x)\|_Y$. The graph of $h_{YX}^c$ is the Pareto subfrontier. 
Indeed, if $c=a+b$ is an $X2$-decomposition with $\|a\|_X=x$ and $\|b\|_Y=y$, then we have 
$\alpha_{2X}(x)=a$ and $h_{YX}(x)=\|c-\alpha_{2X}(x)\|_Y=\|c-a\|_Y=\|b\|_Y=y$. We now prove some properties of the Pareto sub-frontier.

\begin{lemma}
We have $\alpha_{2Y}(h_{YX}(x))=c-\alpha_{2X}(x)$.
\end{lemma}
\begin{proof}
Let $a=\alpha_{2X}(x)$ and $b=c-a$. Then $c=a+b$ is an $X2$-decomposition, therefore also a $2Y$-decomposition.
So $c=b+a$ is a $Y2$-decomposition. Let $y=\|b\|_Y=h_{YX}(x)$. Then we have $b=\alpha_{2Y}(y)=\alpha_{2Y}(h_{YX}(x))$.
%
%
%
\end{proof}
\begin{lemma}
The function $h_{YX}(x)$ is a strictly decreasing homeomorphism and its inverse is $h_{XY}(x)$.
\end{lemma}
\begin{proof}
We have have
$$
\alpha_{2X}(h_{XY}(h_{YX}(x)))=c-\alpha_{2Y}(h_{YX}(x))=\alpha_{2X}(x).
$$
The function $\alpha_{2X}$ is injective, because the function $f_{2X}(x)=\|c-\alpha_{2 X}(x)\|_2$ is injective.
It follows that $h_{XY}(h_{YX}(x))=x$. By symmetry, we also have $h_{YX}(h_{XY}(y))=y$, so $h_{XY}$ is the inverse
of $h_{YX}$.

The function $h_{YX}(x)=\|c-\alpha_{2X}(x)\|_Y$ is continuous, because $\alpha_{2X}$ and $\|\cdot \|_Y$ are continuous.
This proves that $h_{YX}$ is a homeomorphism. By the Intermediate Value Theorem, it has to be
strictly increasing or strictly decreasing. Since $h_{YX}(\|c\|_X)=0\leq h_{YX}(0)=\|c\|_Y$, the function $h_{YX}$ must be
strictly decreasing.
\end{proof}

\begin{proposition}
The function $h_{YX}$ is Lipschitz continuous.
\end{proposition}
\begin{proof}
Since $\|\cdot\|_2$ and $\|\cdot\|_X$ are norms on a finite dimensional vector space, there exists positive constant $K>0$ such that
$\|c\|_Y\leq K\|c\|_2$. 
Suppose that $c=a_1+b_1$ and $c=a_2+b_2$ are $XY$-decompositions, with $x_2:=\|a_2\|_X>x_1:=\|a_1\|_X$. It follows that $y_2:=\|b_2\|_Y<y_1:=\|b_1\|_Y$.
We have
\begin{multline*}
K^{-2}(y_1-y_2)^2\leq  K^{-2}\|b_1-b_2\|_X^2   \leq \|b_1-b_2\|_2^2=\langle a_2-a_1,b_1-b_2\rangle\leq\\ \leq  \langle a_2,b_1\rangle+\langle a_1,b_2\rangle-\langle a_1,b_1\rangle-\langle a_2,b_2\rangle=
\langle a_2,b_1\rangle+\langle a_1,b_2\rangle-x_1y_1-x_2y_2\leq \\ \leq  x_2y_1+x_1y_2-x_1y_1-x_2y_2=(x_2-x_1)(y_1-y_2)
\end{multline*}

We conclude that
$$
\frac{|y_2-y_1|}{|x_2-x_1|}=\frac{y_1-y_2}{x_2-x_1}\leq K^2.
$$
\end{proof}

\subsection{Differentiating the Pareto curve}
The function $f_{2X}(x)$ is differentiable. A special case (but with a similar prove) was treated in \cite[\S2]{vdBF}.

\begin{proposition}\label{prop:derivativefsquared}
The function $f_{2X}(x)$ is differentiable on $[0,\|c\|_X]$, and
$$
f_{2X}'(x)f_{2X}(x)=-h_{YX}(x).
$$
\end{proposition}
\begin{proof}
If $0\leq x_1,x_2\leq \|c\|_X$ then we have
\begin{multline*}
(f_{2X}(x_2))^2-(f_{2X}(x_1))^2\geq \|c-\alpha_{2X}(x_2)\|_2^2-\|c-\alpha_{2X}(x_1)\|_2^2-\|\alpha_{2X}(x_2)-\alpha_{2X}(x_1)\|_2^2=\\
2\langle c-\alpha_{2X}(x_1),\alpha_{2X}(x_1)-\alpha_{2X}(x_2)\rangle=2x_1h_{YX}(x_1)-2\langle c-\alpha_{2X}(x_1),\alpha_{2X}(x_2)\rangle\geq\\
\geq 2(x_1-x_2)h_{YX}(x_1)=-2(x_2-x_1)h_{YX}(x_1).
\end{multline*}
Reversing the roles of $x_1,x_2$ gives us
$$
(f_{2X}(x_2))^2-(f_{2X}(x_1))^2\leq -2(x_2-x_1)h_{YX}(x_2).
$$
if $x_2>x_1$ then we obtain
$$
-2h_{YX}(x_1)\leq \frac{(f_{2X}(x_2))^2-(f_{2X}(x_1))^2}{x_2-x_1}\leq -2h_{YX}(x_2)
$$
and if $x_1>x_2$ then we have
$$
-2h_{YX}(x_2)\leq \frac{(f_{2X}(x_2))^2-(f_{2X}(x_1))^2}{x_2-x_1}\leq -2h_{YX}(x_1).
$$
Since $h_{YX}$ is continuous, it follows that $(f_{2X})^2$ is differentiable on $[0,\|c\|_X]$ with derivative $-2h_{YX}$. Since $f_{2X}(x)$ is positive on $[0,\|c\|_X)$,
it is differentiable on $[0,\|c\|_X)$. We have
$$
2f_{2X}'(x)f_{2X}(x)=\frac{d}{dx} \Big(f_{2X}(x)\Big)^2=-2h_{YX}(x).
$$
\end{proof}
\begin{proof}[Proof of Proposition~\ref{prop:areas}]
From Proposition~\ref{prop:derivativefsquared} follows that 
$$\textstyle \frac{d}{dx}\big(-\frac{1}{2}f_{2X}(x)^2\big)=-f_{2X}'(x)f_{2X}(x)=h_{YX}(x).
$$
So the area to the right of $x=\|a\|_X$ is
$$
\int_{\|a\|_X}^{\|c\|_X} h_{YX}(x)\,dx=\left[-{\textstyle \frac{1}{2}}f_{2X}(x)^2\right]^{\|c\|_X}_{\|a\|_X}=-{\textstyle \frac{1}{2}}(0^2-\|b\|_2^2)={\textstyle \frac{1}{2}}\|b\|_2^2.
$$
Similarly, the area below the graph and above the line $y=\|b\|_Y$ is equal to $\|a\|_2^2$. The area below the graph of $h_{YX}$ is equal to $\frac{1}{2}\|c\|_2^2=\frac{1}{2}\|a\|_2^2+\frac{1}{2}\|b\|_2^2+\langle a,b\rangle$.
The area of the $\|a\|_X\times \|b\|_Y$ rectangle is $\|a\|_X\|b\|_Y=\langle a,b\rangle$.
\end{proof}

The solution for ${\bf M}^c_{2X}(\lambda)$ can be obtained from a regularized quadratic minimization problem.
\begin{proposition}\label{prop:BPDNformulations}
The vector $a$ is a solution to ${\bf M}_{2X}^c(\lambda)$ if and only if
$$
\textstyle \|c-a\|_Y+\frac{1}{2\lambda}\|a\|_2^2
$$
is minimal.
\end{proposition}
\begin{proof}
We can choose $a$ such that $\|c-a\|_Y+\frac{1}{2\lambda}\|a\|_2^2$ is minimal. Let $b=c-a$. Then $c=a+b$ is an $X2$-decomposition,
so $\langle a,b\rangle=\|a\|_X\|b\|_Y$. The function
$$
f(t)=\textstyle \|tb\|_Y+\frac{1}{2\lambda}\|c-tb\|_2^2=t\|b\|_Y+\frac{1}{\lambda}\langle c-bt,c-bt\rangle 
$$
has a minimum at $t=1$. So we have
$$
0=f'(1)=\textstyle \|b\|_Y-\frac{1}{\lambda}\langle a,b\rangle=\|b\|_Y-\frac{1}{\lambda}\|a\|_X\|b\|_Y
$$
and $\|a\|_X=\lambda$. This shows that $a$ is solution ${\bf M}_{2X}^c(\lambda)$. Since $M_{2X}^c(\lambda)$ has a unique solution, this unique solution $a$ must minimize $\|c-a\|_Y+\frac{1}{2\lambda}\|a\|_2^2$.
\end{proof}

A similar argument shows that $a$ is a solution to ${\bf M}_{2X}^c(\lambda)$ if and only if $\|c-a\|_Y+\frac{1}{\lambda}\|a\|_2$ is minimal.

\section{Tight vectors}
Suppose that $V$ is an $n$-dimensional vector space with a positive definite bilinear form, and that $\|\cdot\|_X$ and $\|\cdot\|_Y$ are norms which are dual to each other.
From the definitions it is clear that $f_{YX}^c(x)\leq h_{YX}^c(x)$. Recall that $c$ is {\em tight} if we have equality for all $x\in [0,\|c\|_X]$.
If $c\in V$ is tight and rigid, then  $\alpha_{2X}(x)=\alpha_{YX}(x)$ for $x\in [0,\|c\|_X]$. If every vector in $V$ is tight, then $\|\cdot\|_X$ and $\|\cdot\|_Y$ are called {\em tight norms}. In this section we study properties of tight vectors and tight norms.

\subsection{An example of a norm that is not tight}
Consider the norm on $\R^2$ defined by $\|c\|_X=\max\{|c_2|,|2c_2-c_1|\}$ for $c=(c_1,c_2)^t$. Its dual norm is given by $\|c\|_Y=\max\{|c_1+c_2|,|3c_1+c_2|\}$.
The unit balls are polar duals of each other:

\centerline{\includegraphics[width=3in]{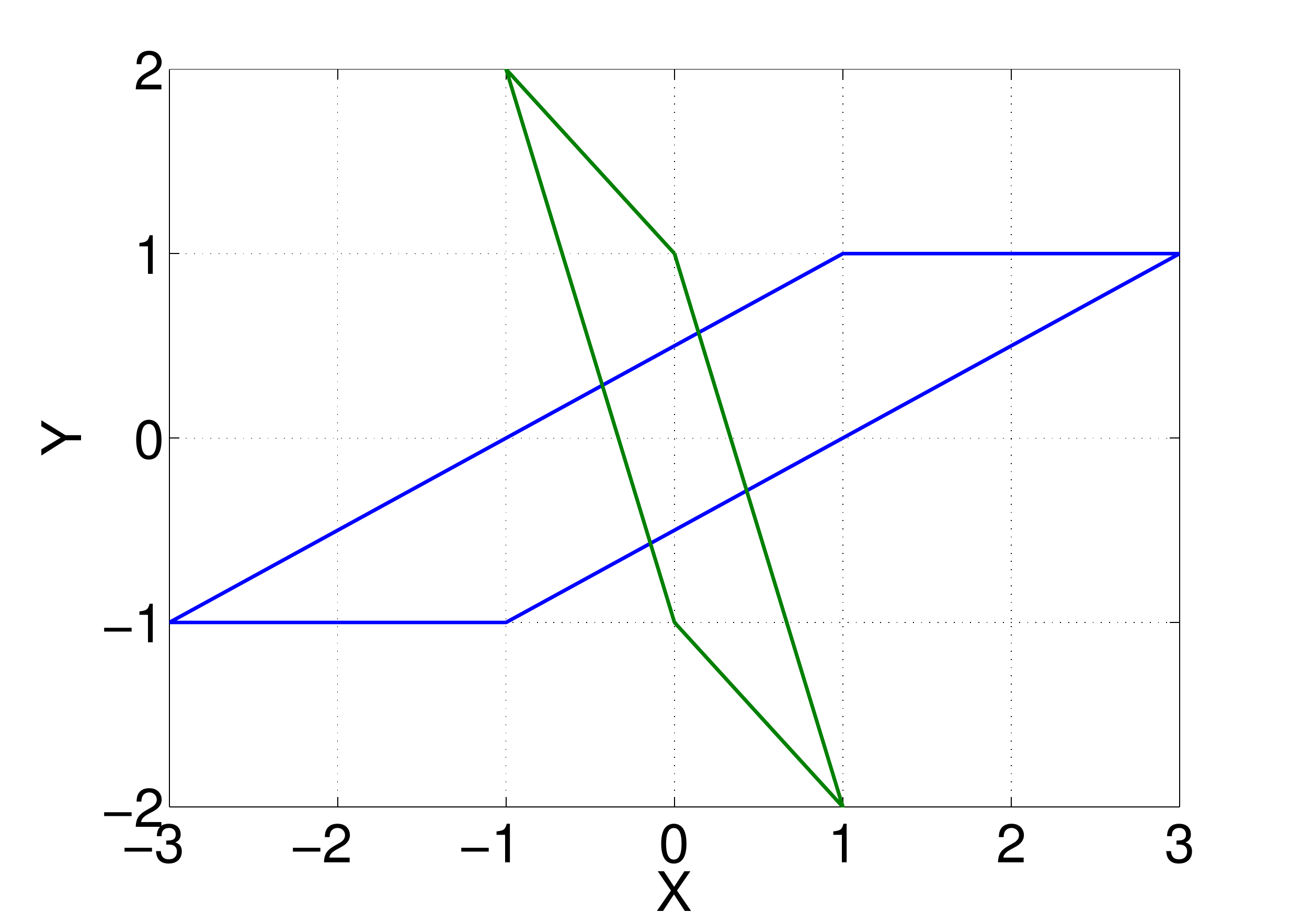}}

Consider the vector $c=(3,12)^t$. Below are the functions $f_{YX}^c$ and $h_{YX}^c$.
We see that $f^c_{YX}$ and $h^c_{YX}$ are not the same, so $c$ is not tight. The example shows that $h_{YX}^c$ is not always convex.

\centerline{\includegraphics[width=4in]{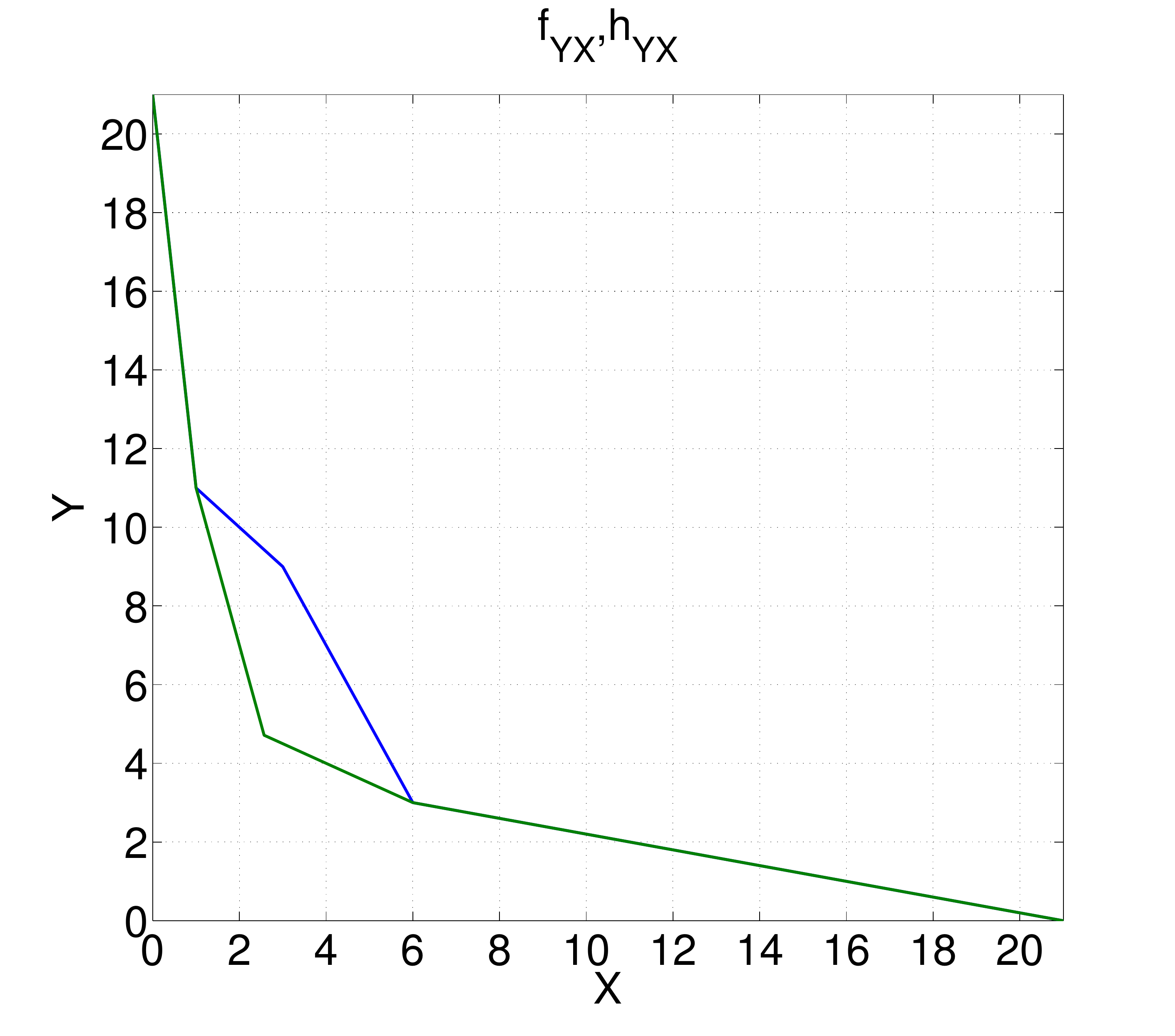}}

The trajectories of $a_{YX}^c$ (green) and $a_{2X}^c$ (blue) are sketched in the  graph below.

\centerline{\includegraphics[width=2.5in]{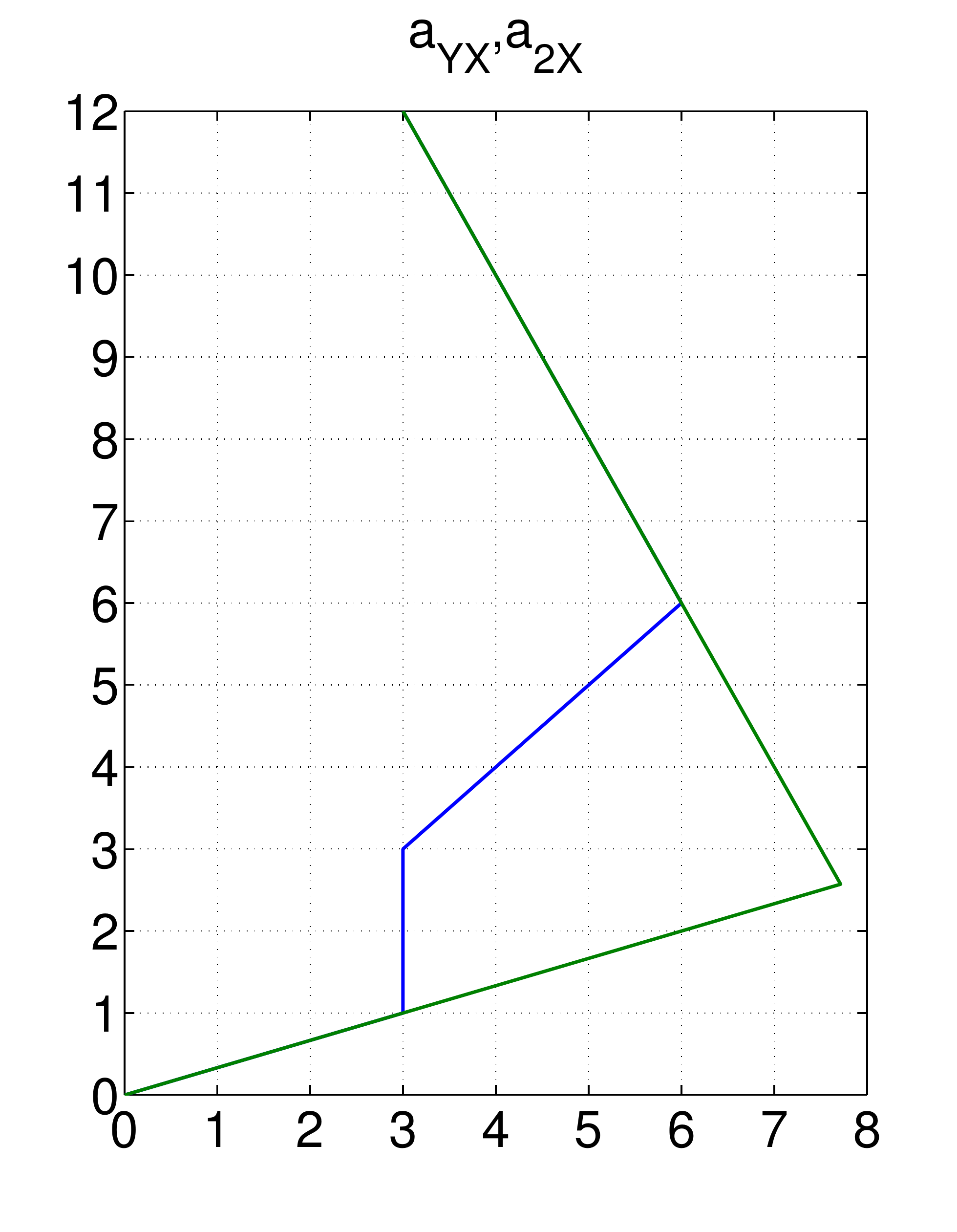}}

For every positive value of $x$, $a_{2X}^c(x)/x$ lies on the unit ball $B_X$. In fact, $a_{2X}^c(x)/x$ is the vector in $B_X$ that is closest to $c/x$ with respect to the euclidean distance. In the  graph below, $c/x$ and $a_{2X}^c(x)/x$ are plotted for  for various values of $x$. Note that $a_{2X}^c(x)/x$ is constant  on the intervals $(0,1]$ and $[3,6]$.
On these intervals $a_{2X}^c(x)$ moves on a line through the origin. On the other intervals $[1,3]$ and $[6,21]$, $a_{2X}^c(x)$ moves on a line through $c$.

\centerline{\includegraphics[width=2in]{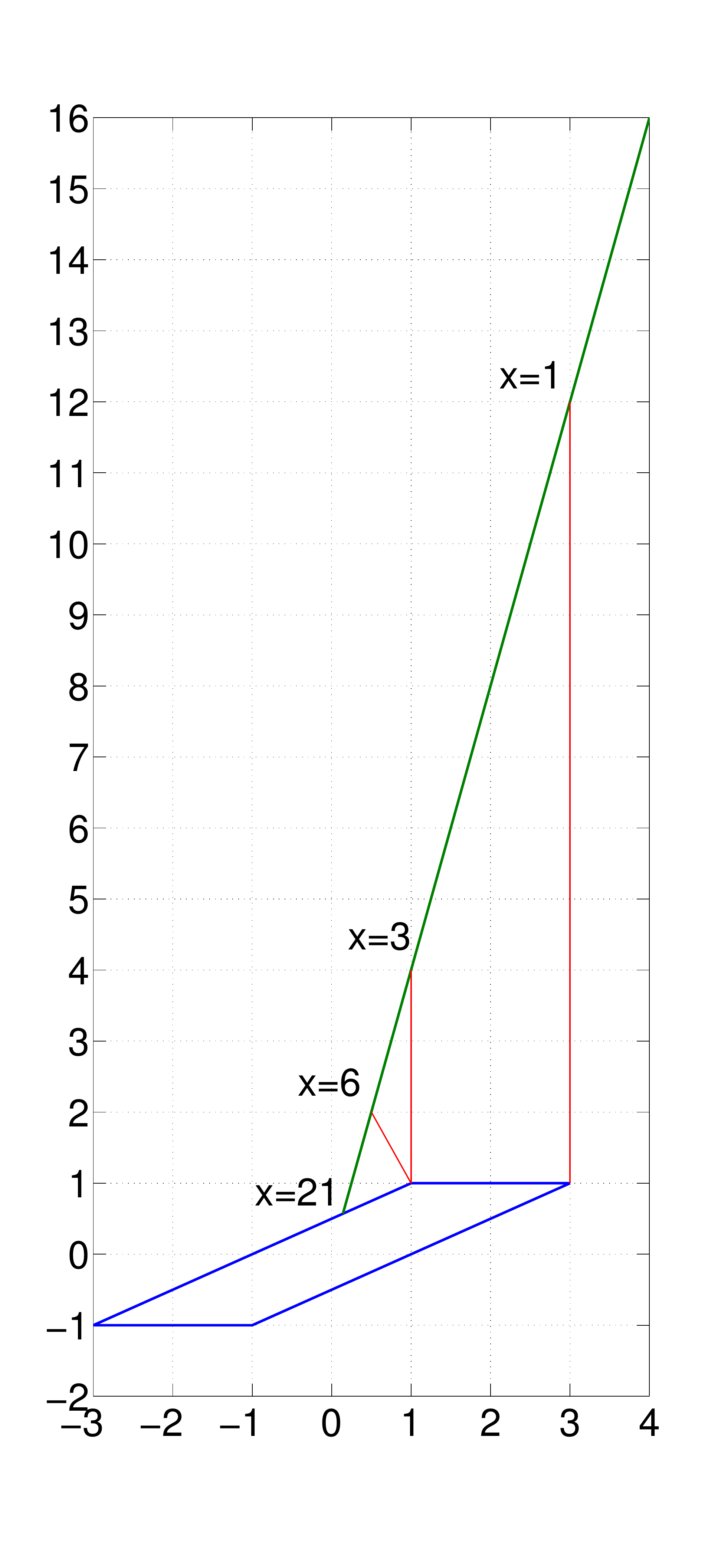}}

The singular value region for the vector $c$ is as follows:

\centerline{\includegraphics[width=2.5in]{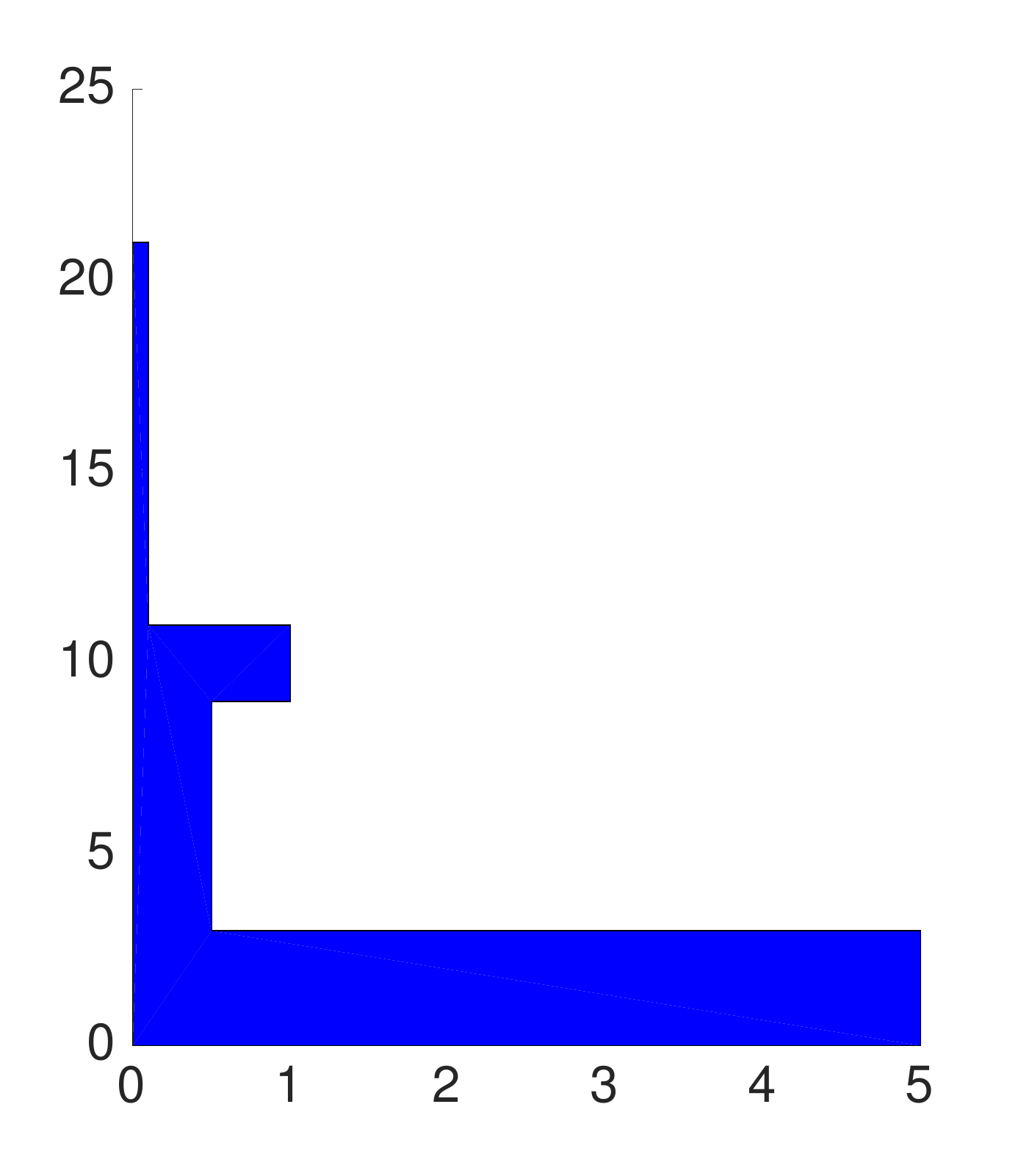}}

If we want to iterpret the singular value region in terms of singular values, we must allow negative multiplicities.
The singular values of $c$ are  21 with multiplicity $0.1$, $11$ with multiplicity $0.9$, $9$ with multiplicity $-0.5$ and $3$ with multiplicity $4.5$.
\subsection{The Pareto sub-frontier of sums}

We have defined the concatenation of two graphs,  and now we define the concatenation of two paths in $V$ in a similar manner.
If $\beta:[0,t]\to V$ and $\gamma:[0,s]\to V$ are curves with $\beta(0)=\gamma(0)=0$, then we define the concatenation $\beta\star\gamma:[0,s+t]\to V$ by
$$
(\beta\star \gamma)(x)=\begin{cases}
\beta(x) & \mbox{if $0\leq x\leq t$,}\\
\beta(t)+\gamma(s-t) & \mbox{if $t\leq x\leq s+t$.}
\end{cases}
$$
\begin{theorem}\label{thm:cabtight}
Suppose that $c\in V$ is tight and that $c=a+b$ is an $XY$-decomposition.
\begin{enumerate}
\item  $a$ and $b$ are tight as well;
\item  $\|c\|_X=\|a\|_X+\|b\|_X$ and $\|c\|_Y=\|a\|_Y+\|b\|_Y$;
\item $f_{YX}^c=f_{YX}^a\star f_{YX}^b$;
\item $\alpha_{YX}^c=\alpha_{YX}^a\star \alpha_{YX}^b$.
\end{enumerate}
\end{theorem}
\begin{proof}
Suppose that $0\leq x\leq \|a\|_X$. Choose a vector $d$ such that $\|d\|_X=x$ and $\|a-d\|_Y=f_{YX}^a(x)$.
We have
$$
h_{YX}^c(x)-\|b\|_Y= f_{YX}^c(x)-\|b\|_Y   \leq \|c-d\|_Y-\|c-a\|_Y\leq \|a-d\|_Y=f_{YX}^a(x)\leq h_{YX}^a(x).
$$
Suppose that one of the inequalities above is strict for some $x\in [0,\|a\|_X]$. 
Integrating $t$ from $0$ to $x$ yields
\begin{multline*}
{\textstyle \frac{1}{2}}\|a\|^2_2=\int_0^{\|a\|_X} h_{YX}^a(t)\,dt>\int_0^{\|a\|_X} (h_{YX}^c(x)-\|b\|_Y)\,dt=\\={\textstyle \frac{1}{2}}(\|c\|_2^2-\|a\|_2^2)-\|a\|_X\|b\|_Y=
{\textstyle \frac{1}{2}}(\|a+b\|_2^2-\|b\|_2^2-2\langle a,b\rangle)={\textstyle \frac{1}{2}}\|a\|_2^2.
\end{multline*}
Contradiction. We conclude that $h_{YX}^a(x)=f_{YX}^a(x)=h_{YX}^c(x)-\|b\|_Y$ for all $x\in [0,\|a\|_X]$. In particular, $a$ is tight. By symmetry, $b$ is tight as well. This proves (1).
For $x=0$ we get $\|a\|_Y=h_{YX}^a(0)=h_{YX}^c(0)-\|b\|_Y=\|c\|_Y-\|b\|_Y$. So we have $\|c\|_Y=\|a\|_Y+\|b\|_Y$ and by symmetry we also have
$\|c\|_X=\|a\|_X+\|b\|_X$. This proves (2).  

If $0\leq x\leq \|a\|_X$, then we have 
$$
f_{YX}^c(x)=h_{YX}^c(x)=h_{YX}^a(x)+\|b\|_Y=f_{YX}^a(x)+f_{YX}^b(0)=(f_{YX}^a\star f_{YX}^b)(x).
$$
By symmetry, if $0\leq y\leq \|b\|_Y$, then we have
$$
f_{XY}^c(y)=f_{XY}^b(y)+\|a\|_X.
$$
It follows that
$$
f_{YX}^b(f_{XY}^c(y)-\|a\|_X)=f_{YX}^b(f_{XY}^b(y))=y.
$$
Substituting $y=f_{YX}^c(x)$ yields
$$
(f_{YX}^a\star f_{YX}^b)(x)=f_{YX}^b(x-\|a\|_X)=f_{YX}^c(x).
$$
This proves (3). 

If $x\in [0,\|a\|_X]$, then we have 
$$
\|(\alpha_{YX}^a\star \alpha_{YX}^b)(x)\|_X=\|\alpha_{YX}^a(x)\|_X=x
$$
So we get
$$
\|c-(\alpha_{YX}^a\star \alpha_{YX}^b)(x)\|_Y=\|c-\alpha_{YX}^a(x)\|_Y\leq \|b\|_Y+\|a-\alpha_{YX}^a(x)\|_Y=\|b\|_Y+h_{YX}^a(x)=f_{YX}^c(x).
$$
If $x\in [\|a\|_X,\|c\|_X]$ then we have
\begin{multline*}
\|(\alpha_{YX}^a\star \alpha_{YX}^b)(x)\|_X=\|a+\alpha_{YX}^b(x-\|a\|_X)\|_X\leq \\ \leq \|a\|_X+\|\alpha_{YX}^b(x-\|a\|_X)\|_X=
\|a\|_X+(x-\|a\|_X)=x
\end{multline*}
and
\begin{multline*}
\|c-(\alpha_{YX}^a\star \alpha_{YX}^b)(x)\|_Y=\|c-a-\alpha_{YX}^b(x-\|a\|_X)\|_Y=\\=\|b-\alpha_{YX}^b(x-\|a\|_X)\|_Y=f_{YX}^b(x-\|a\|_X)=f_{YX}^c(x).
\end{multline*}
So for all $x\in [0,\|c\|_X]$ we have
$$
\|(\alpha_{YX}^a\star \alpha_{YX}^b)(x)\|_X\leq x
$$
and
$$
\|c-(\alpha_{YX}^a\star \alpha_{YX}^b)(x)\|_{Y}\leq f_{YX}^c(x).
$$
So $(\alpha_{YX}^a\star \alpha_{YX}^b)(x)$ is the unique solution to ${\bf M}_{YX}^c(x)$ and therefore equal to $\alpha_{YX}^c(x)$. This proves (4).
\end{proof}
\begin{proof}[Proof of Proposition~\ref{prop:denoisetransitive}]
We already now part (1) in the case $x_1\leq x_2$. Assume that $x_1>x_2$.
Let  $c=a+b$ be the $XY$-decomposition with $\|a\|_X=x_1$. This is also a $X2$-decomposition and $a=\proj_X(c,x_1)$.
Now $a$ and $b$ are tight by Theorem~\ref{thm:cabtight}(1).
Let $a=a_1+a_2$ be an $XY$-decomposition with $\|a_1\|_X=x_2$. Now $a=a_1+a_2$ is also an $X2$-decomposition
and $a_1=\proj_X(a,x_2)=\proj_X(\proj_X(c,x_1),x_2)$. 
We have $a_1=\alpha_{YX}^a(x_1)=(\alpha_{YX}^a\star \alpha_{YX}^b)(x_1)=\alpha_{YX}^c(x_1)$.
So $c=a_1+(c-a_1)$ is an $XY$-decomposition (and $X2$-decomposition) and $a_1=\proj_X(c,x_2)$. 
This proves that $\proj_X(c,x_2)=\proj_X(\proj_X(c,x_1),x_2)$ and part (1)  has been proved.

Suppose that $c=a+b$ is an $X2$-decomposition with $\|a\|_X=x_1$. Then $b=\shrink_X(c,x_1)$.
Let $b=b_1+b_2$ be an $X2$-decomposition with $\|b_1\|_X=x_2$.  Then $b_2=\shrink_X(b,x_2)=\shrink_X(\shrink_X(c,x_1),x_2)$.
Similar reasoning as before shows that $c=(a+b_1)+b_2$ is an $X2$-decomposition.
Also $(a+b_1)$ is tight and $(a+b_1)=a+b_1$ is an $XY$-decomposition, so $\|a+b_1\|_X=\|a\|_X+\|b_1\|_X=x_1+x_2$.
Therefore, $b_2=\shrink_X(c,x_1+x_2)$ and we are done.
\end{proof}

\begin{lemma}\label{lem:XYstar}
If $c=a+b$ is an $X2$-decomposition and $a$ and $b$ are tight, then we have $h^{c}_{YX}= h^{a}_{YX}\star h^{b}_{YX}$.
\end{lemma}
\begin{proof}
Suppose that $0\leq x\leq \|a\|_X$ and let $d=\alpha_{YX}^a(x)=\alpha_{2X}^a(x)$.
We get 
\begin{multline*}
\|d\|_X\|c-d\|_Y\geq \langle d,c-d\rangle=\langle d,a-d\rangle+\langle d,b\rangle=\|d\|_X\|a-d\|_Y+\langle a,b\rangle-\langle a-d,b\rangle\geq\\ \geq
\|d\|_X\|a-d\|_Y+\|a\|_X\|b\|_Y-\|a-d\|_X\|b\|_Y=\|d\|_X\|a-d\|_Y+\|d\|_X\|b\|_Y\geq \|d\|_X\|c-d\|_Y.
\end{multline*}
It follows that 
$$\|d\|_X\|c-d\|_Y=\langle d,c-d\rangle=\|d\|_X(\|a-d\|_Y+\|b\|_Y)$$ 
So  $c=d+(c-d)$ is an $X2$-decomposition and $\|c-d\|_Y=\|a-d\|_Y+\|b\|_Y$.  We get
$$
h^c_{YX}(x)=\|c-d\|_Y= \|b\|_Y+\|a-d\|_Y=h^{a}_{YX}(x)+\|b\|_Y=(h_{YX}^a\star h_{YX}^b)(x).
$$
Suppose that $\|a\|_X\leq x\leq \|c\|_X$ and define $y=h^{c}_{YX}(x)$. Then we have $0\leq y\leq \|b\|_Y$ and
by symmetry we get
$$
x=h^c_{XY}(y)=(h_{XY}^b\star h_{XY}^a)(y)=h^{b}_{XY}(y)+\|b\|_X=h^{b}_{XY}(h^c_{YX}(x))+\|b\|_X
$$
and
$$
(h^a_{YX}\star h^b_{YX})(x)=h^b_{YX}(x-\|b\|_X)=h^b_{YX}(h^b_{XY}(h^c_{YX}(x)))=h^c_{YX}(x).
$$

We conclude that $h^c_{YX}=h^a_{YX}\star h_{YX}^b$. 
\end{proof}

We will show later in Proposition~\ref{prop:cistight} that under the assumptions of Lemma~\ref{lem:XYstar} the vector $c$ is tight. 

\section{The slope decomposition}\label{sec:slope}
\subsection{unitangent vectors}
In this section we study the slope decomposition. We show that a vector $c$ is tight if and only if it has a slope decomposition. We also will show that
the Pareto frontier is always piecewise linear for a tight vector. In that case, the different slopes in the Pareto frontier correspond to different summands in the slope decomposition of $c$.

For a vector $c\in V$ we have the inequality $\|c\|_2^2=\langle c,c\rangle\leq \|c\|_X\|c\|_Y$.
\begin{definition}
We call a vector $c\in V$ {\em unitangent} if $\|c\|_2^2=\|c\|_X\|c\|_Y$.
\end{definition}
Unitangent vectors are the simplest kind of tight vectors. As we will see, their Pareto frontier is linear, i.e., has only one slope.
Recall that $\|c\|_Y$ is the maximal value of the functional $\langle c,\cdot\rangle $ on the unit ball $B_X$. Now $c$ is unitangent if and only if 
the maximum of $\langle c,\cdot\rangle$ is attained at $c/\|c\|_X\in B_X$. 
\begin{proposition}
If $c\in V$ is unitangent, then it is tight and we have $\alpha_{YX}^c(x)=(x/\|c\|_X)c$ and $f_{YX}^c(x)=\|c\|_Y(1-x/\|c\|_X)$ for $x\in [0,\|c\|_X]$.
\end{proposition}
\begin{proof}
Suppose that $\|a\|_X\leq x$. Then we have 
\begin{multline*}
\|c-a\|_Y\|c\|_X\geq \langle c-a,c\rangle=\|c\|_2^2-\langle a,c\rangle\geq \|c\|_2^2-\|a\|_X\|c\|_Y=\\=
\|c\|_Y(\|c\|_X-\|a\|_X)\geq \|c\|_Y(\|c\|_X-x).
\end{multline*}
It follows that
$$
\|c-a\|_Y\geq \|c\|_Y(1-x/\|c\|_X).
$$
Since $a$ was arbitrary, we conclude that $f_{YX}^c(x)\geq \|c\|_Y(1-x/\|c\|_X)$. 

If we take $a=(x/\|c\|_X)c$, then we have $\|a\|_X=x$ and $\|c-a\|_Y=\|c\|_Y(1-x/\|c\|_X)\leq f_{YX}^c(x)$.
We conclude that $f_{YX}^c(x)=\|c\|_Y(1-x/\|c\|_X)$ and $\alpha_{YX}^c(x)=a=(x/\|c\|_X)c$.
\end{proof}
\subsection{Faces of the unit ball}
Suppose that $B$ is a compact convex subset of a finite dimensional $\R$-vector space $V$. Recall that a convex closed subset $F$ of $B$ is a face
if $v,w\in B$, $0<t<1$ and $tv+(1-t)w\in F$ implies that $v,w\in F$. The following lemma is easily shown by induction and is left to the reader.
\begin{lemma}\label{lem:faceproperty}
If $F$ is a face of $B$, $v_1,v_2,\dots,v_r\in B$, $t_1,t_2,\dots,t_r>0$ such that $t_1+t_2+\cdot+t_r=1$ and $t_1v_1+\cdots+t_rv_r\in F$,
then $v_1,v_2,\dots,v_r\in F$.
\end{lemma}
\begin{proof}
We prove the statement by induction on $r$.
This is clear from the definition of a face for $r\leq 2$. Suppose $r>2$.
Let $t_i'=t_i/(1-t_r)$ and $w=t_1'v_1+\cdots+t_{r-1}'v_{r-1}$. We have $t_rv_r+(1-t_r)w=t_1v_1+\cdots+t_rv_r\in F$, so $v_r,w\in F$.
Since $t_1'+\cdots+t_{r-1}'=1$ and $w=t_1'v_1+\cdots+t_{r-1}'v_{r-1}\in F$ we get $v_1,\dots,v_{r-1}$ by induction.
\end{proof}
\begin{lemma}
If $B\subseteq V$ is a compact convex set, then the smallest face containing $v\in B$ is
$$
F=\{w\in B\mid \mbox{$(1+t)v-t w\in B$ for some $t>0$}\}.
$$
\end{lemma}
\begin{proof}
Suppose that $w_1,w_2\in F$ and $sw_1+(1-s)w_2\in B$ for some $s>0$. There exist $t>0$ such that $(1+t)v-tw_1,(1+t)v-tw_2\in B$,
so we have
$$
s\big((1+t)v-tw_1\big)+(1-s)\big((1+t)v-tw_2\big)=(1+t)v-t\big(sw_1+(1-s)w_2\big)\in B
$$
and therefore $sw_1+(1-s)w_2\in F$. This proves that $F$ is a face of $B$.

Suppose that $F'$ is any face of $B$ containing $v$. If $w\in F$ then there exists $t>0$ such that $(1+t)v-tw\in B$.
Since $v\in F'$ is a convex combination of $(1+t)v-tw,w\in B$, we have $w,(1+t)v-tw\in F'$. So $F'$ contains $F$.
\end{proof}

Let $B_X$ be the unit ball for the norm $\|\cdot\|_X$.

\begin{lemma}\label{lem:facialXcone}
A convex cone $C$ in $V$ is a facial $X$-cone if and only if it has the following norm-sum property: If $a,b\in V$, $a+b\in C$ and $\|a+b\|_X=\|a\|_X+\|b\|_X$
then we have $a,b\in C$.
\end{lemma}
\begin{proof}
Suppose that $C$ is a cone in $V$. If $C=\{0\}$ then $C$ is a facial $X$-cone and has the norm-sum property. 
Assume now that $C\neq \{0\}$.
Let $\partial B_X=\{c\in V\mid \|c\|_X=1\}$ be the unit sphere.
Take $F=C\cap \partial B_X$ so that $C=\R_{\geq 0}F$. We have to show that $F$ is a face of $B_X$ if and only if $C$ has the norm-sum property. 

Suppose that $F$ is a face. If $a+b\in C$ and $\|a+b\|_X=\|a\|_X+\|b\|_X$ then we have
$$
\frac{a+b}{\|a+b\|_X}=t\frac{a}{\|a\|_X}+(1-t)\frac{b}{\|b\|_X}
$$
where $t=\|a\|_X/\|a+b\|_X$ and $a/\|a\|_X,b/\|b\|_X,(a+b)/\|a+b\|_X\in F$. If $t=0$ or $t=1$ then $a=0$ or $b=0$ and $a,b\in C$. Otherwise, $0<t<1$
and $a/\|a\|_X,b/\|b\|_X\in F$ because $F$ is a face. We conclude that $a,b\in C$. So $C$ has the norm-sum property.

Conversely, suppose that $C$ has the norm-sum property, $a,b\in B_X$ and $0<t<1$ such that $t a+(1-t)b\in F$
then we have
$$
\|t a+(1-t)b\|_X=1=t+(1-t)\geq \|t a\|_X+\|(1-t)b\|_X\geq \|t a+(1-t)b\|_X
$$
and the inequalities are equalities. It follows that $\|a\|_X=\|b\|_X=1$
The norm-sum property gives $t a, (1-t)b\in C$, so $a,b\in C$. We conclude that $a,b\in C\cap \partial B_X=F$.
\end{proof}

\begin{lemma}\label{lem:smallestXcone}
Suppose that $a\in V$ is nonzero, and $C$ is the set of all
 $b\in V$ for which there exists $\varepsilon>0$ such that
$\|a-\varepsilon b\|_X+\|\varepsilon b\|_X=\|a\|_X$. Then $C$ is the smallest facial $X$-cone containing $a$.
\end{lemma}
\begin{proof}
If $a\in V$, then we have
 $\|a-\varepsilon b\|_X+\|\varepsilon b\|_X=\|a\|_X$ and every facial $X$ cone containing $a$ must also contain $\varepsilon b$ and $b$ by Lemma~\ref{lem:facialXcone}.
 Now $C$ itself is a facial $X$ cone: if $b_1,b_2\in C$ then there exists $\varepsilon_1,\varepsilon_2>0$ such that
 $\|a-\varepsilon_i b_i\|_X+\|\varepsilon_ib_i\|_X=\|a\|_X$ for $i=1,2$. We can replace $\varepsilon_1$ and $\varepsilon_2$ by the minimum of the two and assume that $\varepsilon_1=\varepsilon_2=\varepsilon$. 
 We have
 \begin{multline*}
\|a\|_X\leq\textstyle \frac{1}{2}( \|a-\varepsilon b_1\|_X+\|\varepsilon b_1\|_X+\|a-\varepsilon b_2\|_X+\|\varepsilon b_2\|_X)\leq\\
\textstyle \leq  \|a-\frac{1}{2}\varepsilon (b_1+b_2)\|_X+\|\frac{1}{2}\varepsilon (b_1+b_2)\|_X\leq \|a\|_X.
\end{multline*}
We must have equalities everywhere, so $\frac{1}{2}(b_1+b_2)$ and $b_1+b_2$ lie in $C$ by Lemma~\ref{lem:facialXcone}.
\end{proof}


\begin{definition}
For  $c\in V$ and $0<x<\|c\|_X$ we define ${\mathcal F}_{X}(x)$ as the smallest face of $B_X$ containing $\alpha_{2X}(x)/x$. 
\end{definition}
\begin{lemma}
Suppose that $c$ is a tight vector. If $0<x_1<x_2<\|c\|_X$ then we have
${\mathcal F}_{X}(x_1)\subseteq {\mathcal F}_{X}(x_2)$.
\end{lemma}
\begin{proof}
Suppose that $c$ is tight and $0<x_1<x_2<\|c\|_X$.
For $t=x_2/(x_2-x_1)$ we have
$$(1+t)\alpha_{YX}(x_2)/x_2-t\alpha_{YX}(x_1)/x_1=(\alpha_{YX}(x_2)-\alpha_{YX}(x_1))/(x_2-x_1)$$
which lies in the unit ball $B_X$. This proves that $\alpha_{YX}(x_1)/x_1$ lies in ${\mathcal F}_{X}(x_2)$.
We conclude that ${\mathcal F}_{X}(x_1)\subseteq {\mathcal F}_{X}(x_2)$.
\end{proof}
\begin{proposition}\label{prop:fpiecewise}
If $c\in V$ is tight, then $\alpha_{YX}(x)$ and $f_{YX}(x)$ are piecewise linear.
\end{proposition}
\begin{proof}
We can divide up the interval $[0,\|c\|_X]$ into finitely many intervals such that on each interval ${\mathcal F}_{X}$ is constant.
Suppose that $(x_1,x_2)$ is an open interval on which ${\mathcal F}_{X}$ equal to $F$. The affine hull of $F$ is of the form $d+W$
where  $W\subset V$ is a subspace and $d\in W^\perp$.  Now $\alpha_{YX}(x)$ is the vector in $xd+W$ closest to $c$ (in the euclidean norm). If we define $a(x)=\alpha_{YX}(x)-dx$, then $a(x)\in W$ is the vector closest to $c-dx$.
So $a(x)$ is the orthogonal projection of $c-dx$ onto $W$. Since $d\in W^{\perp}$, $a(x)$ is the orthogonal projection of $c$ onto $W$, so $a(x)=a$ is constant. This proves that  $\alpha_{YX}(x)=a+dx$ is linear.

Because $\langle c-a,a\rangle=0$, we have
$$
xh_{YX}(x)=\langle c-\alpha_{XY}(x),\alpha_{XY}(x)\rangle=\langle c-a-dx,a+dx\rangle=\langle c,d\rangle x+\langle d,d\rangle x^2.
$$

So $h_{YX}(x)$ is linear for $x\in (x_1,x_2)$.
\end{proof}
\begin{proof}[Proof of Theorem~\ref{thm:sparser}]
Suppose that $c=a+b$ is an $X2$-decomposition and let $x=\|a\|_X$ and $y=\|b\|_Y$. 
The smallest face of $B_X$ containing $x^{-1}a$ is ${\mathcal F}_X(x)$ and the smallest face containing $y^{-1}b$ is ${\mathcal F}_Y(y)$. For every $a'\in {\mathcal F}_X(x)$ and every $b'\in {\mathcal F}_Y(y)$ we have
$\langle a',b'\rangle \leq 1$. We have $\langle x^{-1}a,y^{-1}b\rangle=1$. Since $y^{-1}b$ lies in the relative interior of ${\mathcal F}_Y(y)$, we have $\langle x^{-1}a,b'\rangle=1$ for all $b'\in {\mathcal F}_Y(y)$. 
Since $x^{-1}a$ lies in the relative interior of ${\mathcal F}_X(x)$, we have $\langle a',b'\rangle=1$
for all $a'\in{ \mathcal F}_X(x)$ and all $b'\in {\mathcal F}_Y(y)$. 
It follows that 
$$
\gsparse_X(a)+\gsparse_Y(b)=\dim {\mathcal F}_X(x)+1+\dim {\mathcal F}_Y(y)+1\leq n+1.
$$

If $c$ is tight, them we have ${\mathcal F}_X(x)\subseteq {\mathcal F}_X(\|c\|_X)$ and 
$$
\gsparse_X(a)=\dim {\mathcal F}_X(x)+1\leq \dim {\mathcal F}_X(\|c\|_X)+1=\gsparse_X(c).
$$
\end{proof}

\subsection{Proof of Theorem~\ref{theo:tightslope}}
Suppose that $c$ is tight. Then $h^{c}_{YX}=f^c_{YX}$ is piecewise linear by Proposition~\ref{prop:fpiecewise}. Suppose that $z_0<z_1<\cdots<z_r=\|c\|_X$
such that $h^{c}_{YX}$ is linear on each interval $[z_{i-1},z_i]$, and that  $h^c_{YX}$ is not differentiable at $z_1,\dots,z_{r-1}$.
Let $a_i=\alpha_{2X}^c(z_i)$, and define $c_i=a_i-a_{i-1}$ for $i=1,2,\dots,r$.
We have $a_{i-1}=\alpha_{2 X}^c(z_{i-1})=\alpha_{2 X}^{a_i}(z_{i-1})$ so $a_i=a_{i-1}+c_i$ is an $X2$-decomposition for all $i$.
By induction we get
$$
f^c_{YX}=f^{c_1}_{YX}\star \cdots \star f^{c_r}_{YX}.
$$
The area under the graph of $f^{c_i}_{YX}=h^{c_i}_{YX}$  is ${\textstyle \frac{1}{2}}\|c_i\|_2^2$.
The area under the graph of $f^{c_1}_{YX}\star \cdots \star f^{c_r}_{YX}$ is
$$
\sum_{i<j}\|c_i\|_X\|c_j\|_Y+{\textstyle \frac{1}{2}}\sum_{i}\|c_i\|_X\|c_i\|_Y.
$$
So we have
$$
\sum_{i<j}\langle c_i,c_j\rangle+{\textstyle \frac{1}{2}}\sum_i \langle c_i,c_i\rangle={\textstyle \frac{1}{2}}\|c\|_2^2=
\sum_{i<j}\|c_i\|_X\|c_j\|_Y+{\textstyle \frac{1}{2}}\|c_i\|_X\|c_i\|_Y\geq
\sum_{i<j}\langle c_i,c_j\rangle+{\textstyle \frac{1}{2}}\sum_i\langle c_i,c_i\rangle.
$$
This proves that $\langle c_i,c_j\rangle =\|c_i\|_X\|c_j\|_Y$ for all $i\leq j$.
This shows that $c=c_1+c_2+\cdots+c_r$ is a slope decomposition.

Conversely, suppose that $c=c_1+c_2+\cdots+c_r$ is a slope decomposition. 
We will show that $c$ is tight. Since $c_1+\cdots+c_{r-1}$ is also a slope decomposition, so by induction we may assume that
$c_1+\cdots+c_{r-1}$ is tight, and 
$$
h^{c_1+\cdots+c_{r-1}}_{YX}=f^{c_1+\cdots+c_{r-1}}_{YX}=f^{c_1}_{YX}\star \cdots\star f^{c_{r-1}}_{YX}=h^{c_1}_{YX}\star \cdots \star h^{c_{r-1}}_{YX}.
$$
Since $c=(c_1+\cdots+c_{r-1})+c_{r}$ is an $X2$-decomposition, it follows from Lemma~\ref{lem:XYstar} that
$$
h^{c}_{YX}=h^{c_1+\cdots+c_{r-1}}_{YX}\star h^{c_r}_{YX}=h^{c_1}_{YX}\star \cdots\star h^{c_r}_{YX}.
$$
Suppose that
$$\|c_1\|_X+\cdots+\|c_{i-1}\|_X\leq x\leq \|c_1\|_X+\cdots+\|c_i\|_X.$$
We have
\begin{multline*}
\sum_{j<i} \|c_j\|_X\|c_i\|_Y+\sum_{j\geq i} \|c_j\|_Y\|c_i\|_X=\sum_{j<i} \langle c_j,c_i\rangle+\sum_{j\geq i}\langle c_j,c_i\rangle  =\\=\langle c-a,c_i\rangle+\langle a,c_i\rangle\leq \|c-a\|_Y\|c_i\|_X+\|a\|_X\|c_i\|_Y=f_{YX}(x)\|c_i\|_X+x\|c_i\|_Y.
\end{multline*}
So we have
$$
f_{YX}^c(x)\geq \sum_{j=i}^r \|c_j\|_Y+\frac{\|c_i\|_Y}{\|c_i\|_X}\Big(\sum_{j=1}^{i-1} \|c_j\|_X-x\Big)=(f_{YX}^{c_1}\star f_{YX}^{c_2}\star \cdots\star f_{YX}^{c_r})(x).
$$
It follows that 
$$
f^c_{YX}\leq h^c_{YX}=h^{c_1}_{YX}\star \cdots\star h^{c_r}_{YX}=f^{c_1}_{YX}\star\cdots\star f^{c_r}_{YX}\leq f^{c}_{YX}.
$$
We get  $f^c_{YX}=h^c_{YX}$, so $c$ is tight. We have proven part (1).

(2) Suppose that $c$ is tight and $c=c_1+c_2+\cdots+c_r$ is a slope decomposition.
Then the graph $f^{c_i}_{YX}$ is a straight line segment from $(0,y_i)$ to $(x_i,0)$ where $y_i=\|c_i\|_Y$ and
$x_i=\|c_i\|_X$. Since $f^c_{YX}=f^{c_1}_{YX}\star f^{c_2}_{YX}\star \cdots \star f^{c_r}_{YX}$, we have that
$f^c_{YX}$ is the graph through the points $(x_1+\cdots+x_i,y_{i+1}+\cdots+y_r)$ for $i=0,1,2,\dots,r$.
\subsection{Properties of the slope decomposition}
\begin{lemma}\label{lem:abctight}
If $c=c_1+\cdots+c_m$ is an $XY$-slope decomposition, then $c_1,\dots,c_m$ are linearly independent.
\end{lemma}
\begin{proof}
Suppose that we have
$$
c_m=\sum_{i=1}^{m-1} \lambda_i c_i.
$$
Because $\mu_{XY}(c_i)=\|c_i\|_Y/\|c_i\|_X>\mu_{XY}(c_m)=\|c_m\|_Y/\|c_m\|_X$ we get
\begin{multline*}
\|c_m\|_X\|c_m\|_Y=\langle c_m,c_m\rangle=\sum_{i=1}^{m-1}\lambda_i\langle c_i,c_m\rangle\leq \\ \leq \sum_{i=1}^{m-1}
|\lambda_i|\|c_i\|_X\|c_m\|_Y<\sum_{i=1}^{m-1}|\lambda_i|\|c_i\|_Y\|c_m\|_X\leq \|c_m\|_X\|c_m\|_Y
%
\end{multline*}
Contradiction. This proves that $c_1,\dots,c_r$ are linearly independent.
\end{proof}
\begin{proposition}\label{prop:cistight}
Suppose that $c=a+b$ is an $X2$-decomposition and $a$ and $b$ are tight. Then $c$ is tight. 
\end{proposition}
\begin{proof}
Since $a$ and $b$ are tight, they have slope decompositions, say $a=a_1+\cdots+a_r$ and $b=b_1+\cdots+b_s$.
We have
\begin{multline*}
\|a\|_X\|b\|_Y=\langle a,b\rangle=\sum_{i=1}^r\sum_{j=1}^s \langle a_i,b_j\rangle\leq
\sum_{i=1}^r\sum_{j=1}^s \|a_i\|_X\|b_j\|_Y=\\=\left(\sum_{i=1}^r \|a_i\|_X\right)\left(\sum_{j=1}^s \|b_j\|_Y\right)\leq
\|a\|_X\|b\|_Y.
\end{multline*}
It follows that $\langle a_i,b_j\rangle=\|a_i\|_X\|b_j\|_Y$ for all $i<j$.
If $\mu_{XY}(a_r)>\mu_{XY}(b_1)$ then
$$
c=a_1+\cdots+a_r+b_1+\cdots+b_s
$$
is a slope decomposition.

Suppose that $\mu_{XY}(a_r)\leq \mu_{XY}(b_1)$. We have
\begin{multline*}
\|a_r\|_Y\|a_r+b_1\|_X\geq \langle a_r,a_r+b_1\rangle=\\=\langle a_r,a_r\rangle+\langle a_r,b_1\rangle=
\|a_r\|_X\|a_r\|_Y+\|a_r\|_X\|b_1\|_Y\geq \|a_r\|_X\|a_r+b_1\|_Y
\end{multline*}
so $\mu_{XY}(a_r)\geq \mu_{XY}(a_r+b_1)$. Similarly, we have
\begin{multline*}
\|a_r+b_1\|_Y\|b_1\|_X\geq \langle a_r+b_1,b_1\rangle=\\=\langle a_r,b_1\rangle+\langle b_1,b_1\rangle=
\|a_r\|_X\|b_1\|_Y+\|b_1\|_X\|b_1\|_Y\geq \|a_r+b_1\|_X\|b_1\|_Y,
\end{multline*}
so $\mu_{XY}(a_r+b_1)\geq \mu_{XY}(b_1)\geq \mu_{XY}(a_r)\geq \mu_{XY}(a_r+b_1)$.
We conclude that $\mu_{XY}(a_r)=\mu_{XY}(b_1)=\mu_{XY}(a_r+b_1)$ and
$$
c=a_1+\cdots+a_{r-1}+(a_r+b_1)+b_2+\cdots+b_s
$$
is a slope decomposition.

Since $c$ has a slope decomposition, it is tight.
\end{proof}

\subsection{The unit ball of a tight norm}
\begin{proposition}\label{prop:normtight}
A norm $\|\cdot\|_X$ is tight if and only if every face $F$ of the unit ball $B_X$ contains a unitangent vector that is perpendicular to $F$. 
\end{proposition}
\begin{proof}
Suppose that $\|\cdot\|_X$ is tight and that $F$ is a face of $B_X$ (other than $B_X$ itself). Choose $c\in F$ in the relative interior.
Then we have a slope decomposition
$$
c=c_1+c_2+\cdots+c_r.
$$
If $t_i=\|c_i\|_X/\|c\|_X$ then we have $t_1+t_2+\cdots+t_r=1$ and
$$
\frac{c}{\|c\|_X}=t_1\frac{c_1}{\|c_1\|_X}+t_2\frac{c_2}{\|c_2\|_X}+\cdots+\cdots+t_r\frac{c_r}{\|c_r\|_X}.
$$
Now $c_1,\dots,c_r\in F$ by Lemma~\ref{lem:faceproperty}. We have $\langle c_r,c\rangle=\|c\|_X\|c_r\|_Y=\|c_r\|_Y$. For any other vector $b\in B_X$ we have $\langle c_r,b\rangle\leq\|b\|_X\|c_r\|_Y
\leq \|c_r\|_Y$. So  the functional $ \langle c_r,\cdot\rangle$ on F is maximal at $c$, and therefore maximal  and constant on the face $F$.
It follows that $c_r$ is perpendicular to $F$.

Now we show the converse. Suppose that every face $F$ of the unit ball $B_X$ contains a unitangent vector that is perpendicular to $F$. Suppose that $c\in V$ is  a vector with $\|c\|_X=1$. Let $F$ be the smallest face of $B_X$ that contains $c/\|c\|_X$. By induction on $\dim F$ we show that $c$ is tight.
The case $\dim F=0$ is clear. Suppose that $\dim F>0$. There exists a vector $b\in F$ that is unitangent and orthogonal to $c$.  Choose $t$ maximal such that $v=c+t(c-b)=(1+t)c-tb\in F$. Let $a'=\frac{1}{1+t}v$ and $b'=\frac{t}{t+1}b$.
Then we have $c=a'+b'$. Since $v$ lies in a face of smaller dimension, we know by induction that $v$ and $a'$ are tight.
Because $b'$ is unitangent, it is also tight. We have
$$
\|a'\|_X=\frac{\|v\|_X}{t+1}=\frac{1}{t+1}=\frac{\|b\|_X}{t+1}.
$$
$$
\|b'\|_Y=\frac{t\|b\|_Y}{t+1}.
$$
$$
\langle a',b'\rangle=\frac{t\langle v,b\rangle}{(t+1)^2}=\frac{t\langle b,b\rangle}{(t+1)^2}.
$$
So $\langle a',b'\rangle=\|a'\|_X\|b'\|_Y$.
It follows that $c=a'+b'$ is an $X2$-decomposition. By Lemma~\ref{lem:abctight}, $c$ is tight.
\end{proof}
\begin{example}
Consider again Examples~\ref{ex:1.10} and~\ref{ex:l1linf}. Suppose that $c=(c_1,\dots,c_n)^t\in \R^n$ and define $\lambda_1>\lambda_2>\cdots>\lambda_r>0$ by
$$
\{|c_1|,|c_2|,\dots,|c_n|\}\setminus \{0\}=\{\lambda_1,\dots,\lambda_r\}.
$$
Define $m_i$ be the multiplicity of $\lambda_i$, i.e., $m_i$ is the number of values of $j$ for which $|c_j|=\lambda_i$.
We define vectors $c^{(1)},\dots,c^{(r)}\in \R^n$ as follows:
$$
c^{(i)}=(c^{(i)}_1,\dots,c^{(i)}_n)
$$
and
$$
c^{(i)}_j=\begin{cases}
\sgn(c_j)(\lambda_i-\lambda_{i+1}) & \mbox{if $|c_j|\geq \lambda_i$ and}\\
0 & \mbox{if $|c_j|<\lambda_i$.}
\end{cases}
$$
We use the convention $\lambda_{r+1}=0$. We have
\begin{equation}\label{eq:cslope}
c=c^{(1)}+\cdots+c^{(r)}
\end{equation}
We have
$
\|c^{(i)}\|_\infty=\lambda_i-\lambda_{i+1}
$. 
and
$$
\|c^{(i)}\|_{1}=(m_1+m_{2}+\cdots+m_i)(\lambda_i-\lambda_{i+1}).
$$
If $i<j$ then we have
$$
\langle c^{(i)}, c^{(j)}\rangle =(m_1+m_{2}+\cdots+m_j)(\lambda_i-\lambda_{i+1})(\lambda_j-\lambda_{j+1})=\|c^{(i)}\|_\infty \|c^{(i)}\|_{1}.
$$
We have
$$
\mu_{1\infty}(c^{(i)})=\frac{\|c^{(i)}\|_\infty}{\|c^{(i)}\|_1}=\frac{1}{m_1+m_2+\cdots+m_i}.
$$
This shows that (\ref{eq:cslope}) is a slope decomposition. So the norms $\|\cdot\|_\infty$ and $\|\cdot\|_1$ are tight. 
\end{example}

\section{The Singular Value Decomposition of a matrix}\label{sec:matrixSVD}
\subsection{Matrix norms}
In this section, we will study the singular value decomposition of a matrix using our terminology and the results we have obtained.  
Let $V=\C^{m\times n}$ be the space of $m\times n$-matrices. We have a bilinear form $\langle \cdot,\cdot\rangle$ on $V$ defined
by
$$
\langle A,B \rangle =\Re(\trace(AB^\star))=\Re(\trace(B^\star A)),
$$
where $B^\star$ denotes the conjugate transpose of $B$ and $\Re(\cdot)$ denotes the real part. We will study  3 norms on the vector space $V$ namely the {\em euclidean norm}, the {\em nuclear norm}
and the {\em spectral norm}, and express each of these in terms of the singular values of a matrix.

The matrix $A^\star A$ is a nonnegative definite Hermitian matrix and its
eigenvalues are nonnegative and real.
The euclidean $\ell_2$-norm of a matrix is given by
$$
\|A\|_2=\sqrt{\langle A,A\rangle}=\sqrt{\Re(\trace(A^\star A))}=\sqrt{\trace(A^\star A)}.
$$
Since $A^\star A$ is positive semi-definite Hermitian, we can choose a unitary matrix $U$ such that
$U^\star A^\star A U$ is a diagonal matrix with diagonal entries $\lambda_1^2,\lambda_2^2,\dots,\lambda_n^2$ where $\lambda_1\geq \lambda_2\geq \cdots\geq \lambda_n\geq 0$ are the {\em singular values of $A$}.  We have
$$
\|A\|_2=\sqrt{\trace(A^\star A)}=\sqrt{\trace(U^\star A^\star A U)}=\sqrt{\lambda_1^2+\cdots+\lambda_n^2}.
$$

Let $D$ be the diagonal matrix
$$
\begin{pmatrix} \lambda_1 & & &\\
& \lambda_2 & &\\
&&\ddots &\\
&&& \lambda_n\end{pmatrix}
$$
Then $UDU^\star$ is the unique positive semi-definite Hermitian matrix whose square is $A^\star A$, and we will denote this matrix by $\sqrt{A^\star A}$.

 We define the {\em spectral norm} or {\em operator norm} $\|A\|_\sigma$ of $A$  by
$$
\|A\|_Y=\|A\|_\sigma=\max\{\|Av\|_2\mid v\in \C^n\mbox{ and } \|v\|_2=1\}.
$$
We have
$$
\|A\|_\sigma=\left\|\begin{pmatrix}
\lambda_1 &&\\
& \ddots &\\
&&\lambda_n\end{pmatrix}\right\|_\sigma=\lambda_1.
$$

The {\em nuclear norm} $\|\cdot\|_\star$ is defined by
$$
\|A\|_X=\|A\|_\star=\trace(\sqrt{A^\star A})=\lambda_1+\lambda_2+\cdots+\lambda_n.
$$
The proof of the following well-known result will be useful for the discussion that follows.
\begin{lemma}\label{lem:sigmadual}
The norms $\|\cdot\|_\sigma$ and $\|\cdot\|_*$ are dual to each other.
\end{lemma}
\begin{proof}
Let $A\in V$. We use the notation as before, where $\sqrt{A^\star A}=UDU^\star$, and $D$ is the diagonal matrix whose diagonal entries are the singular values $\lambda_1,\lambda_2,\dots,\lambda_n$. Let $u_1,u_2,\dots,u_n$ be the columns of $U$. These vectors form an orthonormal basis and for any $B\in V$ we have
\begin{multline}\label{eq:dual1}
\langle A,B\rangle=\Re( \trace(B^\star A))=\Re (\trace(U^\star B^\star A U))=\Re\Big(\sum_{i=1}^n \langle Au_i,Bu_i\rangle\Big)\leq \\ \leq\sum_{i=1}^n \|Au_i\|_2\|Bu_i\|_2=\sum_{i=1}^n \lambda_i\|B\|_\sigma=\|A\|_\star \|B\|_\sigma.
\end{multline}
Because $(AU)^\star(AU)=D^2$, the columns of $AU$ are orthogonal. The matrices $W=AUD^{-1}$ is unitary.
So $A=WDU^\star$. If $w_1,\dots,w_n$ are the columns of $W$, then the singular value decomposition of $A$ is
$$
\sum_{i=1}^n \lambda_iw_iu_i^\star.
$$
Let $r$ be maximal such that $\lambda_r\neq 0$ and define the block matrix
$$
E=\begin{pmatrix} I_r & 0\\
0 & 0\end{pmatrix}
$$
For  $B=WEU^\star$ we have 
\begin{multline}\label{eq:dual2}
\langle A,B\rangle=\Re(\trace(B^\star A))=\Re(\trace(UEW^\star A))=\Re(\trace(EW^\star A U))=\\=
\Re(\trace(ED))
=\Re(\trace(D))=\lambda_1+\cdots+\lambda_r=\|A\|_\star.
\end{multline}
From (\ref{eq:dual1}) and (\ref{eq:dual2}) follows  that $\|\cdot\|_\star$ is the dual norm of $\|\cdot\|_\sigma$.
\end{proof}
\subsection{Slope decomposition for matrices}
Suppose that $C$ is a complex $m\times n$ matrix and let $\lambda_1>\lambda_2>\dots>\lambda_r>0$ be the nonsingular values of $C$ with multiplicities $m_1,m_2,\dots,m_r$ respectively.
We can write $C=WDU^\star$ where $U,W$ are unitary and $D$ is of the form
$$
\left(\begin{array}{cccc|c}
\lambda_1 I_{m_1} & & & &0 \\
 & \lambda_2I_{m_2} & & & 0 \\
 & & \ddots &&  \vdots \\
 &&& \lambda_r I_{m_r} & 0 \\ \hline
0 &0 &\cdots &0 & 0
 \end{array}\right)
 $$
 where $I_d$ is the $d\times d$ identity matrix and the zeros represent possible empty zero blocks.  The norms can be expressed as follows:
\begin{eqnarray*}
\|C\|_\star& =&\textstyle \sum_{i=1}^r m_i\lambda_i\\
\|C\|_\sigma & = & \lambda_1\\
\|C\|_2 & = &\textstyle \sqrt{\sum_{i=1}^r m_i\lambda_i^2}
\end{eqnarray*}
 Define
$$
C^{(j)}=(\lambda_j-\lambda_{j+1})W\begin{pmatrix} I_{k_j} & 0\\
0 & 0\end{pmatrix}U^\star
$$
for $j=1,2,\dots,r$ where $k_j=m_1+m_2+\cdots+m_j$ and  $\lambda_{r+1}=0$.
We have
\begin{equation}\label{eq:Aslope}
C=C^{(1)}+\cdots+C^{(r)}.
\end{equation}
\begin{proposition}
The expression (\ref{eq:Aslope}) is a slope decomposition. 
In particular, the spectral and the nuclear norms are tight.
\end{proposition}
\begin{proof}
We have
\begin{eqnarray*}
\|C^{(j)}\|_Y=\|A^{(j)}\|_\sigma &=& \lambda_j-\lambda_{j+1}\\
\|C^{(j)}\|_X=\|C^{(j)}\|_\star&=& k_j(\lambda_j-\lambda_{j+1})\\
\mu_{XY}(C^{(j)})&=&\frac{1}{k_j}.
\end{eqnarray*}
In particular, $\mu_{XY}(C^{(j)})$ is strictly decreasing as $j$ increases.

If $i\leq j$ then we have
\begin{multline*}
\langle C^{(i)},C^{(j)}\rangle=\Re\Big(\trace\big((C^{(i)})^\star C^{(j)}\big)\Big)=\\
=(\lambda_i-\lambda_{i+1})(\lambda_j-\lambda_{j+1})\Re\Big(\trace\Big(U\begin{pmatrix}I_{k_i} & 0\\0 & 0\end{pmatrix}
W^\star W\begin{pmatrix} I_{k_j} & 0\\0 & 0\end{pmatrix}U^\star\Big)\Big)=\\
=(\lambda_i-\lambda_{i+1})(\lambda_j-\lambda_{j+1})\Re\Big(\trace\Big(U\begin{pmatrix}I_{k_i}& 0\\0 & 0\end{pmatrix}
\begin{pmatrix} I_{k_j}& 0\\0 & 0\end{pmatrix}U^\star\Big)\Big)=\\
=(\lambda_i-\lambda_{i+1})(\lambda_j-\lambda_{j+1})\Re\Big(\trace\Big(U\begin{pmatrix}I_{k_i} & 0\\0 & 0\end{pmatrix}U^\star\Big)\Big)=\\
=k_i(\lambda_i-\lambda_{i+1})(\lambda_j-\lambda_{j+1})=\|C^{(i)}\|_{\star}\|C^{(j)}\|_\sigma.
\end{multline*}
This proves that  (\ref{eq:Aslope})  is the slope decomposition. 
\end{proof}
\subsection{Principal Component Analysis}
In Principal Component Analysis (PCA), one finds a low rank matrix that approximates the matrix $A$ by truncating the singular value decomposition.  For a given threshold $y$, let $s$ be maximal such that $\lambda_s>y$. 
Then
$$
C'=W\left(\begin{array}{cccc|c}
\lambda_1 I_{m_1} & & & &0 \\
 & \lambda_2I_{m_2} & & & 0 \\
 & & \ddots &&  \vdots \\
 &&& \lambda_r I_{m_r} & 0 \\ \hline
0 &0 &\cdots &0 & 0
 \end{array}\right)U^\star
 $$
is a low rank approximation of $C$.
 This method is called {\em hard threshholding}. Replacing $C$ by the approximation $C'$ is an effective way to reduce the dimension of a large scale problem. Let us compare this to the $XY$-decomposition (or equivalently $X2$- or $2Y$-decomposition) of $C$. Let us define 
$$
A=W\left(\begin{array}{cccc|c}
(\lambda_1-y) I_{m_1} & & & & 0 \\
 &( \lambda_2-y) I_{m_2} & & & 0 \\
 & & \ddots &&  \vdots \\
 &&& (\lambda_s-y) I_{m_s} & 0 \\ \hline
0 &0 &\cdots &0 & 0
 \end{array}\right)U^\star
 $$
 and 
 $$
 B=C-A=
 W\left(\begin{array}{cccccc|c}
yI_{m_1} & & & & & &  0 \\
 &\ddots & & & & & \vdots \\
 & & yI_{m_s} && &  & 0 \\
 &&& \lambda_{s+1} I_{m_{s+1}} &  && 0 \\
 &&&&\ddots & & \vdots \\
 &&&& & \lambda_r I_{m_r} & 0\\ \hline
0 & \cdots & 0 &0 & \cdots & 0 & 0
 \end{array}\right)U^\star
 $$

\begin{lemma}
The expression $C=A+B$ is an $XY$-decomposition.
\end{lemma}
\begin{proof}
We have
 $\|A\|_*=\|A\|_X=\sum_{i=1}^s m_i(\lambda_i-y)$ and $\|B\|_\sigma=\|B\|_Y=y$.
Now the lemma follows from
$$
\langle A,B\rangle=\sum_{i=1}^sm_i(\lambda_i-y)y=\|A\|_\star\|B\|_\sigma=\|A\|_X\|B\|_Y.
$$
\end{proof}
In particular, $A=\shrink_Y(C,y)$. The operator $\shrink_Y(\cdot,y)$ is soft-threshholding with threshold level $y$ (see~\cite{Donoho}).
Unlike hard thresholding, soft thresholding is continuous.
The Pareto frontier  (which is also the  sub-frontier) is given by
$$
f_{XY}^C(y)=h_{XY}^C(y)=m_i\sum_{i=1}^rm_i\max\{\lambda_i-y,0\}.
$$

\subsection{The singular value region for matrices}
The Pareto frontier of $C$ encodes the singular values. We will describe in detail how to obtain the singular values from the Pareto frontier.
Note that we consider $x=h_{XY}(y)$ as a function of $y$, rather then considering its inverse function $y=h_{YX}(x)$.
If we differentiate $h_{XY}^C(y)$ with respect to $y$ we get
$$
(h_{XY}^C(y))' = -\sum_{i=1}^s m_i
$$
if $\lambda_s>y>\lambda_{s+1}$  and $1\leq s\leq r$ (with the convention that $\lambda_{r+1}=0$).

If we plot $y$ against $x=-(h_{XY}^C(y))'$ then we get the singular value region. From the descriptions above, it is now clear
that this region can be described as an $m_1\times \lambda_1$ bar, followed by an $m_2\times \lambda_2$ bar etc. So the singular value region
from Definition~\ref{def:SVregion} is the same as the singular value region for matrices as described in the introduction.

\subsection{Rank minimization and low rank matrix completion}
Let $V=\C^{m\times n}$ be the set of $m\times n$ matrices. We will study the low rank matrix completion from the viewpoint of competing dual norms.
\begin{problem}[Low Rank Matrix Completion (LRMC)]
Given a $m\times n$ matrix $C$ where the entries $(i_1,j_1),\dots,(i_s,j_s)$ are missing. Fill in the missing entries such that the resulting matrix 
has minimal rank.
\end{problem}
The low rank matrix completion problem has applications in collaborative filtering and recommender systems such as the Netflix problem.
The low rank matrix completion problem is a special case of the rank minimization problem.
\begin{problem}[Rank Minimization (RM)]
Suppose that $W$ is a subspace of $V$, and let $A\in V$. 
Find a matrix $B\in A+W$ of minimal rank. 
\end{problem}
The Low Rank Matrix Completion problem can be formulated as a rank minimization problem as follows. Complete $C$ to a matrix $A$ in some way (for example, set all the missing entries equal to $0$). Then, Let $W\subseteq V$ be the subspace spanned by all matrices $e_{i_k,j_k}$, $k=1,2,\dots,s$.
Here $e_{p,q}$ is the matrix with all $0$'s except for a 1 in position $(p,q)$. Find $B\in A+W$ with minimal rank using RM. Then $B$ is also   the solution to the LRMC problem.

Let us consider the Rank Minimization problem. Using the philosophy of convex relaxation, we consider the following problem instead (see~\cite{CR,CT2,RFP}):
\begin{problem}\label{problem:minnuclear}
Find a matrix $D\in C+W$ with $\|D\|_\star$ minimal.
\end{problem}
Let $Z$ be the orthogonal complement of $W$ and let $\pi_{Z}$ be the orthogonal projection onto $Z$. 
The problem does not change when we replace $C$ by $\pi_Z(C)$, so  we may assume that $C\in Z$ without loss of generality.
We define a  norm $\|\cdot\|_{X}$ on $Z$ by 
$$
\|C\|_{X}=\min\{\|D\|_\star \mid D\in C+W\}.
$$
So Problem~\ref{problem:minnuclear} is essentially the problem of determining the value of $\|C\|_X$. 
In the presence of noise, we would like to find a matrix $A$ such that  $\|A\|_X$ and $\|C-A\|_2$ are small. 
This leads to the following optimization problem.
\begin{problem}\label{problem:matrixcomp}
For a fixed parameter $\lambda$, minimize
$$
{\textstyle \frac{1}{2}}\|C-A\|_2^2+\lambda\|A\|_X
$$
\end{problem}
We can write $A=\pi_Z(D)$ such that $\|A\|_X=\|D\|_\star$. We can reformulate the problem as:
\begin{problem}\label{problem:matrixcomp2}
For a fixed parameter $\lambda$, minimize
$$
{\textstyle \frac{1}{2}}\|C-\pi_Z(D)\|_2^2+\lambda\|D\|_\star
$$
\end{problem}

 The dual norm to $\|\cdot\|_X$ is defined by
$$
\|B\|_Y=\|\pi_Z(B)\|_\sigma.
$$
Problem~\ref{problem:matrixcomp} is equivalent to
\begin{problem}
Minimize $\|A\|_2$ under the constraint $\|C-\pi_Z(A)\|_\sigma\leq \lambda$.
\end{problem}

\section{Restricting Norms}
Suppose that $V$ is a finite dimensional $\R$-vector space with a positive definite bilinear form $\langle\cdot,\cdot\rangle$, and $\|\cdot\|_X$ is a norm on $V$. For a subspace $W$ of $V$, it is natural to ask whether the $X2$-decompositions of vectors in $W$ are always within the space $W$.
In this section we will give a sufficient criterion for $W$ to have this property.
\begin{definition}
A subspace $W\subseteq V$ is called a {\em nice slice} if we have $\|\pi_W(c)\|_X\leq \|c\|_X$ for all $c\in V$, where $\pi_W:V\to W$ is the orthogonal projection.
\end{definition}
\begin{lemma}
If $W$ is a nice slice, then we also have $\|\pi_W(c)\|_Y\leq \|c\|_Y$ for all $c\in V$, where $\|\cdot\|_Y$ is the norm dual to $\|\cdot\|_X$.
\end{lemma}
\begin{proof}
Choose a vector $d$ with $\|d\|_X=1$ and
$\langle d,\pi_W(c)\rangle=\|\pi_W(c)\|_Y$.
We have
$$
\|\pi_W(c)\|_Y=\langle d,\pi_W(c)\rangle=\langle \pi_W(d),c\rangle\leq \|\pi_W(d)\|_X\|c\|_Y\leq \|d\|_X\|c\|_Y=\|c\|_Y.
$$
\end{proof}
Let $\OO(V)$ be the orthogonal group consisting of all $g\in \GL(V)$ with the property 
$$\langle g\cdot v,g\cdot w\rangle=\langle v,w\rangle$$
for all $v,w\in V$.
\begin{lemma}
Suppose that $G\subseteq \OO(V)$ is a subgroup with the properties
$\|g\cdot v\|_X=\| v\|_X$
for all $v,w\in V$ and all $g\in G$. Then the space $V^G=\{v\in V\mid \forall g\in G\ g\cdot v=v\}$ of $G$-invariant vectors is a nice slice.
\end{lemma}
\begin{proof}
We can replace $G$ with its closure, so without loss of generality we may assume that $G$ is a compact Lie group.
Let $\pi_W$ be the projection onto $W:=V^G$. We have
$$
\pi_W(v)=\int_{g\in G} g\cdot v\,d\mu
$$
where $d\mu$ is the normalized Haar measure. This shows that $\pi_W(v)$ lies in the convex hull of all $g\cdot v$, $g\in G$. Since $\|g\cdot v\|_X\leq \|v\|_X$ for all $v\in V$, we also have $\|\pi_W(v)\|_X\leq \|v\|_X$ for all $g\in G$.
\end{proof}
Suppose that $W\subseteq V$ is a nice slice. Let $\|\cdot\|_{X'}$ and $\|\cdot\|_{Y'}$ be the restrictions of the norms $\|\cdot\|_X$ and $\|\cdot\|_Y$ to $W$.
\begin{lemma}
The norms $\|\cdot\|_{X'}$ and $\|\cdot\|_{Y'}$ are also  dual to each other.
\end{lemma}
\begin{proof}
If $v,u\in  W$ and $\|u\|_Y=1$ then we have
$$\langle v,u\rangle\leq \|v\|_X\|u\|_Y=\|v\|_{X'}.
$$
For a given $v\in W$, there exists $w\in V$ with $\|w\|_Y=1$ and 
$$
\langle v,w\rangle=\|v\|_X=\|v\|_{X'}.
$$
We get
$$
\langle v,\pi_W(w)\rangle=\langle \pi_W(v),w\rangle=\langle v,w\rangle=\|v\|_{X'}
$$
Define $u=\pi_W(w)/\|\pi_W(w)\|_{Y'}$. We have $\|u\|_{Y'}=1$.
 It follows that
$$
\|v\|_{X'}\geq \langle v,u\rangle=\frac{\langle v,\pi_W(w)\rangle}{\|\pi_W(w)\|_{Y'}}= \frac{\|v\|_{X'}}{\|\pi_W(w)\|_{Y'}}\geq\|v\|_{X'},
$$
because $\|\pi_W(w)\|_{Y'}=\|\pi_W(w)\|_Y\leq \|w\|_Y=1$.
So $\langle v,u\rangle=\|v\|_{X'}$. 
\end{proof}

\begin{lemma}
Suppose that $c\in W$.
\begin{enumerate}
\item 
If $c=a+b$ is an $X2$-decomposition, then $a,b\in W$ and $c=a+b$ is an $X'2$ decomposition.
\item If $c=a+b$ is an $XY$ decomposition then $c=\pi_W(a)+\pi_W(b)$ is an $X'Y'$ decomposition,
$\|\pi_W(a)\|_{X'}=\|a\|_X$ and $\|\pi_W(b)\|_{Y'}=\|b\|_Y$.
\end{enumerate}
\end{lemma}
\begin{proof}
(1) If $c=a+b$ is an $X2$ demposition, then we have 
$$c=\pi_W(c)=\pi_W(a)+\pi_W(b),$$ 
$\|\pi_W(a)\|_X\leq \|a\|_X$ and $\|\pi_W(b)\|_2\leq \|b\|_2$. It follows that $c=\pi_W(a)+\pi_W(b)$
is also an $X2$-decomposition and by uniqueness we have $\pi_W(a)=a$ and $\pi_W(b)=b$.
Suppose that $c=a'+b'$ with $a',b'\in W$ and $\|b'\|_2=\|b\|_2$.
We get 
$$
\|a'\|_{X'}=\|a'\|_X\geq \|a\|_X=\|a\|_{X'}
$$
because $c=a+b$ is an $X2$-decomposition. This shows that $c=a+b$ is an $X'2$-decomposition.

(2) Suppose that $c=a+b$ is an $XY$-decomposition. We get
$$
c=\pi_W(c)=\pi_W(a)+\pi_W(b),
$$
$\|\pi_W(a)\|_X\leq \|a\|_X$ and $\|\pi_W(b)\|_Y\leq \|b\|_Y$. It follows that 
$c=\pi_W(a)+\pi_W(b)$ is also an $XY$ decomposition.  A similar argument as in (1) shows that this is also an $X'Y'$-decomposition. 
Since $\|\pi_W(a)\|_X\leq \|a\|_X$, we must have $\|b\|_Y\geq\|\pi_W(b)\|_Y\geq \|b\|_Y$, because $c=a+b$
is an $XY$-decomposition. It follows that 
$$\|\pi_W(b)\|_{Y'}=\|\pi_W(b)\|_{Y}=\|b\|_Y.$$
Similarly, we get $\|\pi_W(a)\|_{X'}=\|a\|_X$.
\end{proof}
In particular, we have $f^c_{YX}=f^c_{Y'X'}$ and $h^c_{YX}=h^c_{Y'X'}$.
\begin{example}
Suppose that $V=\C^{m\times n}$ and $W=\R^{m\times n}\subseteq V$. The orthogonal projection $\pi_W:V\to W$ is
given by $\pi_W(A)=\Re(A)$. 
Suppose that $A=A_1+A_2 i\in \C^{m\times n}$ where $A_1,A_2\in \R^{m\times n}$.
Choose a unit vector $v\in \R^n$ such that $\|\pi_W(A)v\|_2=\|\pi_W(A)\|_\sigma$.
Then we have
$$
\|\pi_W(A)\|_\sigma=\|\pi_W(A)v\|_2=\|A_1v\|_2\leq \sqrt{\|A_1v\|_2^2+\|A_2v\|_2^2}=\|Av\|_2\leq \|A\|_\sigma \|v\|_2=\|A\|_\sigma.
$$
This  shows that $W$ is a nice slice.
\end{example}
\section{1D total variation denoising}
\subsection{The total variation norm and its dual}
In this section we discuss the application to 1-dimensional total variation denoising. This example is particularly interesting because the corresponding norms are tight.

Define the difference map $D:\R^{n+1}\to \R^{n}$ by 
$$
D(b)=
\begin{pmatrix}
-1 & 1 &  & &\\
& -1 & 1 & &\\
& & \ddots & \ddots & \\
& & & -1 & 1
\end{pmatrix}
\begin{pmatrix}
b_1\\
b_2\\
\vdots\\
b_{n+1}\end{pmatrix}
=\begin{pmatrix}
b_2-b_1\\
b_3-b_2\\
\vdots\\
b_{n+1}-b_n\end{pmatrix}.
$$
The map $D$ is surjective, and the kernel is spanned by the vector ${\bf 1}=(1\ 1\cdots 1)^t$.
The dual map $D^\star:\R^{n}\to \R^{n+1}$ is given by 
$$
D^\star\begin{pmatrix}
a_1\\a_2\\\vdots \\a_n\end{pmatrix}
=\begin{pmatrix}
-a_1\\a_1-a_2\\ \vdots\\ a_{n-1}-a_{n}\\a_{n}\end{pmatrix}.
$$
If we compose the two maps we get
$$
DD^\star (a)=\begin{pmatrix}
2a_1-a_2\\
2a_2-a_1-a_3\\
2a_3-a_2-a_4\\
\vdots\\
2a_{n-1}-a_{n-2}-a_{n}\\
2a_{n}-a_{n-1}
\end{pmatrix}.
$$
The linear map $DD^\star$ is invertible. 
Let $V\subset \R^{n+1}$ be the subspace defined by
$$
V=\{a\in \R^{n+1}\mid a_1+\cdots+a_{n+1}=0\}.
$$
The image of $D^\star$ is  $V$
If $b=D^\star(DD^\star)^{-1}a$, then $b$ is the vector of minimal length with the property $Db=a$.

We define a norm $\|\cdot\|_X$ on $\R^n$ by
$$
\|a\|_X=\|D^\star a\|_1.
$$
Another norm $\|\cdot\|_Y$ on $\R^n$ is given by
$$
\|a\|_Y=\min \{ \|b\|_\infty\mid b\in \R^{n+1}\mbox{ and }Db=a\}.
$$
\begin{lemma}
Suppose that $b$ is a vector with $Db=a$. Let $m_+(b)=\max\{b_1,\dots,b_{n+1}\}$ and $m_-(b)=\min\{b_1,\dots,b_{n+1}\}$.
Then we have $\|a\|_Y=\frac{1}{2}(m_+(b)-m_-(b))$.
\end{lemma}
\begin{proof}
The vectors that map to $a$ under $D$ are of the
form $b-\lambda {\bf 1}$. 
We have $\|b-\lambda {\bf 1}\|_\infty=\max\{m_+-\lambda,m_--\lambda\}$.
This quantity is minimal if $\lambda=\frac{1}{2}(m_+(b) +m_-(b))$. In that case we have $\|b-\lambda {\bf 1}\|_\infty=\frac{1}{2}(m_+(b)-m_-(b))$.
\end{proof}

\begin{lemma}
The norms $\|\cdot\|_X$ and $\|\cdot\|_Y$ are dual to each other.
\end{lemma}
\begin{proof}
Suppose that $a,b\in \R^n$. Choose $e\in \R^{n+1}$ such that $De=b$ and $\|e\|_\infty=\|b\|_Y$.
Then we have
$$
\langle a,b\rangle=\langle a,De\rangle=\langle D^\star a,e\rangle\leq \|D^\star a\|_1\|e\|_\infty=\|a\|_X\|b\|_Y.
$$
Suppose that $a$ is nonzero and let $c=D^\star a\neq 0$. Define
$$
e=\begin{pmatrix}
\sgn(c_1)\\
\sgn(c_2)\\
\vdots\\
\sgn(c_{n+1})
\end{pmatrix}.
$$
Because $c_1+\cdots+c_{n+1}=0$, the set $\{c_1,\dots,c_{n+1}\}$ has positive and negative elements.
This implies that $m_+(e)=1$ and $m_-(e)=-1$.
If $b=De$ then we have $\|b\|_Y=\frac{1}{2}(m_+(e)-m_-(e))=1$.
$$
\langle a,b\rangle=\langle a,De\rangle=\langle D^\star a,e\rangle=\langle c,e\rangle=\|c\|_1=\|a\|_X.
$$
This shows that the norms $\|\cdot\|_X$ and $\|\cdot\|_Y$ are dual.
\end{proof}
For a vector $a\in \R^n$, $\|D^\star a\|_1$ is its total variation. Given a signal $c$ and an $\varepsilon>0$, 
a solution $a$ to the problem ${\bf M}_{X2}(\varepsilon)$ minimizes the total variation $\|D^\star a\|_1$
under the constraint $\|c-a\|_2\leq \varepsilon$. The function $a$ is typical a piecewise constant function.
Below is an example, where the blue function is $c$ and the red function is $a$.

\centerline{\includegraphics[width=8in]{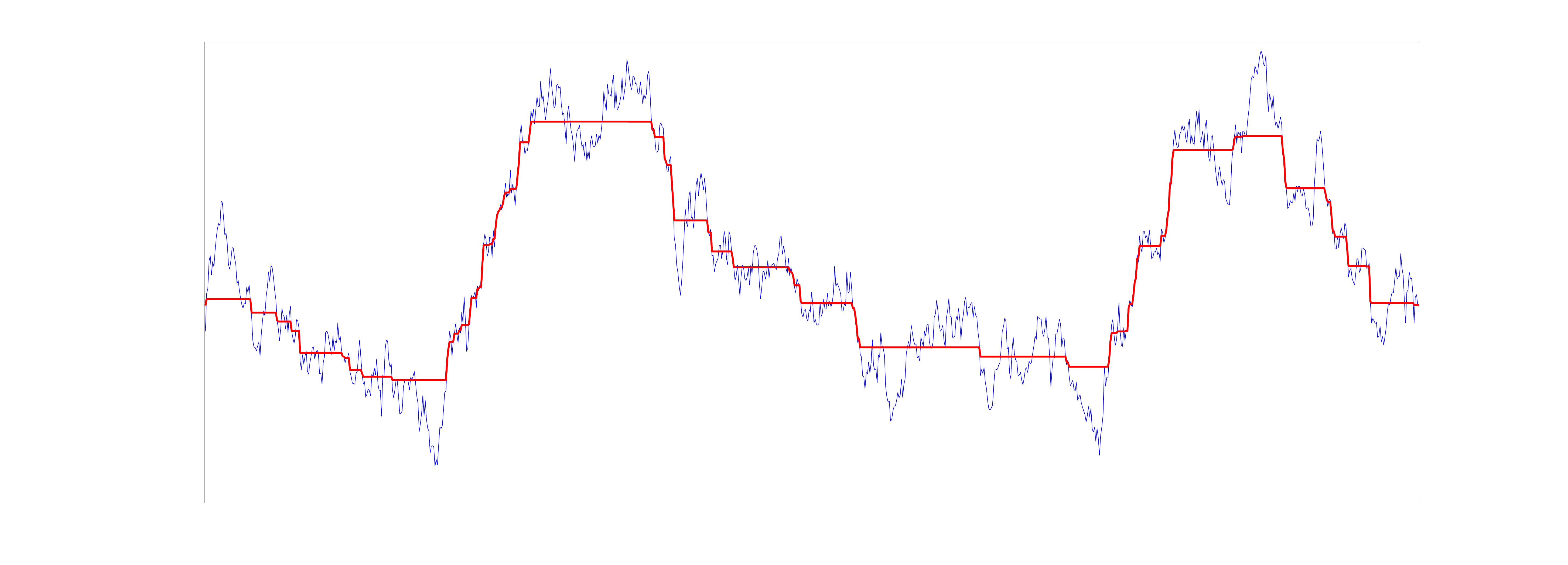}}

As we increase the value of $\varepsilon$, the sparsity decreases. Below we draw the signal $c$ in blue, and (a vertical translation of) the denoised signal $a=\shrink_Y(\varepsilon)$
for various values of $\varepsilon$ in red.

\centerline{\includegraphics[width=8in]{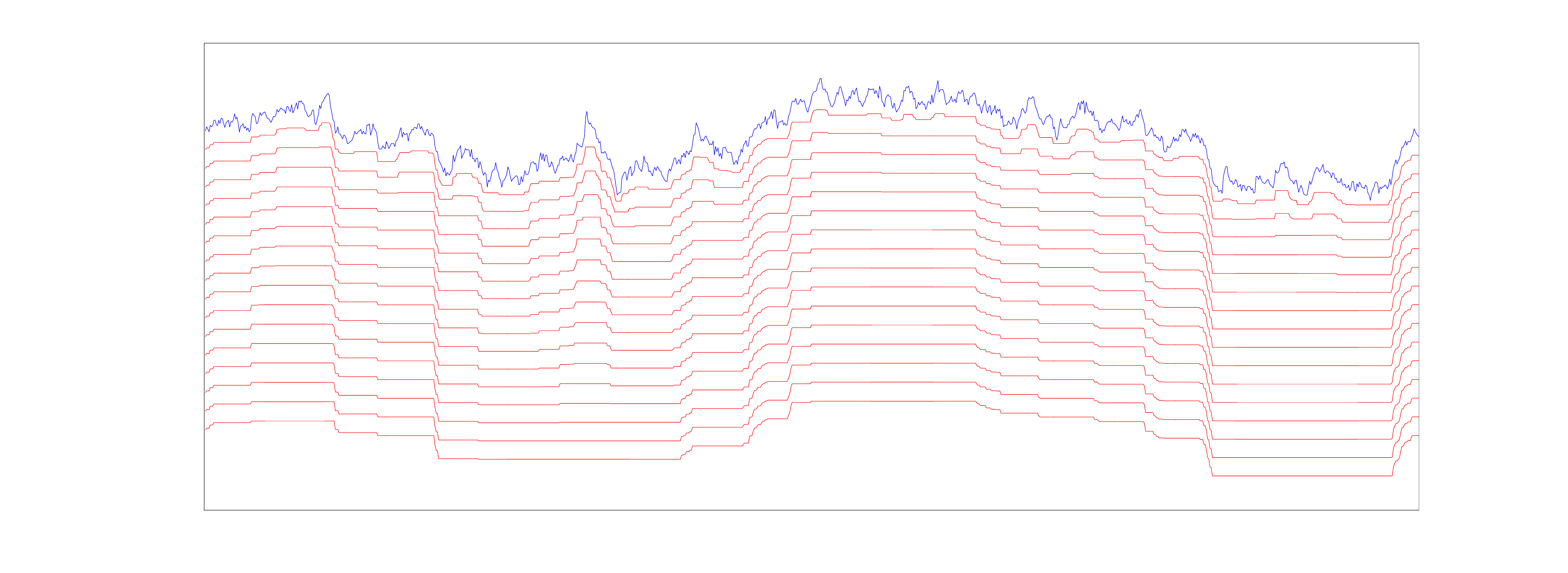}}

\subsection{Description of the unit ball}
We now will describe the unit balls $B_X$ and $B_Y$. 
\begin{definition}
A vector $(s_1\ s_2\ \cdots\ s_{n+1})^t$ is called a {\it signature sequence} if $s_1,s_2,\dots,s_{n+1}\in \{-1,0,1\}$
and $\{1,-1\}\subseteq \{s_1,\dots,s_{n+1}\}$.
For a signature sequence $s$, we define $\widetilde{F}_s$ as the set of all vectors $x\in \R^{n+1}$
such that $x_i=s_i$ when $s_i=\pm 1$, and $|x_i|\leq 1$ if $s_i=0$. We define $F_s=D(\widetilde{F}_s)$.
\end{definition}
\begin{lemma}
The set $F_s$ is a face of the unit ball $B_Y$.
\end{lemma}
\begin{proof}
The set $\widetilde{F}_s$ is a face of the unit ball $B_1\subset \R^{n+1}$ for the $\ell_1$ norm. 
 Suppose that $b=tb_1+(1-t)b_2\in F_s$ 
with $b_1,b_2\in B_Y$ and $0<t<1$. Then we can write $b_i=D(\widetilde{b}_i)$ with $\|\widetilde{b}_i\|_\infty\leq 1$ for $i=1,2$.
We have 
$$D(t\widetilde{b}_1+(1-t)\widetilde{b}_2)=tb_1+(1-t)b_2=b=D\widetilde{b}$$ for some $\widetilde{b}\in \widetilde{F}_s$.
We get that $\widetilde{b}$ and $t\widetilde{b}_1+(1-t)\widetilde{b}_2$ differ by a multiple of   ${\bf 1}$. 
Since the maximum and minimum entry of $\widetilde{b}$ are $1$ and $-1$ respectively, and $\|t\widetilde{b}_1+(1-t)\widetilde{b}_2\|_\infty\leq 1$ 
we deduce that $\widetilde{b}=t\widetilde{b}_1+(1-t)\widetilde{b}_2$. Now $\widetilde{b}_1,\widetilde{b}_2$ lie in the unit ball $B_1$ and $\widetilde{b}$ lies in the face $\widetilde{F}_s$. 
It follows that $\widetilde{b}_1,\widetilde{b}_2\in \widetilde{F}_s$ and $b_1=D(\widetilde{b}_1),b_2=D(\widetilde{b}_2)\in F_s$.
\end{proof}

The dimension of $F_s$ is equal to the number of $0$'s in the signature sequence $s$. The restriction of $D$ to $\widetilde{F}_s$ is injective, so $F_s$ and $\widetilde{F}_s$ have the same dimension.
\begin{proposition}
The faces of $B_Y$ of dimension $<n$ are exactly all $F_s$ where $s$ is a signature sequence. 
\end{proposition}
\begin{proof}
Suppose that $F$ is a proper face of the polytope $B_Y$. Then there exists a vector
 $a\in \R^n$ with $\|a\|_X=1$  and $F=\{b\in B_Y\mid \langle a,b\rangle=1\}$.
Let $\widetilde{a}=D^\star a$ and  $s_i=\sgn(\widetilde{a}_i)$ for $i=1,2,\dots,n+1$.
Since $\widetilde{a}\neq 0$, and $\widetilde{a}_1+\cdots+\widetilde{a}_{n+1}=0$, the vector $\widetilde{a}$ must have positive and negative coordinates.
In the sequence $s_1,\dots,s_{n+1}$ the elements $1$ and $-1$ both must appear.
For a vector $b\in \R^n$ with $\|b\|_Y=1$, let $\widetilde{b}\in \R^{n+1}$ be the unique vector with $D\widetilde{b}=b$
and $\|\widetilde{b}\|_\infty=1$. Then we have
$$
\langle a,b\rangle=\langle a ,D\widetilde{b}\rangle=\langle D^\star a,\widetilde{b}\rangle=\langle \widetilde{a},\widetilde{b}\rangle
$$
with $\|\widetilde{a}\|_1=\|\widetilde{b}\|_\infty=1$. Now $\langle \widetilde{a},\widetilde{b}\rangle=\sum_{i=1}^{n+1}\widetilde{a}_i\widetilde{b}_i=1$ if and only if for every $i$, $\widetilde{a}_i=s_i=0$ or  $\widetilde{b}_i=s_i$.
In other words, $\langle \widetilde{a},\widetilde{b}\rangle=1$ if and only if $\widetilde{b}\in \widetilde{F}_s$.
So $\langle a,b\rangle=\langle \widetilde{a},\widetilde{b}\rangle=1$ if and only if $b=D\widetilde{b}\in F_s$.
\end{proof}
\begin{theorem}
The norms $\|\cdot\|_X$ and $\|\cdot\|_Y$ are tight.
\end{theorem}
\begin{proof}
Suppose that $s$ is a signature vector. It suffices to construct a unitangent vector $u\in F_s$ by Proposition~\ref{prop:normtight}.
We define a vector $v\in \R^{n+1}$ as follows. If $s_i=\pm 1$ then $v_i=s_i$.
The other coordinates of $v$ are obtained by linear interpolation. If $i<j$,
 $s_i, s_{j}\neq 0$ and $s_{i+1}=\cdots=s_{j-1}=0$ then we define 
 $$v_k=\left(\frac{k-i}{j-i}\right)s_i+\left(\frac{j-k}{j-i}\right)s_j
 $$
 whenever $i<k<j$.
 If $s_1=s_2=\cdots=s_{i-1}=0$ and $s_i\neq 0$ then we define $v_k=s_i$ for $k<i$.
 If $s_{j+1}=s_{j+2}=\cdots=s_{n+1}=0$ and $s_j\neq 0$ then we define $v_k=s_j$ for $k>j$.
 For example, if $s=(0,0,1,0,0,0,-1,0,-1,0,0,1)$ then $v=(1,1,1,\frac{1}{2},0,-\frac{1}{2},-1,-1,-1,-\frac{1}{3},\frac{1}{3},1)$.

 From the construction follows that for every $i$ we have $(D^\star D v)_is_i=|(D^\star D v)_i|$.
 We define $u=Dv$. We have $\|u\|_Y=\|v\|_\infty=1$. From the construction it is clear that $v\in \widetilde{F}_s$ and $u\in F_s$.
 We have
 $$
 \langle u,u\rangle=\langle Dv,Dv\rangle=\langle D^\star Dv,v\rangle=\sum_{i=1}^{n+1} (D^\star D v)_i v_i=
 \sum_{i=1}^{n+1} |(D^\star u)_i|=\|u\|_X=\|u\|_X\|u\|_Y,
 $$
 so $u$ is unitangent.
 \end{proof}
We briefly discuss the combinatorics of the polytope $B_Y$.
Let $h_i$ be the number of faces of $B_Y$ of dimension $i$. For $i<n$, $h_i$ is the number of signature sequences with $i$ zeroes. We also have $h_n=1$.
The generating function for $h_0,h_1,\dots,h_n$
is $H(t)=h_0+h_1t+\cdots+h_nt^n$. The generating function for the set  $\{1,0,-1\}^{n+1}$ is $(2+t)^{n+1}$.
The generating function for $\{1,0\}^{n+1}$ and for $\{-1,0\}^{n+1}$ is $(1+t)^{n+1}$. The generating function for $\{(0,0,\dots,0)\}$ is $t^{n+1}$. So the generating function for the set of signature sequences, using inclusion-exclusion,
is $(2+t)^{n+1}-2(1+t)^{n+1}+t^{n+1}$. There is one face of dimension $n$ that does not correspond to a signature sequences, so we have 
$$H(t)=(2+t)^{n+1}-2(1+t)^{n+1}+t^{n+1}+t^n. 
$$
In particular, $h_0=2^{n+1}-2$ is the number of vertices of the ball $B_Y$, and $h_{n-1}=4{n+1\choose 2}-2{n+1\choose 2}=
n^2+n$ is the number of facets of $B_Y$, which is the number of vertices of $B_X$.
The total number of faces of $B_Y$ (and $B_X$)  is $H(1)=3^{n+1}-2^{n+2}+2$.

\begin{example}
Let $n=3$. We have the following signature sequences and corresponding unitangent vectors (written as row vectors):
$$
\begin{array}{c|c}
(1,1,1,-1),(1,0,1,-1),(0,1,1,-1),(0,0,1,-1)  &  (0,0,-2)^*\\
(1,1,0,-1),(0,1,0,-1) & (0,-1,-1) \\
(1,1,-1,1),(0,1,-1,1) & (0,-2,2)^*\\
(1,1,-1,-1),(1,1,-1,0),(0,1,-1,-1),(0,1,-1,0) & (0,-2,0)^*\\
(1,0,0,-1) &
(-\frac{2}{3},-\frac{2}{3},-\frac{2}{3}) \\
(1,0,-1,1) & (-1,-1,2)\\
(1,0,-1,-1),(1,0,-1,0) & (-1,-1,0)\\
(1,-1,1,1),(1,-1,1,0)  & (-2,2,0)^*\\
(1,-1,1,-1) & (-2,2,-2)^*\\
(1,-1,0,1) & (-1,-1,1)\\
(1,-1,-1,1) & (-2,0,2)^*\\
(1,-1,-1,-1),(1,-1,-1,0),(1,-1,0,-1),(1,-1,0,0) & (-2,0,0)^*\\
(-1,1,1,1),(-1,1,0),(-1,1,0,1),(-1,1,0,0) & (2,0,0)^*\\
(-1,1,1,-1) & (2,0,-2)^*\\
(-1,1,0,-1) & (1,1,-1)\\
(-1,1,-1,1) & (2,-2,2)^*\\
(-1,1,-1,-1),(-1,1,-1,0) & (2,-2,0)^*\\
(-1,0,1,1),(-1,0,1,0) & (1,1,0)\\
(-1,0,1,-1) & (1,1,-2)\\
(-1,0,0,1) & 
(\frac{2}{3},\frac{2}{3},\frac{2}{3}) \\
(-1,-1,1,1),(-1,-1,1,0),(0,-1,1,1),(0,-1,1,0) & (0,2,0)^*\\
(-1,-1,1,-1),(0,-1,1,-1) & (0,2,-2)^*\\
(-1,-1,0,1),(0,-1,0,1) & (0,1,1)\\
(-1,-1,-1,1)(-1,0,-1,1),(0,-1,-1,1),(0,0,-1,1) & (0,0,2)^*
\end{array}
$$
Every vertex of the unit ball $B_Y$ is unitangent. These vectors are marked with $\star$ and they correspond
to signature sequences that have no zeroes. Below is a 2-dimensional projection of the unit ball $B_Y$.

\centerline{\includegraphics[width=4in]{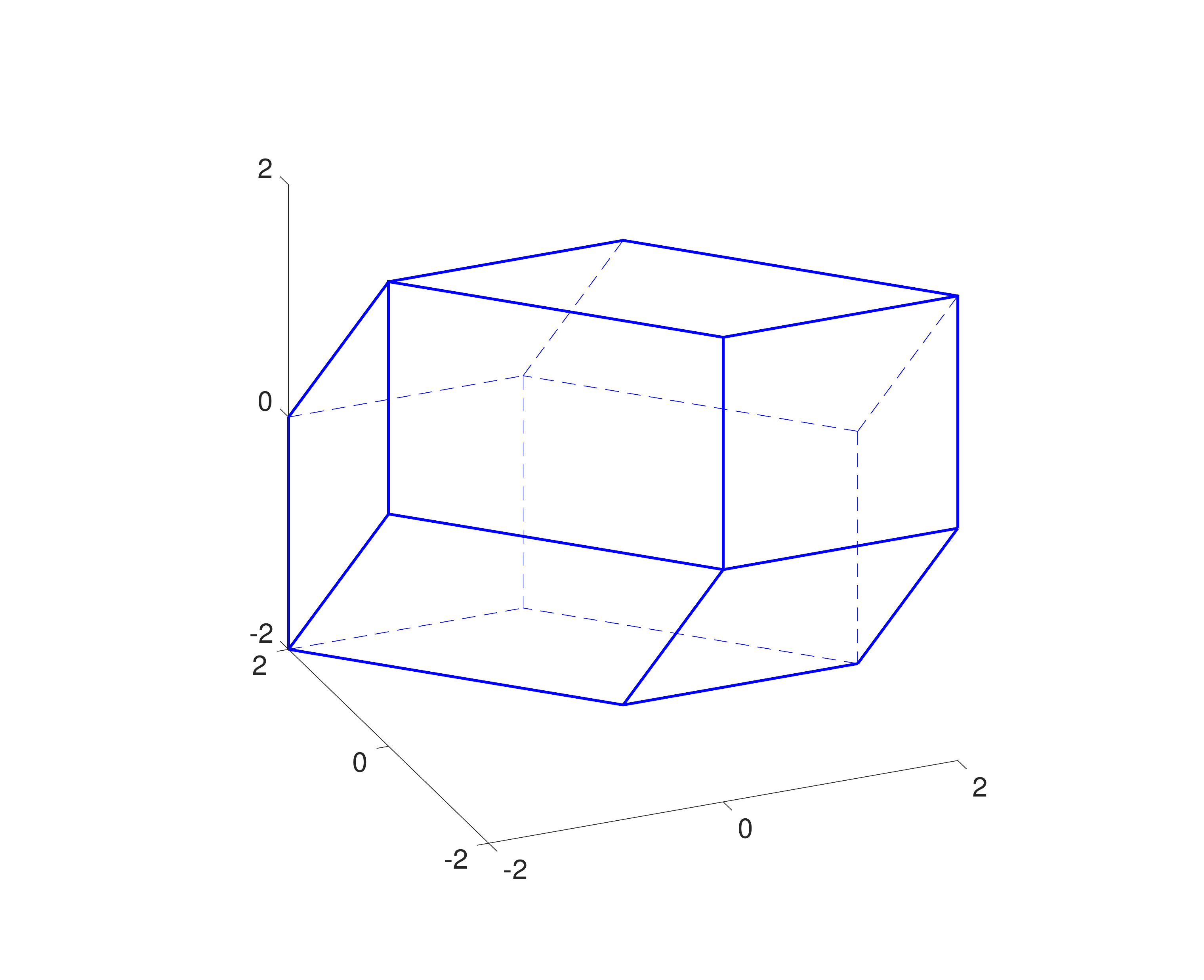}}

The unit ball $B_X$ is dual to the polytope $B_Y$ and is shown below:

\centerline{\includegraphics[width=4in]{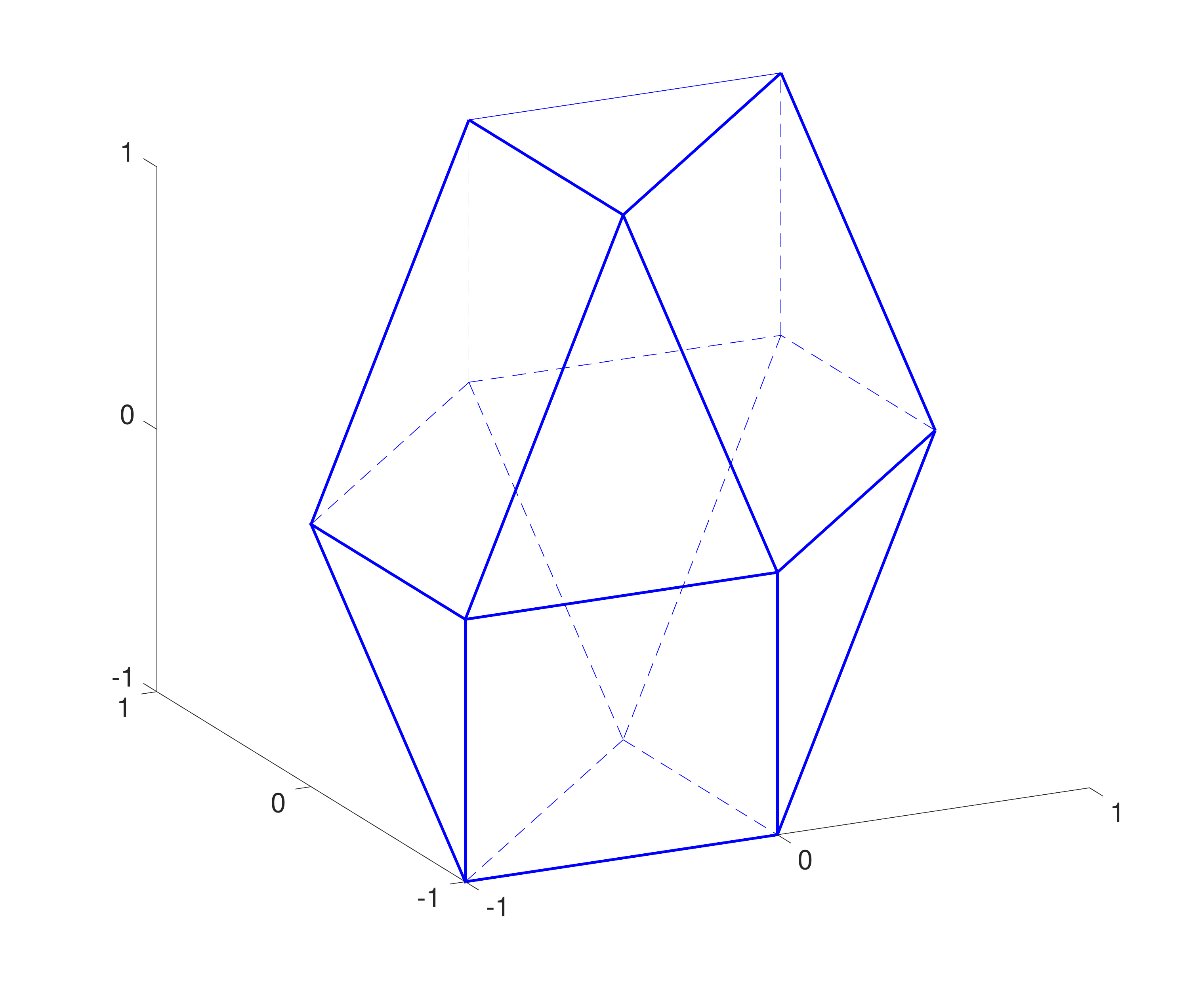}}

An example of an $XY$-slope decomposition is:
$$
\begin{pmatrix}
5\\-2\\3\end{pmatrix}=
\frac{3}{4}\begin{pmatrix}\frac{2}{3} \\[3pt] \frac{2}{3} \\[3pt]  \frac{2}{3}\end{pmatrix}+
\begin{pmatrix} 2\\0\\0\end{pmatrix}+
\frac{5}{4}\begin{pmatrix}2\\ -2 \\ 2\end{pmatrix}
$$

\end{example}
If $a\in \R^n$, then we have 
$$
\gsparse_Y(a)=|\{i\mid 1\leq i\leq n-1,\ a_i\neq a_{i+1}\}|+1.
$$

\subsection{The taut string method}
The restriction of $D:\R^{n+1}\to \R^n$ to $V$ gives an isomorphism between $V$ and $\R^{n}$.
Now we can view the bilinear form $\langle\cdot,\cdot\rangle$ as a bilinear form on $V$ and for $a,b\in V$ we have
$$
\langle a,b\rangle=\langle Da,Db\rangle=\sum_{i=1}^n (a_{i+1}-a_i)(b_{i+1}-b_i).
$$
In particular, we have
$$
\|a\|_2=\sqrt{\langle a,a\rangle}=\sqrt{\sum_{i=1}^n (a_{i+1}-a_i)^2}.
$$
We can also view $\|\cdot\|_X$ and $\|\cdot\|_Y$ as norms on $V$ and for $c=(c_1,\dots,c_{n+1})\in V$ we have
$$
\|c\|_X=\|D^\star Dc\|_1
$$
and
$$
\|c\|_Y=\textstyle \frac{1}{2}(\max_i\{c_i\}-\min_i\{c_i\}).
$$
For $c\in \R^{n+1}$, we consider the following  optimization problem:\\

\noindent ${\bf TS}^c(\varepsilon)$: Find a vector $a\in \R^{n+1}$   with  $\|c-a\|_\infty\leq \varepsilon$
such that $\|D^\star Da\|_1$ is minimal.\\

We call this the Taut String problem. This problem, and some generalizations to higher order, were studied in \cite{MvdG}.
If the vectors $a,b,c$ are discretized functions, Then the graph $a$ lies between $c-\varepsilon$ and $c+\varepsilon$. Now $D^\star Da$ is a discrete version of the second derivative. 
The value $|D^\star Da)_i|$ is a measure of how much the graph bends at vertex $i$. So $\|D^\star Da\|_1$ is the total amount of bending and we try to minimize this. Visually we can see $a$ as a string between $c-\varepsilon$ and $c+\varepsilon$ and we pull  $a$ on both ends so that  the string is taut.
 The {\em Taut String Algorithm} described in \cite{DK} computes $a$ in time $O(n)$ (see also~\cite{Condat}). 
The function $a$ is piecewise linear and $D(a)$ is piecewise constant. Total Variation Denoising of time signals has applications in statistics to estimate a density function from a collection of measurements (see~\cite{BBBB}). It was also used in \cite{BAS,NBWD} for analysing heart rate variabilty signals to predict hemodynamic decompensation.
In the graph below, we have drawn $c-\varepsilon$ and $c+\varepsilon$ for a function $c$, as well as the solution $a$ to the Taut String problem ${\bf TS}^c(\varepsilon)$. It appears as a tight string that is in between the graphs of $c-\varepsilon$ and $c+\varepsilon$ and is a piecewise linear approximation of $c$.

\centerline{\includegraphics[width=8in]{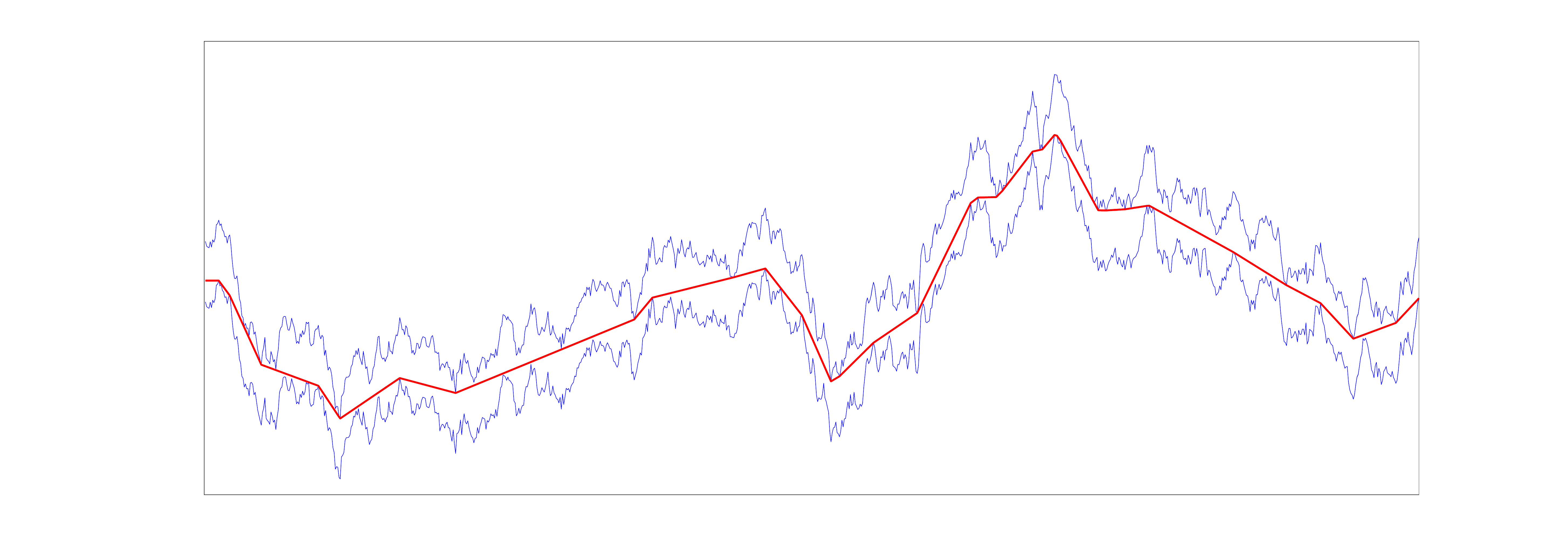}}

As the value of $\varepsilon$ increases, the sparsity decreases:

\centerline{\includegraphics[width=8in]{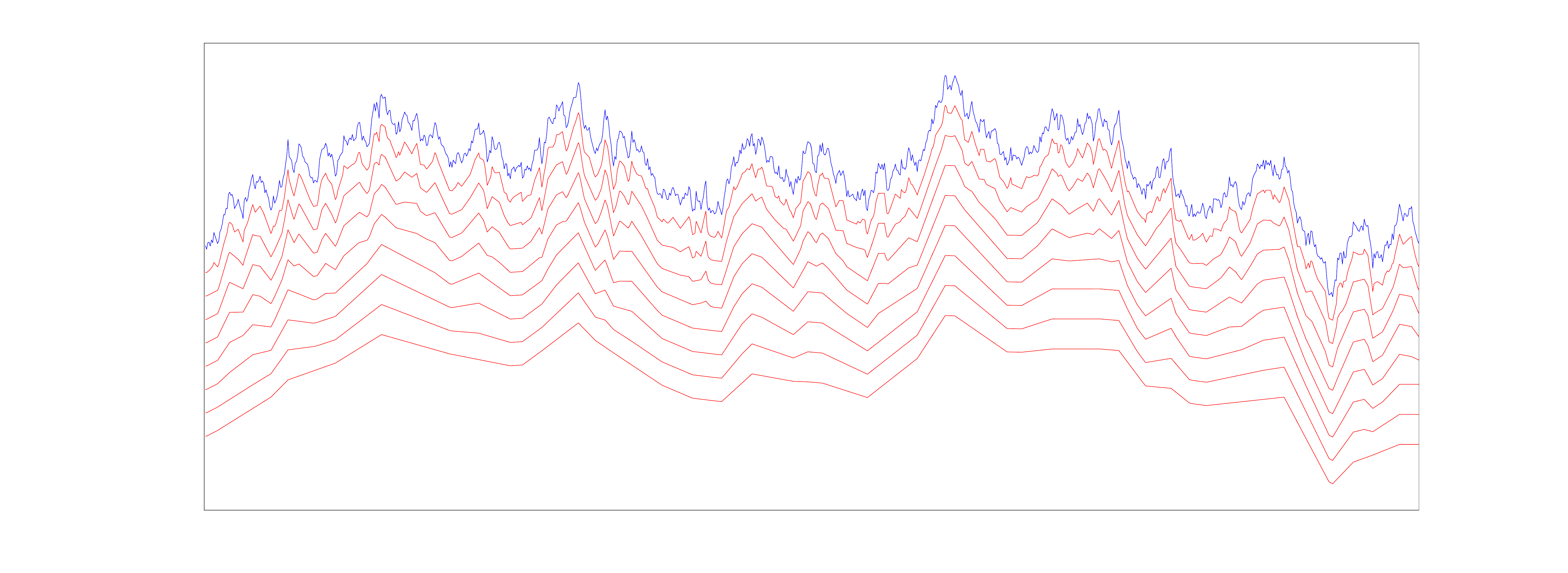}}

Let $\pi:\R^{n+1}\to V$ be the projection onto $V$ defined by 
$$\pi(a_1,\dots,a_{n+1})=(a_1,\dots,a_{n+1})-\frac{\sum_{i=1}^{n+1} a_i}{n+1}(1,1,\dots,1).
$$

\begin{lemma}
Suppose that $c\in V$ and $a\in \R^{n+1}$  minimizes  $\|D^\star Da\|_1$  under the constraint $\|c-a\|_\infty\leq \varepsilon$.
Then $\pi(a)$ is a solution to ${\bf M}_{XY}^{c}(\varepsilon)$ and $c=\pi(a)+\pi(b)$ is an $XY$-decomposition,
where $b=c-a$.
\end{lemma}
\begin{proof}
We have $\|\pi(b)\|_Y\leq \|b\|_\infty\leq \varepsilon$.
Suppose that  $c=a'+b'$ with $a',b'\in V$ and $\|b'\|_Y\leq \varepsilon$. Then there is a vector $\overline{b}\in \R^{n+1}$
with $\pi(\overline{b})=b$ and $\|\overline{b}\|_\infty=\|b'\|_X\leq \varepsilon$. If we define $\overline{a}=c-\overline{b}$, then 
$\pi(\overline{a})=a'$ and we have
$$
\|a'\|_X=\|D^\star Da'\|_1=\|D^\star D\overline{a}\|_1\geq \|D^\star Da\|_1=\|D^\star D\pi(a)\|_1=\|\pi(a)\|_X$$
This shows that $\pi(a)$ is a solution to ${\bf M}_{XY}^{c}(\varepsilon)$.
\end{proof}

\subsection{An ECG example}
For a  $500$ Hz  noisy electrocardiogram (ECG) signal $c\in \R^{2000}$ of 4 seconds long,
we graph the Pareto frontier (and sub-frontier)
of $c$:

\centerline{\includegraphics[width=3in]{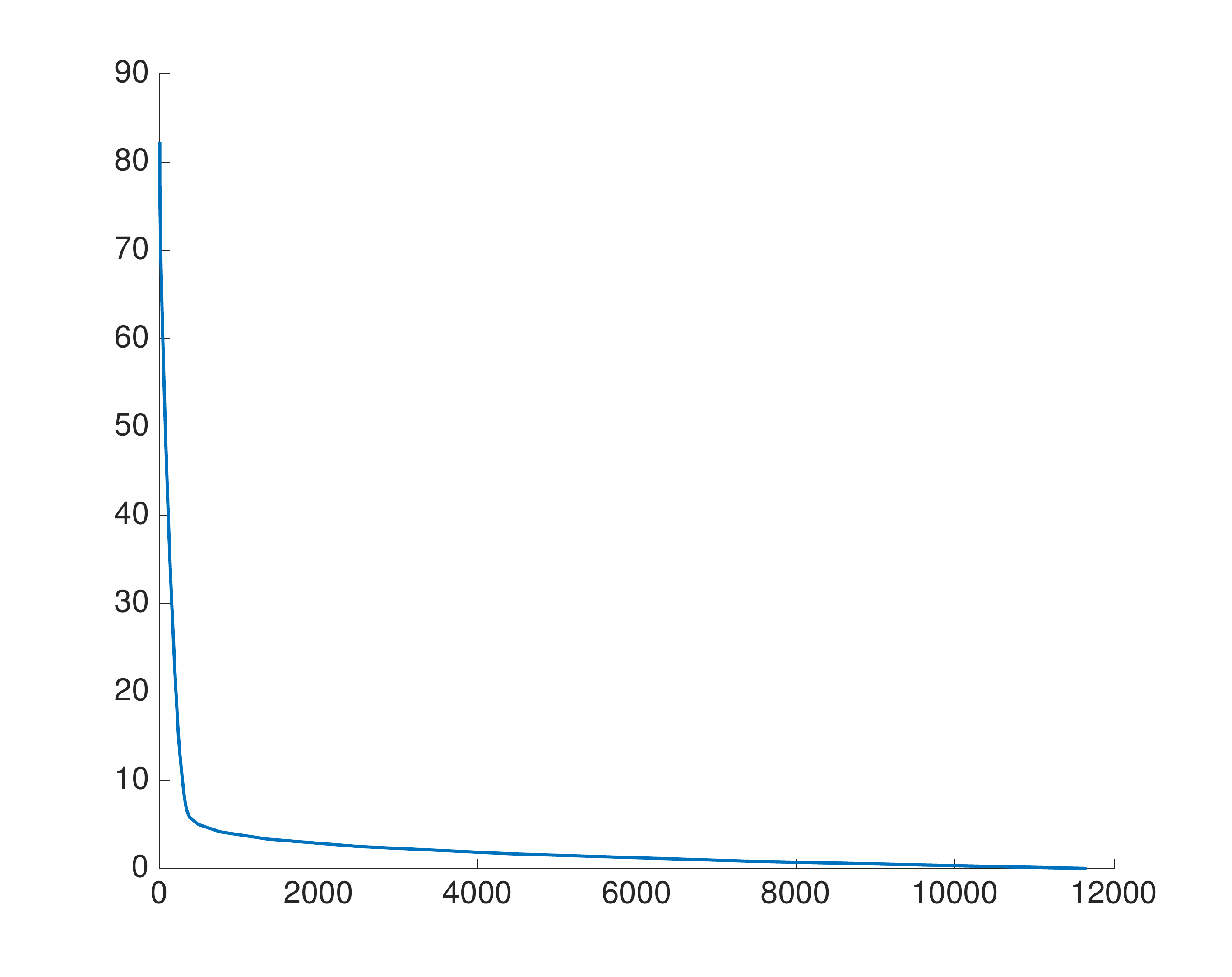}}

(If $c\not\in V$, then we can remove the baseline by replacing $c$ with $\pi(c)\in V$.)
The graph is $L$-shaped, where the vertical leg corresponds to the sparse signal, and the horizontal leg corresponds to noise.
The vertical leg starts near the point $(300,8)$. This means that there exists a decomposition $c=a+b$
with $\|b\|_\infty=8$ and $\|D^\star Da\|_1\approx 300$. The signal $a$ is the denoised signal. 
The singular value region for $c$ is shown below:

\centerline{\includegraphics[width=3in]{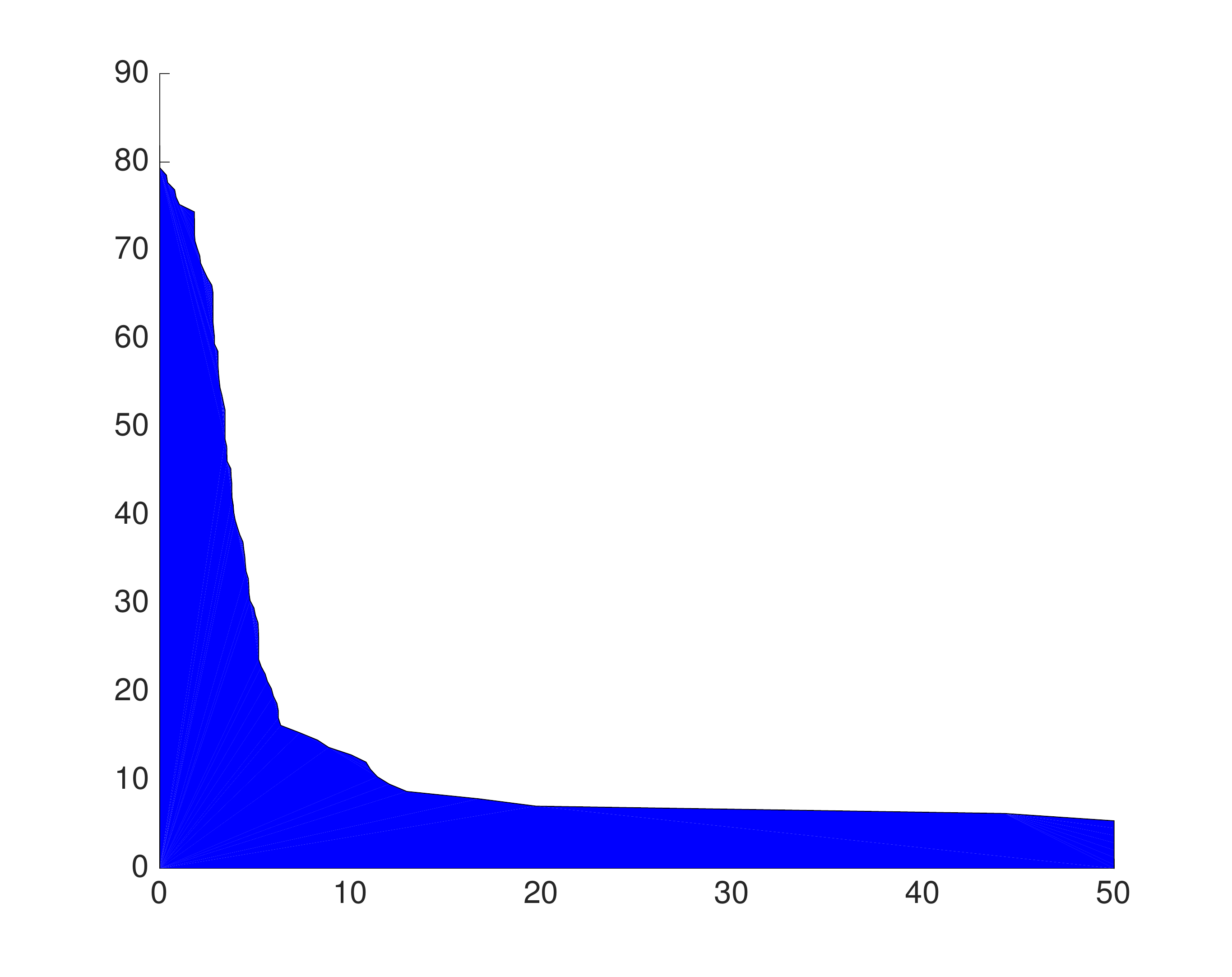}}

The $x$-axis is cut of here in order to better visualize the graph. The graph approaches the $x$-axes slowly and meets the $x$-axis near $x=1000$.
The horizontal leg corresponds to noise. To estimate the noise level,  we look at where the horizontal leg starts. To find a cutoff for the singular values
one proceeds as in principal component analysis.
A reasonable cutoff is again $y=8$. The value 8 is an estimation of the maximum amplitude of noise, which is more than the standard deviation (which is closer to $4$ in this case). Below we graph the noisy signal, and the denoised signals for $y=2,4,6,8,10,12$. The denoised signal for $y=8$ is colored green.

\centerline{\includegraphics[width=8in]{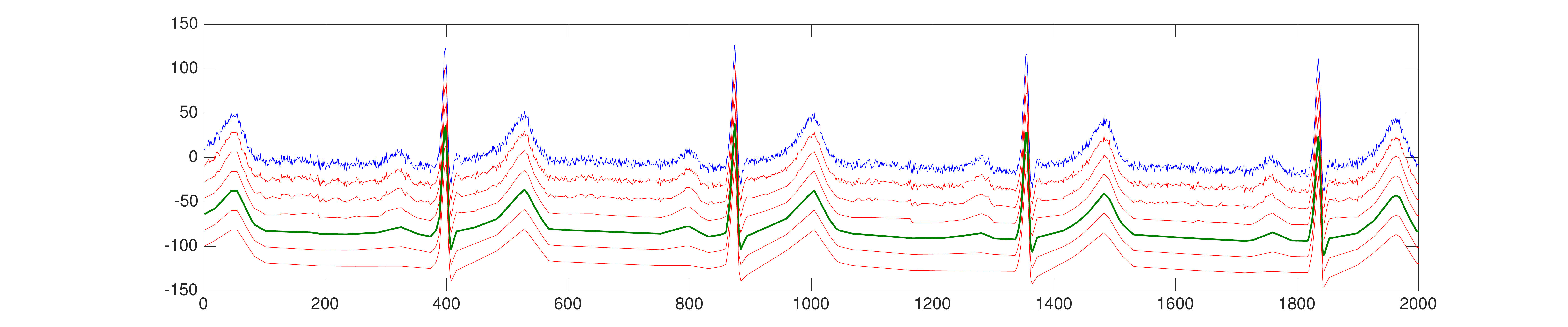}}

\subsection{higher order total variation denoising}
Suppose that $c:[0,1]\times [0,1]\to \R$ is a function. One can also denoise by using a higher derivative to regularize the $\ell_2$ norm by minimizing

$$
{\textstyle \frac{1}{2}}\|c-a\|_2^2+\lambda\|a^{(k)}\|_1 =\frac{1}{2}\int_0^\infty (c(x)-a(x))^2\,dx+\lambda \int_0^1 |a^{(k)}(x)|\,dx
$$
is small. For $k=1$ we just get  1D Total Variation Denoising. We consider a discrete version.

Define $D^{(k)}:\R^{n+k}\to \R^{n}$ as the composition $\underbrace{D\circ D\circ \cdots\circ D}_k$. We can view $D^{(k)}$ as the $k$-th discrete derivative.
For example, for $k=2$ we have
$$
D^{(2)}(v_1,v_2,\dots,v_{n+2})=(v_1-2v_2+v_3,v_2-2v_3+v_4,\dots,v_n-2v_{n+1}+v_{n+2}).
$$
Define a semi-norm on $\R^{n+k}$ by $\|v\|_X=\|D^{(k)}v\|_1$. Restricting the norm to
$$
V=\{(v_1,\dots,v_{n+k})\in \R^{n+k}\mid \textstyle \sum_{i=1}^{n+k} v_ii^d=0\mbox{ for $d=0,1,\dots,k-1$}\}
$$
gives a norm, and let $\|\cdot\|_Y$ be the dual to this norm.

\begin{problem}
Given $c\in \R^{n+k}$, minimize
$\frac{1}{2}\|c-a\|_2^2+\lambda \|a\|_X$.
\end{problem}
For $k=2$, this problem is called $\ell_1$-trend filtering. An overview of $\ell_1$-trend filtering is given in \cite{KKBG}. Some applications are in  financial time series (\cite{YY}), 
macroeconomics (\cite{YJ}) automatic control (\cite{OGLB}), oceanography (\cite{Wunsch}) and geophysics ( \cite{Klinger}). In $\ell_1$-trend filtering, 
the function $a$ is piecewise linear. More generally, for $k>2$, the $(k-1)$-th derivative of $a$ will be piecewise constant. This case has been studied in \cite{MvdG}.
For $k\geq 2$, the norms $\|\cdot\|_X$ and $\|\cdot\|_Y$ are not tight.

\section{The ISTA algorithm for $XY$-decompositions}
A general formulation of the Iterative Shrinkage-Tresholding Algorithm (ISTA) was given in~\cite{DDD}. 
A map $f:\R^n\to \R^n$ is called {\it non-expansive} if it is Lipschitz with Lipschitz constant 1, i.e., $\|f(v)-f(w)\|_2\leq \|v-w\|_2$ for all $v,w\in \R^n$. 
\begin{lemma}
If $\|\cdot\|_X$ is a norm on  $\R^n$, then the functions $\proj_X(\cdot,x)$ and $\shrink_X(\cdot,x)$ are non-expansive.
\end{lemma}
\begin{proof}
Suppose that $v,w\in \R^n$, and let $v'=\proj_X(v,x)$ and $w'=\proj_X(w,x)$.
For $t\in [0,1]$ we have $(1-t)v'+tw'\in B_X(x)$. By definition of $v'=\proj_X(v,x)$ we have
$$
\|(v-v')+t(v'-w')\|_2=\|v-\Big((1-t)v'+tw'\Big)\|_2\geq \|v-v'\|_2.
$$
Squaring both sides yields
$$
2t\langle v-v',v'-w'\rangle+t^2\|w'-v'\|_2^2\geq 0.
$$
If we take the limit $t\to 0$, we get
$$
\langle v-v',v'-w'\rangle\geq 0.
$$
By symmetry, we also get
$$
\langle w'-w,v'-w'\rangle=\langle w-w',w'-v'\rangle \geq 0
$$
Adding both equations, yields
$$
\|v-w\|_2^2-\|v'-w'\|_2^2-\|(v-w)-(v'-w')\|_2^2=2\langle (v-w)-(v'-w'),v'-w'\rangle \geq 0.
$$
It follows that
$$\|v'-w'\|_2\leq \|v-w\|_2\mbox{ and } \|(v-w)-(v'-w')\|_2\leq \|v-w\|_2.
$$
This shows that $\proj_X(\cdot,x)$ and $\shrink_X(\cdot,x)$ are nonexpansive.

\end{proof}

Suppose that $V=\R^n$. Let $D:\R^m\to \R^n$ be a surjective linear map and suppose that $\|\cdot\|_{\overline{X}}$ is a norm on $\R^m$.
We define a norm $\|\cdot\|_X$ on $\R^n$ by
$$
\|c\|_X=\min\{\|\overline{c}\|_{\overline{X}}\mid \overline{c}\in \R^m,\ D\overline{c}=c\}.
$$
The dual norm $\|\cdot\|_Y$ of $\|\cdot\|_X$ is defined by
$$
\|c\|_Y:=\|D^\star c\|_{\overline{Y}}
$$
where $\|\cdot\|_{\overline{Y}}$ is the dual norm to $\|\cdot\|_{\overline{X}}$. Assume that we can easily compute the norms  $\|\cdot\|_{\overline{X}}$, $\|\cdot\|_{\overline{Y}}$ and the projection function $\proj_{\overline{X}}$. We will also assume that the singular values of $D$ all lie in $[0,1]$.
 We have the following algorithm for computing $\proj_{X}(c,x)$ for $c\in V$ and $x\geq 0$. There is one more parameter, $\delta$, which specifies the accuracy of the output.
 We assume $0<\delta<1$ and the closer $\delta$ is to $1$, the  more accurate the output will be. 
\begin{algorithmic}[1]
\Function{${\rm proj}_{X}$}{$c,x,\delta$}
\State $e \leftarrow 0$
\While{ $\langle D e, c-De\rangle\leq \delta\cdot  \|e\|_{\overline{X}} \cdot \|D^\star(c-De)\|_{\overline{Y}}$ }
\State $e\leftarrow \proj_{\overline{X}}(2 D^\star (c-De))+e,x)$
\EndWhile
\State \Return $De$
\EndFunction
\end{algorithmic}
\section{Basis Pursuit Denoising, LASSO and the Dantzig selector}
In Basis Pursuit (BP) one tries to solve the equation $Av=c$ where $A$ is a given $n\times m$ matrix and $v$ is a sparse vector (few nonzero entries). We will assume that $A$ has rank $n$ and that the system is underdetermined, ($m>n$). It was shown in \cite{CT1} that $v$ often can be found by minimizing $\|v\|_1$ under the constraint $Av=c$. This can be done efficiently using linear programming.   We define a norm $\|\cdot\|_X$ on $\R^n$ by 
$$
\|c\|_X=\min\{\|v\|_1\mid Av=c\}.
$$
Note that evaluating the norm $\|c\|_X$ of some vector $c$ is a Basis Pursuit problem. We also have
$$
\sparse_X(c)=\min\{\|v\|_0\mid Av=c\}.
$$

 In the presence of noise, one minimizes $\|v\|_1$ under the constraint $\|Av-c\|_2\leq y$.  This is called Basis Pursuit Denoising (BPDN), see \cite{CDS,CRT}.
If we set $a=Av$, then we minimize $\|a\|_X$ under the  the constraint $\|c-a\|_2\leq y$. If we set $b=c-a$, then we minimize $\|c-b\|_X$ under the constraint $\|b\|_2\leq y$.
This is the  optimization problem ${\bf M}_{X2}^c(y)$. 

Sometimes BPDN is formulated as the problem of minimizing $\ell_1$-regularized function
$$
{\textstyle \frac{1}{2}}\|Av-c\|_2^2+\lambda\|v\|_1.
$$
This is the same problem as minimizing
$$
{\textstyle \frac{1}{2}}\|c-a\|_2^2+\lambda\|a\|_X.
$$
The equivalence between the two formulations of BPDN is well-known (see also Proposition~\ref{prop:BPDNformulations}). 

The LASSO problem asks to minimize $\|c-Av\|_2$ under the constraint $\|v\|_1\leq x$. This is equivalent to minimizing $\|c-a\|_2$ under the constraint $\|a\|_X\leq x$.
This is the optimization problem ${\bf M}_{2X}^c(x)$.

The dual norm of $\|\cdot\|_X$ is defined by
$$
\|c\|_Y:=\|A^\star c\|_\infty.
$$
Since the $X2$-decompositions and the $2Y$ decompositions are the same, we have two more approaches for finding the $X2$-decompositions
(dual LASSO/BPDN):
\begin{enumerate}
\item ${\bf M}^c_{2Y}(y)$: Under the constraint $\|A^\star (c-a)\|_\infty\leq y$  we minimize $\|a\|_2$ (\cite{OPT,vdBF}).
\item ${\bf M}^c_{Y2}(x)$: Under the constraint $\|a\|_2\leq x$, we minimize $\|A^\star  (c-a)\|_\infty$.
\end{enumerate}

 To find an $XY$-decomposition we  can minimize $\|a\|_X$ under the constraint 
$\|c-a\|_Y\leq y$ (${\bf M}^c_{XY}(y)$).  If we set $a=Av$ then this is equivalent to minimizing $\|v\|_1$ under the constraint $\|A^\star (Av-c)\|_\infty\leq y$.
This optimization problem is called the {\em Dantzig selector} (\cite{CT3}). 

The Dantzig selector does not always have the same solution as LASSO (or the other equivalent problems). Some conditions were given in \cite{AR} when the Dantzig selector and LASSO have the same solutions. If $a_1,a_2,\dots,a_m$ are the columns of $A$, then the unit ball $B_X$ is the convex hull of $a_1,\dots,a_m,-a_1,\dots,-a_m$.  Theorem~\ref{prop:normtight} gives a necessary and sufficient condition for this polytope so that  Dantig selector and LASSO have the same solutions for every $c\in V$.

\section{Total Variation Denoising in Imaging}
\subsection{The 2D total variation norm}
We can view a grayscale image as a function $c:[0,1]\times [0,1]\to \R$. In the anisotropic Rudin-Osher-Fatemi (\cite{ROF}) total variation model,  we seek a decomposition $c=a+b$
such that
$$
\frac{1}{2}\int_0^1\int_0^1 (c(x,y)-a(x,y))^2\,dx\,dy+\lambda\int_{0}^1\int_0^1 \|\nabla a(x,y)\|_1\,dx\,dy
$$
is small. (In the isotropic model we replace $\|\nabla a(x,y)\|_1$ by $\|\nabla a(x,y)\|_2$.)

A discrete formulation of the model is as follows.
We can view a grayscale image of $m\times n$ pixels  as a matrix $c\in \R^{m\times n}$. 
We define a map $F:\R^{m\times n}\to \R^{m\times (n-1)}\times \R^{(m-1)\times n}$ by
$$
F(c)=(d,e)
$$
where $d_{i,j}=c_{i,j}-c_{i,j+1}$ and $d_{i,j}=e_{i,j}-e_{i,j+1}$ for all $i,j$. We define a total variation semi-norm by
$$
\|c\|_X=\|F c\|_1=\sum_{i=1}^{m}\sum_{j=1}^{n-1}|c_{i,j}-c_{i,j+1}|+\sum_{i=1}^{m-1}\sum_{j=1}^n|c_{i,j}-c_{i+1,j}|.
$$
The restriction of $\|\cdot \|_X$ to the set
$$V=\{c\in \R^{m\times n}\mid \textstyle \sum_{i,j} c_{i,j}=0\}$$
is a norm. We can always normalize an image by subtracting the average value to obtain an element of $V$.
Let $F^\star:\R^{m\times (n-1)}\times \R^{(m-1)\times n}\to \R^{m\times n}$ be the dual of $F$.
We have
$$
F^\star(d,e)_{i,j}=d_{i,j}-d_{i,j-1}+e_{i,j}-e_{i-1,j},\quad 1\leq i\leq m,\ 1\leq j\leq n.
$$
with the conventions that $d_{i,0}=d_{i,n}=e_{0,j}=e_{m,j}=0$ for all $i, j$.

The dual norm of $\|\cdot\|_X$ is 
$$
\|c\|_Y=\min\{\|(d,e)\|_\infty\mid F^\star(d,e)=c\}.
$$
Total Variation Denoising is usually formulated as follows
\begin{problem}  Minimize 
$$
{\textstyle \frac{1}{2}}\|c-a\|_2^2+\lambda \|Fa\|_1
$$
is minimal. 
\end{problem}
Since the paper of Rudin, Osher and Fatemi (\cite{ROF}), several other algorithms have been proposed, for example Chambolle's algorithm (\cite{Chambolle}), 
the split Bregman method (\cite{GO}) and the efficient  primal-dual hybrid gradient algorithm (\cite{ZC}).

A region in $R\subseteq\{1,2,\dots,m\}\times \{1,2,\dots,n\}$

\subsection{Spareseness and total variation}
We will study the notion of $X$-sparsity in this context. Suppose that $F\in V$ is an image. An $F$-region is a maximal connected subset of $\{1,2,\dots,m\}\times \{1,2,\dots,n\}$ on which $F$ is constant.
\begin{proposition}
For an image $F$ we have $\gsparse(F)=d-1$ where $d$ is the number of connected regions on which $F$ is constant.
\end{proposition}
\begin{proof}
Let $C$ be the smallest  facial $X$-cone containing $F$. 
By Lemma~\ref{lem:smallestXcone}, $F$ lies in $C$ if and only if $\|F-\varepsilon G\|_X+\|\varepsilon G\|_X=\|F\|_X$ for some $\varepsilon>0$.
So
$G$ lies in $C$ if and only if the following properties are satisfied for all $i,j$:
\begin{enumerate}
\item  $G(i,j)>G(i,j+1)$ implies $F(i,j)>F(i,j+1)$;
 \item $G(i,j)<G(i,j+1)$ implies $F(i,j)<F(i,j+1)$;
 \item $G(i,j)>G(i+1,j)$ implies $F(i,j)>F(i+1,j)$;
 \item $G(i,j)<G(i+1,j)$ implies $F(i,j)<F(i+1,j)$.
 \end{enumerate}
 Taking the contrapositive in each statement (and changing the indexing), we see that for all $i,j$ we have:
 \begin{enumerate}
\item  $F(i,j)\geq F(i,j+1)$ implies $G(i,j)\geq G(i,j+1)$;
 \item $F(i,j)\leq F(i,j+1)$ implies $G(i,j)\leq G(i,j+1)$;
 \item $F(i,j)\geq F(i+1,j)$ implies $G(i,j)\geq G(i+1,j)$;
 \item $F(i,j)\leq F(i+1,j)$ implies $G(i,j)\leq G(i+1,j)$.
 \end{enumerate}
 It is clear that for all $G\in C$, we have that $G$ is constant on the connected regions on which $F$ is constant. The function $G$
can have arbitrary values on these  $d$ connected regions as long as the 4 inequalities above and the linear constraint $\sum_{i,j}G(i,j)=0$
are satisfied. It follows that $\gsparse_X(F)=\dim C=d-1$.
\end{proof}

\subsection{The total variation norm is not tight}

Below we denoised the image of the letter C using ROF total variation denoising. The original image has only 2 colors, black and white, and has  geometric $X$-sparsity 1.
In the denoised images, there are various shades of grey, and the geometric $X$-sparsity is more than 1. So ROF denoising may increase the geometric sparsity, so the norms $\|\cdot\|_X$ and $\|\cdot\|_Y$ are not tight.

\centerline{\includegraphics[width=3in]{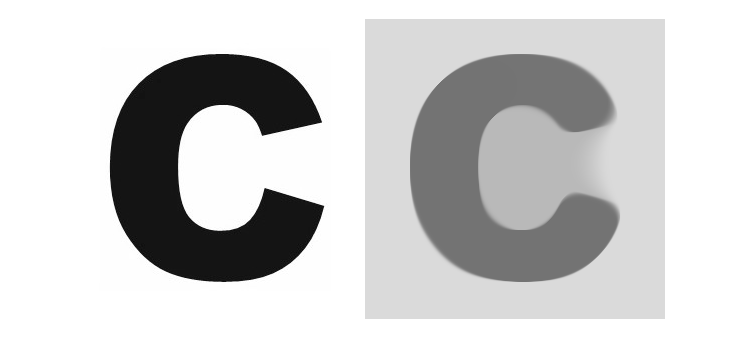}}



\section{Tensor Decompositions}\label{sec:tensor}
\subsection{CP decompositions}
One of the motivations for writing this paper is the study of tensor decompositions. Suppose that $\F$ is the field $\C$ or $\R$, and that
$V^{(1)},\dots,V^{(d)}$ are finite dimensional $\F$-vector spaces. 
Define 
$$V=V^{(1)}\otimes_{\F} V^{(2)}\otimes_{\F}\cdots \otimes_{\F} V^{(d)}.
$$
With a tensor we will mean an element of $V$. 
Elements of $V$ can be thought of as multi-way arrays of size  $n_1\times n_2\times \cdots \times n_d$ where $n_i=\dim V^{(i)}$.
A {\em simple tensor} (also called {\em rank one tensor} or {\em simple tensor}) is a tensor of the form
$$
v^{(1)}\otimes v^{(2)}\otimes \cdots \otimes v^{(d)}
$$
where $v^{(i)}\in V^{(i)}$ for all $i$. Not every tensor is simple, but every tensor can be written as a sum of simple tensors. 
\begin{problem}[Tensor Decomposition]\label{prob:tensor}
Given a tensor $T$, find a decomposition $T=v_1+v_2+\dots+v_r$ where $v_1,\dots,v_r$ are simple tensors and $r$ is minimal.
\end{problem}
Hitchcock defined in \cite{Hitchcock}
 the rank of the tensor $T$ as the smallest $r$ for which such a decomposition exists and this minimal rank decomposition is called  the {\em canonical polyadic decomposition}. Problem~\ref{prob:tensor}
 is  also known as the PARAFAC (\cite{Harshman})  or CANDECOMP (\cite{CC}) model. 
 Finding the rank of a tensor is an NP-hard problem.
 Over $\Q$ this was shown in \cite{Hastad} and in our case, $\F=\R$ or $\F=\C$,  this was proved in~\cite{HL}.
 \subsection{The CoDe model and the nuclear norm}
Even in relatively small dimensions, there are examples of tensors for which the rank is unknown. Using the heuristic of convex relaxation, we consider the following problem:
\begin{problem}[CoDe model]
Given a tensor $T$, find a decomposition $T=v_1+v_2+\cdots+v_r$ where $v_1,\dots,v_r$ are simple tensors and
$\sum_{i=1}^r \|v_i\|_2$ is minimal.
\end{problem}
The nuclear norm for tensors was explicitly given \cite{LC1,LC2}, but the ideas go back to \cite{Grothendieck} and  \cite{Schatten} :
\begin{definition}
The nuclear norm $\|T\|_\star$ of the tensor as the smallest possible value of $\sum_{i=1}^r \|v_i\|_2$ such that $v_1,\dots,v_r$ are simple tensors and $T=\sum_{i=1}^r v_i$.
\end{definition}
A matrix can be viewed as a 2-way tensor and in this case the nuclear norm for tensors coincides with the nuclear norm of the matrix, which is defined as 
$$\|A\|_\star=\operatorname{trace}(\sqrt{A^\star A})=\lambda_1+\lambda_2+\cdots+\lambda_r
$$
where $A^\star$ is the complex conjugate transpose of $A$, $\sqrt{A^\star A}$ is the unique nonnegative definite Hermitian matrix whose square is $A^\star A$ and $\lambda_1,\lambda_2,\dots,\lambda_r$ are the singular values of $A$.
Although finding the nuclear norm of a higher order tensors is also NP-complete (see~\cite{FL}),  it is often easier than determining its rank. In \cite{Derksen} some examples of tensors are given for which the nuclear norm and the optimal decomposition can be computed, but where the rank of the tensors are unknown.

Let $\|\cdot\|_X=\|\cdot\|_\star$ be the nuclear norm. The dual norm, $\|\cdot\|_Y$, is equal to the spectral norm:
\begin{definition}
The spectral norm of a tensor $T$ is defined by
$$
\|T\|_\sigma=\max\{|\langle T,v\rangle|\mid \mbox{$v$ is a simple tensor with $\|v\|_2=1$}\}.
$$
\end{definition}
Finding the spectral norm of a higher-order tensor is also an NP-complete problem (see~\cite{HL}). 

From now on we consider the case $\F=\C$, where $V$ is the tensor product (over $\C$) of several finite dimensional  Hilbert spaces. We have a positive definite Hermitian inner product $\langle\cdot,\cdot\rangle$ on $V$. A real inner product is given by $\langle \cdot,\cdot\rangle=\Re \langle\cdot,\cdot\rangle_\C$.
\subsection{Examples of unitangent tensors}
The following examples come from~\cite{Derksen}:
\begin{example}
The space $\C^{p\times q}$ of complex $p\times q$ matrices has the usual basis $e_{i,j}$ where $1\leq i\leq p$ and $1\leq j\leq q$. Define $V=\C^{p\times q}\otimes \C^{q\times r}\otimes \C^{r\times p}$ and define the tensor
\begin{equation}\label{eq:Tpqr}
T_{p,q,r}=\sum_{i=1}^p\sum_{j=1}^q\sum_{k=1}^r e_{i,j}\otimes e_{j,k}\otimes e_{k,i}.
\end{equation}
This tensor is related to matrix multiplication. It is known that if $\rank(T_{p,q,r})\leq d$, then two $n\times n$ matrices can be multiplied using $O(n^{3\log(d)/\log(pqr)})$ arithmetic operations in $\C$. For most $p,q,r$, the rank of $T_{p,q,r}$ is unknown. For example, the best known lower bound for $\rank(T_{3,3,3})$ is 19, and follows from \cite{Blaser}.
The best known upper bound is 23 and comes from \cite{Laderman}.
It was shown in \cite{Derksen} that $\|T_{p,q,r}\|_\sigma=1$ and  $\|T_{p,q,r}\|_\star=pqr$. 
Now (\ref{eq:Tpqr}) is a convex decomposition, because
$$
pqr=\|T_{p,q,r}\|_{\star}=\sum_{i=1}^p\sum_{i=1}^q\sum_{i=1}^r\|e_{i,j}\otimes e_{j,k}\otimes e_{k,i}\|_2.
$$
Also, we have
$$
\|T_{p,q,r}\|_2^2=pqr=\|T_{p,q,r}\|_\star\|T_{p,q,r}\|_\sigma,
$$
so $T_{p,q,r}$ is unitangent. This means that 
$$T_{p,q,r}=\Big(\lambda T_{p,q,r}\Big)+\Big((1-\lambda)T_{p,q,r}\Big)$$
is an $X2$-decomposition, a $2Y$-decomposition and an $XY$-decomposition 
(as well as $2X$, $Y2$ and $YX$) if $0\leq \lambda\leq 1$.
If we take $\lambda=\|S\|_\star/\|T_{p,q,r}\|_\star=\|S\|_\star/(pqr)$, then 
we have $\|S\|_\star\leq \|\lambda T_{p,q,r}\|_\star$. It follows that
we get
$$
\|T_{p,q,r}-S\|_2\geq \|(1-\lambda)T_{p,q,r}\|_2=\left(1-\frac{\|S\|_\star}{pqr}\right)\sqrt{pqr}=\sqrt{pqr}-\frac{\|S\|_\star}{\sqrt{pqr}}.
$$
Similarly, we get inequalities
$$
\|T_{p,q,r}-S\|_2\geq \sqrt{pqr}-\sqrt{pqr}\|S\|_\sigma
$$
and
$$
\|T_{p,q,r}-S\|_\sigma \geq 1-\frac{\|S\|_\star}{pqr}.
$$
\end{example}

\begin{example}
Let $\Sigma_n$ be the set of permutation of $\{1,2,\dots,n\}$, and for a permutation $\tau$ denote its sign by $\sgn(\tau)$. The determinant tensor is defined by
$$
{\textstyle \det_n}=\sum_{\tau\in \Sigma_n}\sgn(\tau) e_{\tau(1)}\otimes e_{\tau(2)}\otimes \cdots\otimes e_{\tau(n)}\in \C^n\otimes \C^n\otimes \cdots\otimes \C^n.
$$
It was shown in \cite{Derksen} that $\|\det_n\|_\sigma=1$, $\|\det_n\|_\star=n!$ and $\|\det_n\|_2=\sqrt{n!}$.
In particular, $\det_n$ is unitangent. 
Let
$$
\perm_n=\sum_{\tau\in \Sigma_n}e_{\tau(1)}\otimes e_{\tau(2)}\otimes \cdots\otimes e_{\tau(n)}\in \C^n\otimes \C^n\otimes \cdots\otimes \C^n.
$$
be the permanent tensor. In \cite{Derksen} it was calculated that $\|\perm_n\|_\sigma=n!/n^{n/2}$, $\|\perm_n\|_\star=n^{n/2}$
and $\|\perm_n\|_2=\sqrt{n!}$.  The tensor $\perm_n$ is also unitangent.
\end{example}

\subsection{The diagonal SVD and the slope decomposition}
Following \cite{Derksen}, we make the following definitions.
\begin{definition}
Suppose that $v_1,v_2,\dots,v_r$ are simple tensors with $\|v_i\|_2=1$ for all $i$. For a real number $t\geq 1$ we say that $v_1,\dots,v_r$ are $t$-orthogonal if
$$
\sum_{j=1}^r |\langle v_i,w\rangle_\C|^{2/t}\leq 1
$$
for every simple tensors $w$ with $\|w\|_2=1$. 
\end{definition}
Note that $t$-orthogonality implies orthogonality because we can take $w=v_j$ so that
$$
\sum_{j=1}^r|\langle v_i, w\rangle_\C|^{2/t}=1+\sum_{j\neq i} |\langle v_i,v_j\rangle_\C|^{2/t}\leq 1
$$
implies that $v_j$ is orthogonal to all $v_i$ with $i\neq j$. By Pythagoras' theorem, orthogonality is equivalent to $1$-orthogonality.
\begin{definition}
The expression
\begin{equation}\label{eq:TDSVD}
T=\lambda_1v_1+\cdots+\lambda_r v_r
\end{equation}
is called a Diagonal Singular Value Decomposition (DSVD) if 
$v_1,\dots,v_r$ are $2$-orthogonal  simple tensors of length $1$ and $\lambda_1\geq \cdots\geq \lambda_r>0$.
\end{definition}
If (\ref{eq:TDSVD}) is a DSVD, then we have
\begin{eqnarray*}
\|T\|_\star &=& \lambda_1+\lambda_2+\cdots+\lambda_r\\
\|T\|_\sigma &=&\max\{\lambda_1,\lambda_2,\dots,\lambda_r\}\\
\|T\|_2 &=& \sqrt{\lambda_1^2+\lambda_2^2+\cdots+\lambda_r^2}.
\end{eqnarray*}

\begin{theorem}
Suppose that a tensor $T$ has a DSVD with singular values $\lambda_1>\lambda_2>\dots>\lambda_r>0$ and multiplicities $m_1,m_2,\dots,m_r$ respectively.
Then we can write
$$
T=\lambda_1 w_1+\lambda_2w_2+\cdots+\lambda_rw_r
$$
such that
$$
w_i=v_{i,1}+v_{i,2}+\cdots+v_{i,m_i}
$$
for all $i$ and
$$
v_{1,1},\dots,v_{1,m_1},v_{2,1},\dots,v_{2,m_2},\dots,v_{r,1},\dots,v_{r,m_r}
$$
is a sequence of $2$-orthogonal simple unit tensors.  Then the slope decomposition of $T$ is given by
$$
T=u_1+u_2+\cdots+u_r
$$
where 
$$u_i=(\lambda_i-\lambda_{i+1})(w_1+w_2+\cdots+w_i).
$$
\end{theorem}
\begin{proof}
We have
\begin{eqnarray*}
\|u_i\|_\sigma &=& (\lambda_i-\lambda_{i+1})\|w_1+\cdots+w_i\|_\sigma=(\lambda_i-\lambda_{i+1})\\
\|u_i\|_\star & = & (\lambda_i-\lambda_{i+1})\|w_1+\cdots+w_i\|_\star=(\lambda_i-\lambda_{i+1})(m_1+m_2+\cdots+m_i)\\
\mu{\star\sigma}(u_i) & = & \frac{\|u_i\|_\sigma}{\|u_i\|_\star}=(m_1+m_2+\cdots+m_i)^{-1}
\end{eqnarray*}
so we have
$$
\mu_{\star\sigma}(u_1)>\mu_{\star\sigma}(u_2)>\cdots>\mu_{\star\sigma}(u_s)>0.
$$
For $i<j$ we have
\begin{multline*}
\langle u_i,u_j\rangle_\C=(\lambda_i-\lambda_{i+1})(\lambda_j-\lambda_{j+1})\langle w_1+\cdots+w_i,w_1+\cdots+w_j\rangle_\C=\\
(\lambda_i-\lambda_{i+1})(\lambda_j-\lambda_{j+1})(m_1+m_2+\cdots+m_i)=\|u_i\|_\star\|u_j\|_\sigma.
\end{multline*}
So we also have
$$
\langle u_i,u_j\rangle=\Re\langle u_i,u_j\rangle_\C=\|u_i\|_\star\|u_j\|_\sigma.
$$
This proves that $T=u_1+u_2+\cdots+u_s$ is the slope decomposition.

\end{proof}
It was shown in \cite{Derksen} that the tensor $T_{p,q,r}$ has a diagonal singular value decomposition, but $\perm_n$ and $\det_n$
do not for $n\geq 3$.
\subsection{Group algebra tensors}
Suppose that $G$ is a finite group of order $n$. Suppose that there are $m_d$ irreducible representations of dimension $d$.
Then we have $\sum_d m_dd^2=n$. Let $e_g$, $g\in G$ be an orthonormal basis of $\C^n$ and consider the tensor
$$
T_G=\sum_{g\in G}\sum_{h\in G} e_g\otimes e_h\otimes e_{h^{-1}g^{-1}}.
$$
which is related to the multiplication in the group algebra. Then $T_G$ has singular value $\sqrt{n/d}$ with multiplicity $m_dd^3$ for all $d$.
We have
$$
f_{*\sigma}(\lambda)=h_{*\sigma}(\lambda)=\sum_{d}m_dd^3\max\left\{\lambda-\sqrt{\frac{n}{d}},0\right\}.
$$
\centerline{\includegraphics[width=4in]{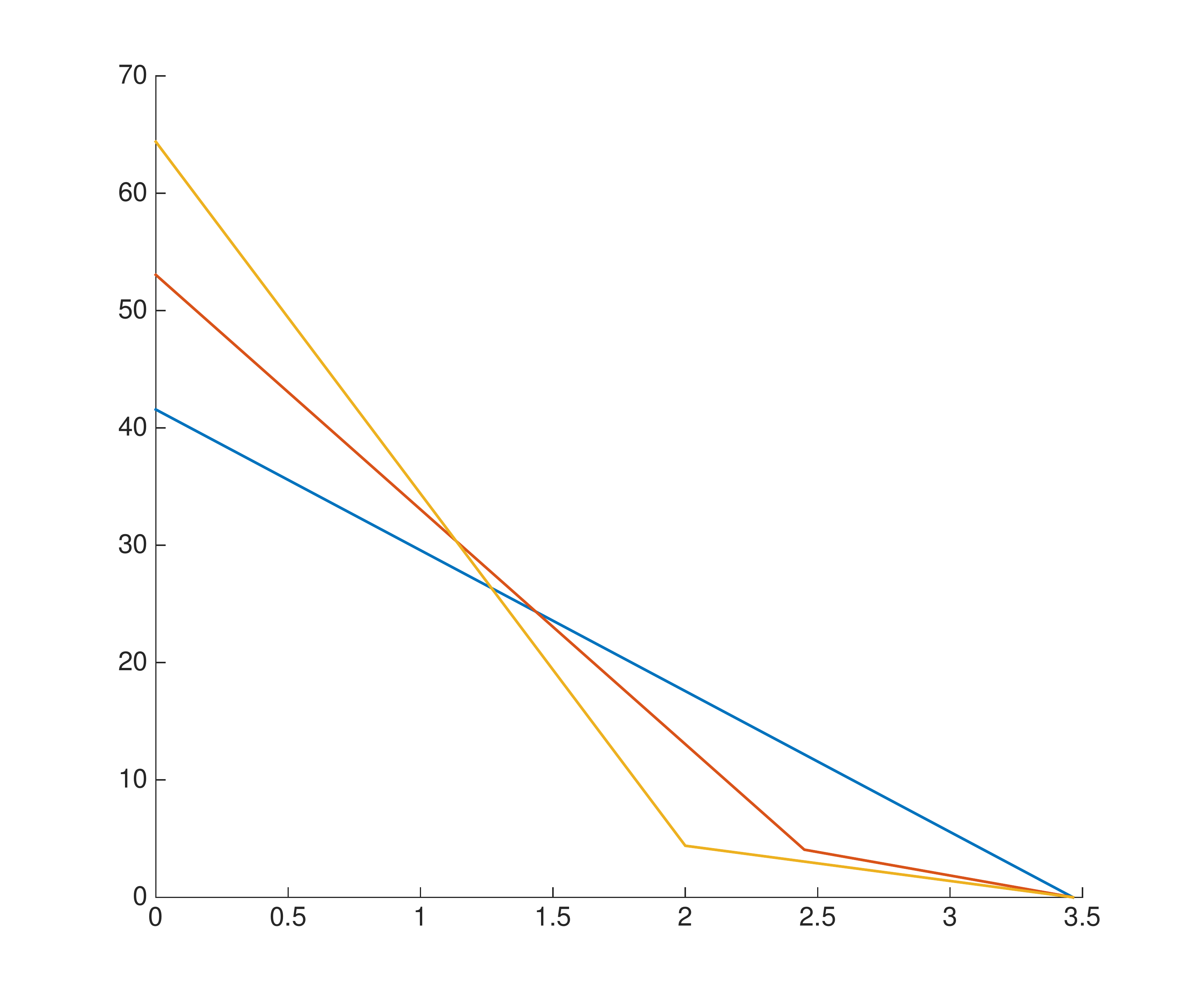}}
In the figure we drew the Pareto sub-frontier of $T_G$ for all groups $G$ of order $G$. The blue graph represents the abelian groups
$\Z/12$ and $\Z/6\times \Z/2$ with only $1$-dimensional representations, the red graph represents the dihedral group $D_{6}$ and
and the semi-direct product $\Z/4\ltimes \Z/3$ with representations of dimension $1,1,1,1,2,2$, and the yellow graph represents
the alternating group ${\bf A}_4$ with representations of dimension $1,1,2,3,3$. 

\subsection{symmetric tensors in $\R^{2\times 2\times 2}$}
For the remainder of the section, let us consider the tensor product space $\R^2\otimes \R^2\otimes \R^2$.
In particular, for $t\in \R$, we will study the symmetric tensor
$$
f_t=e_2\otimes e_1\otimes e_1+e_1\otimes e_2\otimes e_1+e_1\otimes e_1\otimes e_2+t e_2\otimes e_2\otimes e_2.
$$
\begin{proposition}\
\begin{enumerate}
\item We have
$$
\|f_t\|_\sigma=\begin{cases}
|t| & \mbox{if $t\geq 2$ or $t\leq -1$;}\\
\frac{2}{\sqrt{3-t}} & \mbox{if $-1\leq t\leq 2$.} \end{cases}
$$
\item We have
$$
\|f_t\|_\star=\begin{cases}
3-t & \mbox{if $t\leq \frac{1}{3}$};\\
\frac{(1+t)^{3/2}}{\sqrt{t}} & \mbox{if $t\geq \frac{1}{3}$}.\end{cases}
$$
\end{enumerate}
\end{proposition}
\begin{proof}\

(1) By the definition of the spectral norm, we have
$$
\|f_t\|_\sigma=\max\{ \langle f_t,v_1\otimes v_2\otimes v_3\rangle\mid \|v_1\|=\|v_2\|=\|v_3\|=1\}.
$$
By Banach's Theorem (see~\cite{Banach,Friedland}), we may take $v_1=v_2=v_3=v$. If we write $v=xe_1+ye_2$ with $x^2+y^2=1$,
then we have
\begin{multline*}
\|f_t\|_\sigma=\max_{x^2+y^2=1} \langle f_t,(xe_1+ye_2)\otimes (xe_1+ye_2)\otimes (xe_1+ye_2)\rangle\}=\\=
\max_{x^2+y^2=1} 3x^2y+ty^3=\max_{|y|\leq 1} 3y+(t-3)y^3.
\end{multline*}
Let $g_t(y)=3y+(t-3)y^3$ for $y\in [-1,1]$. We get $g_t'(y)=3+3(t-3)y^2$. If $t>2$ then $g_t'(y)$ has no roots in $[-1,1]$
and
$$
\|f_t\|_\sigma=\max\{g(1),g(-1)\}=t.
$$
If $t<2$, then the roots of $g_t'(y)$ are 
$$y=\pm \frac{1}{\sqrt{3-t}}.
$$
We have
\begin{multline*}
\|f_t\|_\sigma=\max\Big\{g(1),g(-1),g\Big(\frac{1}{\sqrt{3-t}}\Big),g\Big(-\frac{1}{\sqrt{3-t}}\Big)\Big\}=\\=\max\Big\{|t|,\frac{2}{\sqrt{3-t}}\Big\}=\begin{cases}
|t| & \mbox{if $t\leq -1$;}\\
\frac{2}{\sqrt{3-t}} & \mbox{if $-1\leq t\leq 2$.} \end{cases}
\end{multline*}

(2) Note that
$$
f_t=\frac{1}{2\sqrt{t}}\Big(e_1+\sqrt{t}e_2)\otimes (e_1+\sqrt{t}e_2)\otimes (e_1+\sqrt{t}e_2)-(e_1-\sqrt{t}e_2)\otimes (e_1-\sqrt{t}e_2)\otimes (e_1-\sqrt{t}e_2)\Big).
$$
This implies that
$$
\|f_t\|_\star\leq \frac{(1+t)^{3/2}}{\sqrt{t}}.
$$
If $t\geq \frac{1}{3}$,  then let $s=2-t^{-1}\in [-1,2]$. 
We get
$$
2+2t=3+st=\langle f_t,f_s\rangle\leq \|f_t\|_\star\|f_s\|_\sigma=\|f_t\|_\star\frac{2}{\sqrt{3-s}}=\|f_t\|_\star \frac{2\sqrt{t}}{\sqrt{1+t}}.
$$
So it follows that
$$
\|f_t\|_\star\geq \frac{(1+t)^{3/2}}{t}
$$
so we must have equality.

For $t=\frac{1}{3}$ we get $\|f_{1/3}\|_\star= \frac{8}{3}$. For $t<\frac{1}{3}$ we get
$$
\textstyle\|f_t\|_\star\leq \|f_{1/3}\|_\star+\|f_t-f_{1/3}\|_\star=\frac{8}{3}+\|(t-\frac{1}{3})e_2\otimes e_2\otimes e_2\|_\star=\frac{8}{3}+\frac{1}{3}-t=3-t.
$$
We have
$$
3-t=\langle f_t,f_{-1}\rangle\leq \|f_t\|_\star\|f_{-1}\|_\sigma=\|f_t\|_\star,
$$
so this shows that $\|f_t\|_\star=3-t$.

\end{proof}

\begin{corollary}
For $0\leq t\leq \frac{1}{3}$
$$
f_0=\Big(\frac{1}{t+1}\Big) f_t+\Big(\frac{t}{1+t}\Big)f_{-1}
$$
is a $\star \sigma$-decomposition, and for $\frac{1}{3}\leq t\leq \frac{1}{2}$
$$
f_0=\left(\frac{1-2t}{(1-t)^2}\right)f_t+\Big(\frac{t^2}{(1-t)^2}\Big)f_{2-t^{-1}}
$$
is a $\star\sigma$-decomposition. A parametrization of the subpareto curve is given by
$$
\left(\frac{3-t}{t+1},\frac{t}{1+t}\right)
$$
if $0\leq t\leq \frac{1}{3}$ and
$$
\left(\frac{(1-2t)(1+t)^{3/2}}{(1-t)^2\sqrt{t}},\frac{2t^{5/2}}{(1-t)^2\sqrt{1+t}}\right)
$$
if $\frac{1}{3}\leq t\leq \frac{1}{2}$.

\end{corollary}
Let us plot the Pareto sub-frontier:

\centerline{\includegraphics[width=3in]{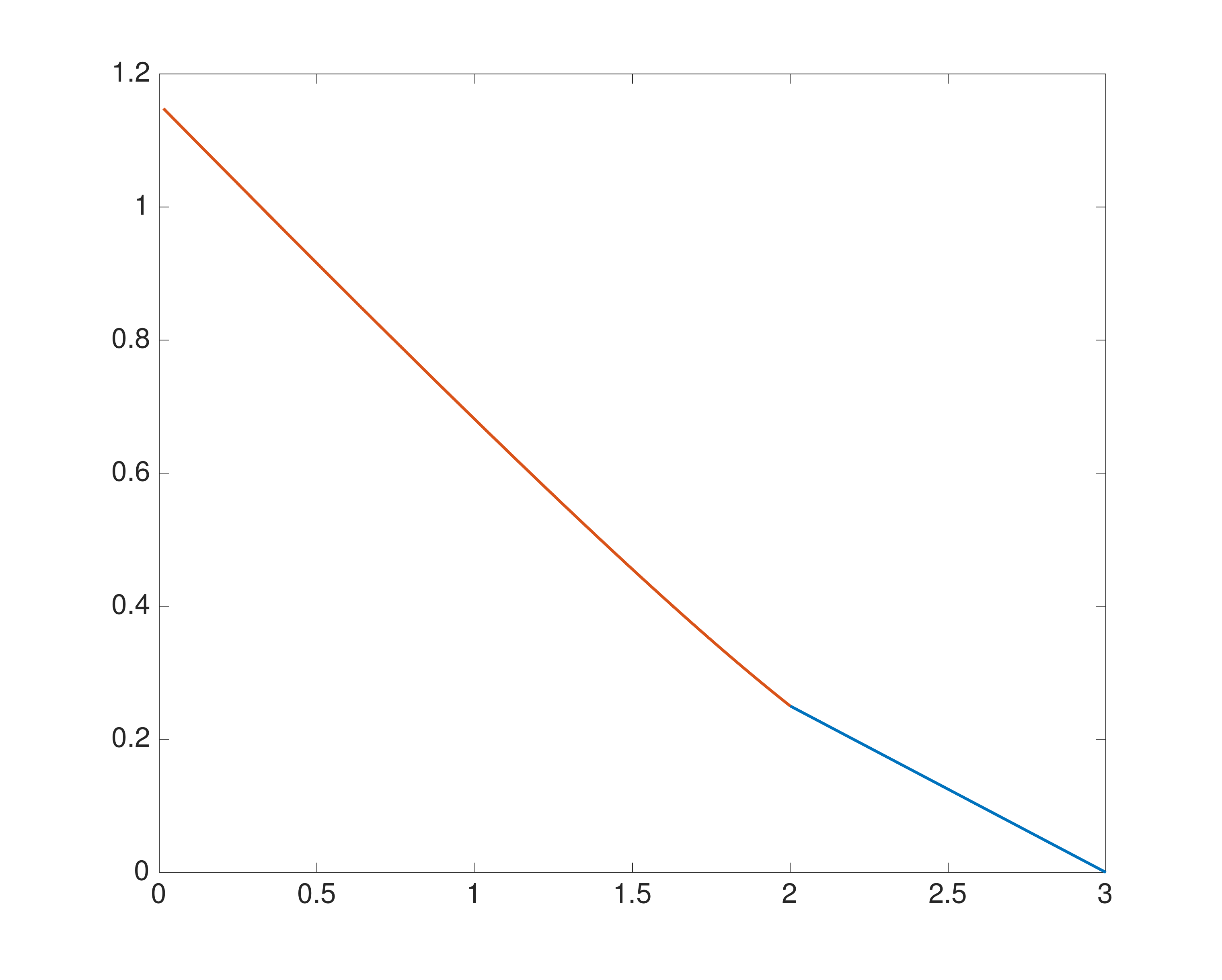}}

The blue part of the graph is linear, but the red part is non-linear. The graph is not piecewise linear, so $f_0$ does not have a slope decomposition. This shows that the nuclear norm and the spectral norm on $\R^2\otimes \R^2\otimes \R^2$ are not tight. 

Below we have plotted the singular value region. The singular value $\frac{2}{\sqrt{3}}$ appears with multiplicity $\frac{27}{13}$, the singular value $\frac{1}{4}$ appears with multiplicity $\frac{4}{3}$. The singular values between $\frac{1}{4}$ and $\frac{2}{\sqrt{3}}$ appear with infinitesemal multiplicities.
The height of the region is the spectral norm $\|f_0\|_\sigma=\frac{2}{\sqrt{3}}$, the area of the region is the nuclear norm $\|f_0\|_\star=3$, and if we integrate $2y^2$ over the region we get  the square of the euclidean norm which is $\|f_0\|_2^2=3$.

\centerline{\includegraphics[width=4in]{SVregion2.eps}}

   \end{document}